\documentclass[11pt]{article}
\usepackage[top=2cm,bottom=2.5cm,right=2.5cm,left=2.5cm]{geometry}
\usepackage[english]{babel}
\usepackage[T1]{fontenc}
\usepackage{times,soul}
\usepackage{mathtools, bm}
\usepackage{amssymb, bm}
\usepackage{graphicx}
\usepackage{mathrsfs}
\usepackage{stmaryrd}
\usepackage{amsthm}
\usepackage{listings}
\usepackage{pict2e}
\usepackage{pgf, tikz}
\usepackage{dsfont}
\usepackage{algorithm2e}
\usepackage{algorithmic}
\usepackage{multicol}
\usepackage{hyperref}
\usepackage{subcaption}
\usepackage{float} 
\usepackage{enumitem}

\usepackage[export]{adjustbox}

\newcommand{\Af}{\mathbb{A}^{\text{fast}}}
\newcommand{\Afodd}{\mathbb{A}^{\text{fast,1}}}
\newcommand{\Afeven}{\mathbb{A}^{\text{fast,2}}}

\newcommand{\As}{\mathbb{A}^{\text{slow}}}

\newcommand{\good}{\mathfrak{G}_N}
\newcommand{\cE}{\mathcal{E}}
\newcommand{\bbE}{\mathbb{E}}
\newcommand{\bbP}{\mathbb{P}}
\newcommand{\bbR}{\mathbb{R}}
\newcommand{\bbT}{\mathbb{T}}

\allowdisplaybreaks

\thicklines

{
      \theoremstyle{plain}
      
  }

\newtheorem{theorem}{Theorem}
\newtheorem{proposition}{Proposition}
\newtheorem{corollary}{Corollary}
\newtheorem{lemma}{Lemma}
\newtheorem{remark}{Remark}
\newtheorem{definition}{Definition}

\numberwithin{equation}{section} 
\numberwithin{lemma}{section} 
\numberwithin{remark}{section} 
\numberwithin{example}{section}
\numberwithin{corollary}{section}
\numberwithin{proposition}{section}

\date{}
\author{Thierry Bodineau\footnote{CNRS, I.H.E.S., 35 Route de Chartres, 91440 Bures-sur-Yvette, France. Email: bodineau[AT]ihes.fr},\quad Pierre Le Bris\footnote{SAMOVAR, Télécom SudParis, Institut Polytechnique de Paris, 91120 Palaiseau, France. Email: pierre.le$\_$bris[AT]telecom-sudparis.eu}}
\title{Linear Landau equation as a limit of a tagged particle in mean field interaction with a free gas}

\begin{document}

\maketitle

\begin{abstract}

We consider a tagged particle in mean field interaction with a free gas of density $N$ at equilibrium. In dimensions $d\geq4$, we prove the convergence of its trajectory, as $N$ goes to infinity, to the one of a diffusion process associated  with the linear Landau equation. The proof of the convergence of the martingale problem relies on two key ingredients:  long time stability results  of the microscopic dynamics, and controls on the probability of particle recollisions. 
\end{abstract}

\tableofcontents

%
%
%
%

\section{Introduction}

%
%
%
%

\subsection{Motivation}

Since the botanist Brown described the irregular movement of particles suspended on a liquid, many works have been dedicated to understanding the underlying physical phenomenon. This eventually lead to the construction of the Brownian motion as a mathematical object (we refer to \cite{Dup06} for a historical review on the subject), raising in particular the question of the emergence of a  stochastic process from the motion of atoms evolving at a microscopic level according to Newton's laws. We refer to \cite[Chapter 8]{Spohn} for a detailed review of different deterministic dynamics from which the behavior of a tagged particle has been studied. 
Several types of microscopic models, based on specific scalings, have been considered, and each case requires dedicated mathematical tools for the analysis of the limiting stochastic behavior. Mathematical results were mainly obtained for two distinct perturbative regimes: the low density limit and the weak interactions.

In the low density limit, the dynamics are tuned such that the collisions between two particles are rare, but the scattering is strong. The evolution of a tagged particle with a very small mass with respect to the background particles can be represented by a Lorentz gas where the background particles are fixed.
If the scatterers are chosen randomly with a low density $\varepsilon$,
the convergence of the tagged particle to a Brownian motion has been established in  \cite{basile2015derivation, LT20} on time intervals diverging as a power of $\varepsilon$.
For a fully interacting hard sphere gas, the convergence to a Brownian motion has been analysed in 
\cite{BGS16,fougeres2024derivation,Ayi} in the Boltzmann-Grad limit, see also \cite{MS24} for a fractional diffusion. For a large and massive tagged particle in interaction with a free gas, the limiting dynamics is given by an Ornstein-Uhlenbeck process \cite{DGL81,DLLS13,KL10,BGS18}.

A different perspective is to consider a dense gas in a weak coupling limit, in which case the interactions are frequent but the scattering is very small. 
There are many ways to tune the density with respect to the interactions, depending on the underlying physical model, and
we refer to \cite{NVW21} for a very complete account of the different limiting behaviors expected depending on the nature of the interaction. Note that there is no mathematical result in this regime for a genuine interacting gas.
For a Lorentz gas (i.e. with fixed scatterers), the convergence to a stochastic process characterized by the linear Landau equation has been derived in \cite{KP79, KP80,DGL87,DR01,PV03}.
A Lorentz gas with long range interactions may interpolate between the linear Boltzmann or Landau equations depending on the decay  of the interaction potential, we refer to  \cite{NSV18} for an indepth discussion.
The Landau equation arises naturally also for systems of interacting particles, e.g. it is obtained in \cite{VW18}  from the truncated BBGKY hierarchy and in \cite{LeBihan24}  the linear Landau equation is derived from a particle system, in the Boltzmann-Grad limit, interacting through a potential approximating  an inverse power law $|x|^{-s}$.

\medskip


\begin{figure}
    	\centering
\includegraphics[width=0.43\linewidth, trim={2cm 2cm 2cm 0},clip]{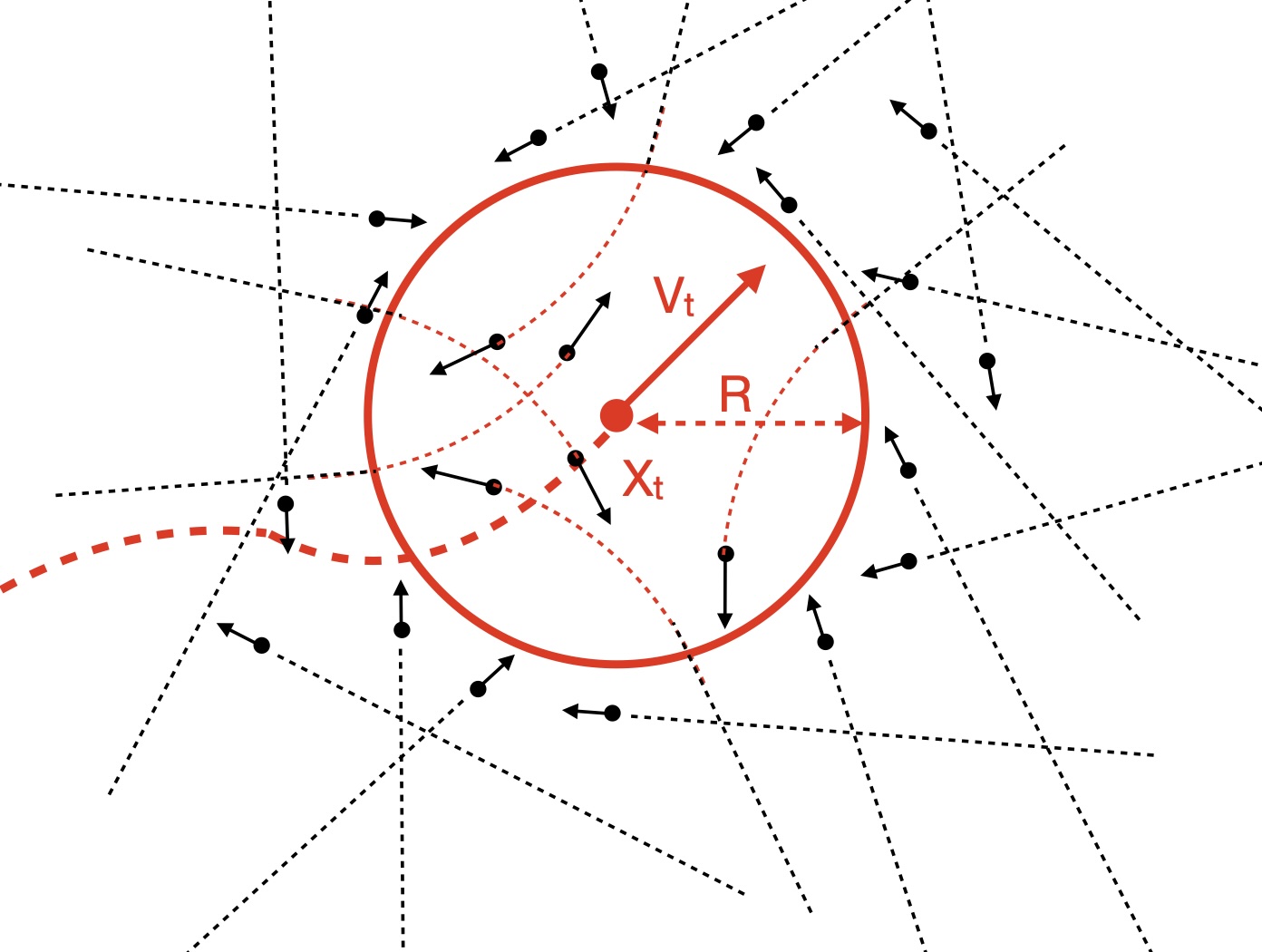}
	\caption{{\small Illustration of the process $(Z_t)_t=(X_t,V_t)_t$ defined in \eqref{eq:def_tagged}. The background particles move in a straight line, until they meet the tagged particle. The latter interacts with the background inside an interaction ball of radius $R$.}}
	\label{fig:illustration_process}
\end{figure}

In this paper, we study a different weak regime with a mean field interaction between 
 the tagged particle and a free gas at equilibrium. We prove that in the mean field limit, the tagged particle follows a stochastic process associated with the linear Landau equation (see Theorem \ref{thm:main}). 
More precisely, we consider a free gas of particles evolving in straight line in a periodic domain $\bbT = [-N^\eta, N^\eta]^d$ for some $\eta \geq 2$. The initial positions $(x_i)_{i\in\mathbb{N}}$ are distributed according to a homogeneous Poisson point process in $\bbT$ of intensity $N\in\mathbb{N}$. Roughly speaking, we thus have $N$ points per unit volume. The initial velocities $(v_i)_{i\in\mathbb{N}}$ are given by  independent (mutually and of the $(x_i)_{i\in\mathbb{N}}$) random variables distributed according to a standard $d$-dimensional normal distribution. This defines an environment of \textit{background particles} and inside this environment, we introduce a \textit{tagged particle}, whose position and velocity at time $t$ are respectively denoted $X_t$ and $V_t$, interacting with the background particles according to a force derived from a potential $\Phi$. More specifically, let $\Phi$ be a compactly supported function in $\mathcal{C}^3_c(\mathbb{R}^d,\mathbb{R})$ such that for some $R>0$
\begin{equation}
\label{assumption Phi}
\text{Support}(\Phi) \subset \mathcal{B}(0,R) 
\  \text{and
$\Phi$ is even, i.e. for all $x\in \mathbb{R}^d$, we have $\Phi(x)=\Phi(-x)$.}
\end{equation}
We consider the process defined for $t\in[0,T]$ by
\begin{equation}
\label{eq:def_tagged}
\frac{d}{dt}X_t=V_t,\qquad \frac{d}{dt}V_t=-\frac{1}{N}\sum_{i\in\mathbb{N}} \nabla\Phi(X_t-x^i_t), \qquad\text{with initial conditions }X_0\text{ and } V_0,
\end{equation}
where for all the background particles, i.e. all $i\in\mathbb{N}$, we introduce also a feedback force (see Figure \ref{fig:illustration_process})
\begin{equation}
\label{eq:def_process_back}
x^i_0=x_i,\quad v^i_0=v_i,\quad \frac{d}{dt}x^i_t=v^i_t,\quad \text{ and }\quad \frac{d}{dt}v^i_t=\frac{1}{N}\nabla\Phi(X_t-x^i_t).
\end{equation}
Thus the background particles behave almost as a free gas up to this feedback force with the tagged particle. 
This force ensures that the dynamics is Hamiltonian, so that there exists an invariant distribution as the energy is conserved. For the sake of conciseness, we may denote $Z_t=(X_t,V_t)$ and $z^i_t=(x^i_t,v^i_t)$.

\medskip

The dynamics \eqref{eq:def_tagged}--\eqref{eq:def_process_back} is deterministic, and the only randomness comes from the initial conditions. 
Our goal is to prove (see Theorem \ref{thm:main} and Corollary \ref{lem:on_a_Landau}) that the tagged particle evolution can be described by a stochastic process
 on a timescale of order $N$ when the density $N$ of background particles increases to infinity.
A priori the mean field force \eqref{eq:def_tagged} acting on the tagged particle could be of order 1. However the background gas is at equilibrium and the force \eqref{eq:def_tagged} is the average over roughly $N$ particles which can be interpreted as almost independent random variables,  uniformly distributed in the support of $\Phi$.
As $\nabla \Phi$ is of 0 mean, we can argue from the central limit theorem (or more precisely from a Hoeffding argument) that
\begin{equation}
\label{eq: TCL force}
\frac{d}{dt}V_t=-\frac{1}{N}\sum_{i\in\mathbb{N}} \nabla\Phi(X_t-x^i_t) = O \left( \frac{1}{\sqrt{N}} \right).
\end{equation}
The tagged particle is therefore driven by the fluctuations of the gas  creating small random forces varying with time. 
These fluctuations are transported by the free gas and the short time correlations are responsible for the specific form of the diffusion coefficient associated with the tagged particle. 
After rescaling time by a factor $N$, the forces coming from the deterministic dynamics will act as  
a random process.
As often in the study of deterministic models, the main difficulty is to show that the force \eqref{eq:def_tagged} decorrelates in time so that it can be approximated by a white noise on some appropriate scale.
Compared with the Lorentz gas model \cite{ KP80,DGL87}, some background particles may  remain close to the tagged particle for a very long time when their relative velocities are small.
Thus the decoupling in time of the force \eqref{eq: TCL force} could be much slower than expected and this requires a specific analysis. 
This is analogous to a resonance phenomenon and it is one of the main difficulty to control the limiting behavior.
Similar issues would arise in  the much more challenging case of a fully interacting mean field gas, which we are going to describe next, and \eqref{eq:def_tagged} can be seen as a toy model for this more general problem.

\medskip

Consider a full mean field system of $N$ interacting particles whose dynamics are given by Newton's law
\begin{equation}
\label{eq:ps_total}
\frac{d}{dt}x^i_t=v^i_t,\qquad \frac{d}{dt}v^i_t=-\frac{1}{N}\sum_{j\neq i}\nabla \Phi(x^i_t-x^j_t),\qquad i\in\{1,...,N\},
\end{equation} 
where $(x^i_t,v^i_t)$ denotes the position and velocity of the $i$-th particle at time $t$. 
For a gas close to equilibrium, the macroscopic density is almost constant and therefore at the main order the force on each particle scales as $1/\sqrt{N}$ similarly to \eqref{eq: TCL force}.
This small random force varies in time and has a contribution on timescales of order $N$, leading to the \textit{Lenard-Balescu equation} 
which is a correction to the mean field equation
(see \cite{DS21, DW23} and discussion therein).
The rigorous derivation of the Lenard-Balescu equation from a particle system is a notoriously difficult problem for many reasons, among them the fact that local well-posedness is only known for smooth interactions (and global well-posedness only near equilibrium, see \cite{DW23}) or that one has to prove estimates on  correlations that are valid up to time $t\simeq N$. Rigorous results of such a derivation are obtained in \cite{Due21, DS21}, though for shorter timescales not reaching $t\simeq N$. See also the sequence of works \cite{NSV18, NVW22, NVW21, NWL19} which provide some first steps towards the justification of the Lenard-Balescu equation, as well as \cite{LeBihan25} for a study of the quantum Lenard-Balescu equation.

It is conjectured that, at equilibrium, a tagged particle in the mean field system \eqref{eq:ps_total} will follow a stochastic process associated with the linear Lenard-Balescu equation \cite{NVW22} after a time rescaling by $N$. 
We refer to \cite{DS21} for a short time analysis of this problem and to \cite{DLB25} for a derivation of the linear Lenard-Balescu equation on the timescale $t\simeq N$
under some assumptions. The authors work in \cite{DLB25} at the level of the BBGKY hierarchy, under two main simplifications: the hierarchy is truncated at an arbitrary level $m_0$ (i.e. assuming that $m$-particle correlations, for 
$m>m_0$, are zero) and the acceleration is driven by a Brownian motion of the correct order (i.e $dv^1_t=-\frac{1}{N}\sum \nabla\Phi(x^1_t-x^i_t)dt+\sqrt{2\kappa/N}dB_t$, which provides hypoelliptic estimates). Their analysis then relies on a careful Dyson expansion of this simplified BBGKY hierarchy to control particle interactions.
Note that the simplified dynamics \eqref{eq:def_tagged}-\eqref{eq:def_process_back} follows the linear Landau equation and not the linear Lenard-Balescu equation as the feedback of the background particles is simpler in our case.

%
%
%
%

\subsection{Formalism and main result}

%
%
%
%

\subsubsection{Initial measure: the grand canonical ensemble}

The initial data will be sampled according to the measure defined below. 
Denote $\gamma$ the density of the $d$ dimensional standard normal distribution, i.e
\begin{align}
\label{eq: gaussian definition}
\gamma(v):=(2\pi)^{-\frac{d}{2}}e^{-\frac{|v|^2}{2}}.
\end{align}
Let $\mathcal{C}^1_{bb}(\mathbb{R}^d,\mathbb{R})$ be the set of bounded continuous functions with bounded continuous derivatives. 
Let $g_0$ be a function in  $\mathcal{C}^1_{bb}(\mathbb{R}^d,\mathbb{R})$ such that 
$\int dV g_0(V) \gamma(V) =1$.
Define, for a given $M\in\mathbb{N}$, 
\begin{align*}
\rho_{M,g_0}\left(Z,z_{1:M}\right):=e^{-\frac{1}{N}\sum_{i=1}^M\Phi(X-x_i)} \; g_0(V) \gamma(V) \, \prod_{i=1}^M\gamma(v_i),
\end{align*} 
where we write $z_{1:M}=(z_1,..., z_M)$ (resp. $x_{1:M}=(x_1,..., x_M)$ and $v_{1:M}=(v_1,..., v_M)$), with $z_i=(x_i,v_i)$, and $Z=(X,V)$.
By convention, if $M=0$, the sum (resp. product) is empty, thus zero (resp. one), and we set $\rho_{0,g_0} \left(Z\right)=  g_0(V) \gamma(V).$ For $F$ a measurable function of $Z$ and of a point process $z_{1:\mathcal{N}}$, symmetric in  $z_{1:\mathcal{N}}$, we define the  \textit{grand canonical measure} as
\begin{align}
\mathbb{E}\left[F(Z,z_{1:\mathcal{N}})\right]=\frac{1}{\mathcal{Z}}\sum_{M=0}^\infty \frac{N^M}{M!} \int_{\mathbb{T}\times\mathbb{R}^d } dZ\int_{\mathbb{T}^M\times \mathbb{R}^{Md}} dz_{1:M}\rho_{M, g_0}\left(Z,z_{1:M}\right)F\left(Z,z_{1:M}\right),
\label{eq:def_gce}
\end{align}
where the normalization constant $\mathcal{Z}$ can in fact be explicitly computed: denoting $|\mathbb{T}|$ the volume of the torus, we have $\mathcal{Z}
=|\mathbb{T}|\exp\left(N\left(\int_{\mathbb{T}}dxe^{-\frac{\Phi(x)}{N}}\right)\right).
$
Let us gather some useful properties of random variables distributed according to this probability measure.


\begin{lemma}\label{lem:N_Poisson}
We have
\begin{itemize}
\item Denoting $\mathcal{N}$ the number of background particles, $\mathcal{N}$ is a random variable distributed according to a Poisson distribution of mean $N\int_{\mathbb{T}}dxe^{-\frac{\Phi(x)}{N}}$.
\item The mean number of particles in $\mathbb{T}$ satisfies
$\frac{\mathbb{E}[\mathcal{N}]}{N|\mathbb{T}|}\xrightarrow[]{N\rightarrow\infty}1$
and for all $k\in\mathbb{N}$ there exists $C_k$ such that 
\begin{equation}\label{eq:borne_moment_N}
\mathbb{E}\left[\mathcal{N}^k\right]\leq C_k N^k |\mathbb{T}|^k.
\end{equation}
\end{itemize}
\end{lemma}
The proof of this lemma is standard as the measure \eqref{eq:def_gce} is essentially a Poisson measure.
Since $\mathcal{N}$ is almost surely finite and the interaction $\Phi$ is smooth compactly supported, Cauchy–Lipschitz theory for standard ordinary differential equations yields as a direct consequence existence and uniqueness for the process.

A key property is the stationarity of the grand canonical measure   \eqref{eq:def_gce}.


\begin{lemma}
For $g_0 =1$,
the grand canonical ensemble defined in \eqref{eq:def_gce} is stationary for the process \eqref{eq:def_tagged}-\eqref{eq:def_process_back} in the sense that for any measurable function $F$
\begin{align*}
\forall t\in\mathbb{R},\quad \mathbb{E}\left[F\left(Z_t, z^{1:\mathcal{N}}_t\right)\right]=\mathbb{E}\left[F\left(Z_0, z^{1:\mathcal{N}}_0\right)\right],
\end{align*}
or even, denoting $Z_{[s:t]}$ (resp. $z^{1:\mathcal{N}}_{[s:t]}$) the trajectory of the tagged (resp. background) particles between times $s$ and $t$, we have 
\begin{align*}
\forall s,t\in\mathbb{R},\quad \mathbb{E}\left[F\left(Z_{[s:s+t]}, z^{1:\mathcal{N}}_{[s:s+t]}\right)\right]=\mathbb{E}\left[F\left(Z_{[0,t]}, z^{1:\mathcal{N}}_{[0,t]}\right)\right].
\end{align*}
\end{lemma}
This lemma can be easily derived,  by conditioning on the total number of particles $\mathcal{N}$ and observing that the law $\rho_{\mathcal{N}}$ (up to normalization) is the Gibbs measure associated with the Hamiltonian dynamics \eqref{eq:def_tagged}-\eqref{eq:def_process_back}.

Note that the grand canonical ensemble \eqref{eq:def_gce} acts on functions that are symmetric in $z_{1:\mathcal{N}}$. To single out the effect of a given particle, we introduce below a non-symmetric version of the expectation. Assume that $F(Z,z_{1:\mathcal{N}})=\frac{1}{N}\sum_{i=1}^\mathcal{N}f(Z,z_{1:\mathcal{N}}, z_i)$ for some function $f$ which is symmetric in its second variable $z_{1:\mathcal{N}}$, in the sense that for any permutation $\sigma$ and any $Z,z, z_{1:\mathcal{N}}$ we have $f(Z,z_{1:\mathcal{N}}, z)=f(Z,z_{\sigma(1):\sigma(\mathcal{N})}, z)$. We define
\begin{align}
\mathbb{E}\left[F(Z,z_{1:\mathcal{N}})\right]=&\mathbb{E}\left[\frac{1}{N}\sum_{i=1}^\mathcal{N}f(Z,z_{1:\mathcal{N}}, z_i)\right]
\nonumber \\
=&\frac{1}{\mathcal{Z}}\sum_{M=0}^\infty \frac{N^M}{M!} \frac{1}{N}\sum_{i=1}^M\int_{\mathbb{T}\times\mathbb{R}^d } dZ\int_{\mathbb{T}^M\times \mathbb{R}^{Md}} dz_{1:M}\rho_{M, g_0} \left(Z,z_{1:M}\right)f(Z,z_{1:M}, z_i)
\nonumber \\
=&\frac{1}{\mathcal{Z}}\sum_{M=0}^\infty \frac{N^{M-1}}{(M-1)!} \int_{\mathbb{T}\times\mathbb{R}^d } dZ\int_{\mathbb{T}^M\times \mathbb{R}^{Md}} dz_{1:M}\rho_{M, g_0}\left(Z,z_{1:M}\right)f(Z,z_{1:M}, z_1) \nonumber \\
=&\mathbb{E}\left[\frac{\mathcal{N}}{N}f(Z,z_{1:\mathcal{N}}, z_1)\right]
=: \hat{\mathbb{E}}\left[f(Z,z_{1:\mathcal{N}}, z_1)\right].
\label{eq: definition E chapeau}
\end{align}
Finally, we may also denote for a set $\mathcal{X}$ both $\mathbb{E}_{\mathcal{X}}[f]=\mathbb{E}[f\mathds{1}_{\mathcal{X}}]$ and $\hat{\mathbb{E}}_{\mathcal{X}}[f]=\hat{\mathbb{E}}[f\mathds{1}_{\mathcal{X}}]$.

%
%
%
%
%

\subsubsection{Main result}


\begin{theorem}
\label{thm:main}
Let $d\geq 4$. Under the grand canonical ensemble  \eqref{eq:def_gce}, the family of processes $((V_{\tau N})_{\tau\in[0,1]})_N$ defined by \eqref{eq:def_tagged}-\eqref{eq:def_process_back} converges in law as $N\rightarrow\infty$ to a diffusion process $(\mathcal{V}_\tau )_{\tau \in[0,1]}$ initially distributed with the probability density $g_0 \gamma$ and satisfying
\begin{equation}
\label{eq:eds_limit}
d\mathcal{V}_\tau=2\Lambda(\mathcal{V}_\tau)d\tau+\sqrt{2}\Sigma(\mathcal{V}_\tau)dB_\tau,
\end{equation}
where $B$ is a Brownian motion in $\mathbb{R}^d$,  $g_0$ is in $\mathcal{C}^1_{bb}(\mathbb{R}^d,\mathbb{R})$ and $\gamma$ is the Gaussian distribution
\eqref{eq: gaussian definition}. Here $\Lambda$ is defined by
\begin{equation}\label{def:drift}
\Lambda(V)=-\int_{\mathcal{B}(0,R)}dx\int_{\mathbb{R}^d}dv\gamma(v+V)\text{Hess}\Phi(x)\int_{-\infty}^0\int_{-\infty}^s\nabla\Phi(x+uv)duds,
\end{equation}
and the matrix $\Sigma$ is the square root of $D$, a $\mathbb{R}^{d\times d}$ positive definite symmetric matrix given by
\begin{equation}\label{def:diffusion} 
D(V)=\int_{\mathcal{B}(0,R)}dx\int_{\mathbb{R}^d}dv\gamma(v+V)\nabla\Phi(x)\otimes\int_{-\infty}^0\nabla\Phi(x+sv)ds.
\end{equation}
\end{theorem}
As proved in Lemma~\ref{lem:prop_coeff} below, the coefficients $\Lambda$ and $D$ can be rewritten using Fourier variables, and satisfy both $\Lambda(V)=-D(V)V$ and $\Lambda(V)=\nabla\cdot D(V)$. 
As a direct consequence of Lemma~\ref{lem:prop_coeff}, we obtain the following lemma, which states that the process \eqref{eq:eds_limit} is indeed the one associated with the classical (linear) Landau equation.


\begin{corollary}
\label{lem:on_a_Landau}
The Fokker-Planck equation associated with the SDE \eqref{eq:eds_limit} is  the linear Landau equation, i.e. $f_t=\text{Law}(\mathcal{V}_t)$ is a weak solution of 
\begin{equation}
\label{eq:lin_landau}
\partial_t f_t=\nabla\cdot(D(V)(\nabla f_t+ V \, f_t)),
 \qquad \text{with} \quad f_0 = g_0 \gamma,
\end{equation}
where 
\begin{align*}
D(V)=\frac{\pi}{(2\pi)^{d}}\int_{\mathbb{R}^d} (k\otimes k)|\hat{\Phi}(k)|^2\int_{\mathbb{R}^d} \gamma(v+V)\delta_{k\cdot v}dvdk.
\end{align*}
\end{corollary}

%
%
%
%

\subsection{Sketch of the proof}

To help the reader navigate through the various technical lemmas, let us start with an informal discussion on the method of proof.
The overall strategy is to show that the sequence of continuous processes $((V_{\tau N})_{\tau\in[0,1]})_N$ is tight and that any limit point satisfies a martingale problem  which determines the limiting process. This guarantees the convergence in law of the processes. 
We stress that our analysis is based on a pathwise study and doesn't use the BBGKY hierarchy. 


\paragraph{The martingale problem.} Itô's formula applied to the SDE \eqref{eq:eds_limit} states that for any test function $f$ the quantity
\begin{align*}
\mathcal{M}^f_\tau:=f(\mathcal{V}_\tau)-f(\mathcal{V}_0)-\int_0^\tau\left(2\Lambda(\mathcal{V}_\sigma)\cdot \nabla f(\mathcal{V}_\sigma)+D(\mathcal{V}_\sigma):\text{Hess}f(\mathcal{V}_\sigma)\right)d\sigma,
\end{align*}
is a martingale. In particular, for any $f$ and $g$ test functions, we  have
\begin{align*}
\bbE\left[g(\mathcal{V}_0)(f(\mathcal{V}_\tau)-f(\mathcal{V}_0))\right]=\bbE\left[g(\mathcal{V}_0)\int_0^\tau\left(2\Lambda(\mathcal{V}_\sigma)\cdot \nabla f(\mathcal{V}_\sigma)+D(\mathcal{V}_\sigma):\text{Hess}f(\mathcal{V}_\sigma)\right)d\sigma\right].
\end{align*}
This property (almost) characterizes the solutions of the SDE \eqref{eq:eds_limit}. Our goal is therefore to show that the rescaled process $(V_{\cdot N})_N$  satisfies also this property  up to a small error:
\begin{equation}
\label{eq:informal_mart}
\mathbb{E}\left[g(V_0)\left(f(V_{\tau N})-f(V_0)-\int_0^\tau\left(2\nabla f(V_{tN})\cdot \Lambda(V_{tN})+\text{Hess}f(V_{tN}):D(V_{tN})\right)dt\right)\right]=O\left(\frac{\tau}{N}\right),
\end{equation}
 The rigorous proof can be found in Proposition~\ref{prop:martingale_formulation} in Section~\ref{sec:martingale_formulation}. We obtain that any limit of $((V_{\tau N})_\tau)_N$ is a solution to the martingale problem associated with $(\mathcal{V}_\tau)_\tau$.
Combined with some standard tightness estimates (see Section~\ref{sec:tight}), \eqref{eq:informal_mart} is enough to identify the limiting process and to conclude on the convergence.

Below, we sketch a formal derivation of \eqref{eq:informal_mart}.
The starting point is the deterministic evolution  \eqref{eq:def_tagged}
\begin{align}
\mathbb{E}\left[g(V_0)(f(V_{\tau N})-f(V_0))\right]=&-\mathbb{E}\left[g(V_0)\int_0^{\tau N}\nabla f(V_t)\frac{1}{N}\sum_{i}\nabla\Phi(X_t-x^i_t)dt\right]\nonumber\\
=&-\hat{\mathbb{E}}\left[g(V_0)\int_0^{\tau N}\nabla f(V_t)\nabla\Phi(X_t-x^1_t)dt\right],
\label{eq:sketch_init}
\end{align}
using notation \eqref{eq: definition E chapeau} defining the non-symmetric expectation $\hat{\mathbb{E}}$.
A priori, the right hand side of \eqref{eq:sketch_init} is of order $O(\tau N)$, as $g$,$\nabla f$ and $\nabla \Phi$ are bounded. 
However, the mean of $\nabla \Phi$ is equal to $0$, thus, by integrating over $x^1$, we can hope to get rid of this leading order term. 
For this we need to study the dependency between the tagged particle and a typical background particle. 
We therefore construct $(\overline{X}_t, \overline{V}_t)_t$, which is the process $(X_t,V_t)_t$ in which we delete the influence of particle $1$ by removing the force term $\frac{1}{N}\nabla\Phi(X_t-x^1_t)$ from the dynamics (see \eqref{eq: dynamique barres}). 
At the main order, we expect that the trajectories are well approximated by 
\begin{equation}
\label{eq:sketch_approx_back}
x^1_t=x^1_0+tv^1_0+O\left(\frac{t^2}{N}\right)
\quad \text{and} \quad  
X_t=\overline{X}_t+O\left(\frac{t^2}{N}\right),\qquad V_t=\overline{V}_t+O\left(\frac{t}{N}\right).
\end{equation}
Indeed, to obtain \eqref{eq:sketch_approx_back}, it is sufficient to notice that
\begin{align}
\label{eq:1/N_approx_tagged}
\frac{d}{dt}v^1_t=\frac{1}{N}\nabla\Phi(X_t-x^1_t)=O\left(\frac{1}{N}\right),
\qquad
\frac{d}{dt}\left(V_t-\overline{V}_t\right)\simeq-\frac{1}{N}\nabla\Phi(X_t-x^1_t)=O\left(\frac{1}{N}\right).
\end{align}
The rigorous proof of these estimates is the subject of Lemma~\ref{lem:influence_une_particule}, and requires also controlling the influence of particle $1$ on the rest of the background particles due to their mutual interaction with the tagged particle.
Assuming that the background particle interacts for a time of order $1$, the leading order in \eqref{eq:sketch_init} is determined by the approximation \eqref{eq:sketch_approx_back} which are 
two independent trajectories
\begin{align}
\hat{\mathbb{E}}\left[g(\overline{V}_0)\int_0^{\tau N}\nabla f(\overline{V}_t)\nabla\Phi(\overline{X}_t-x^1_0 - t v_0^1)dt\right]
= 0,
\label{eq: main order =0}
\end{align}
where we used that $x^1_0$ is uniformly distributed and $\int \nabla \Phi(x) dx = 0$.

To complete the derivation  \eqref{eq:informal_mart}, one has also to take into account the correction of order $1/N$ from \eqref{eq:1/N_approx_tagged} to identify the diffusion coefficient $D$ and the friction coefficient $\Lambda$. 
We expect from  \eqref{eq:1/N_approx_tagged} that 
\begin{align*}
\nabla f(V_t)=&\nabla f(\overline{V}_t)+\text{Hess}f(\overline{V}_t)(V_t-\overline{V}_t)+O\left(\frac{1}{N^2}\right),\\
=&\nabla f(\overline{V}_t)-\text{Hess}f(\overline{V}_t)\int_0^t\frac{1}{N}\nabla\Phi(\overline{X}_s-x^1_0-sv^1_0)ds+O\left(\frac{1}{N^2}\right)
\end{align*}
where we again used $\frac{d}{dt}\left(V_t-\overline{V}_t\right)\simeq-\frac{1}{N}\nabla\Phi(X_t-x^1_t)\simeq-\frac{1}{N}\nabla\Phi(\overline{X}_t-x^1_0-tv^1_0)$. Likewise 
\begin{align*}
& \nabla \Phi(X_t-x^1_t)=\nabla\Phi\left(\overline{X}_t-x^1_0-tv^1_0\right)+\text{Hess}\Phi\left(\overline{X}_t-x^1_0-tv^1_0\right)\left((X_t-\overline{X}_t)-(x^1_t-x^1_0-tv^1_0)\right)\\
&\qquad \qquad \qquad +O\left(\frac{1}{N^2}\right)\\
& \qquad = \nabla\Phi\left(\overline{X}_t-x^1_0-tv^1_0\right)-2\text{Hess}\Phi\left(\overline{X}_t-x^1_0-tv^1_0\right)\int_0^t\int_0^s\frac{1}{N}\nabla\Phi(\overline{X}_u-x^1_0-uv^1_0)duds\\
&\qquad \qquad \qquad
+O\left(\frac{1}{N^2}\right).
\end{align*}
Plugging these estimates back into \eqref{eq:sketch_init}, we obtain
after neglecting the main order \eqref{eq: main order =0}
\begin{align}
\label{eq: rough Taylor estimate}
\mathbb{E}&\left[g(V_0)(f(V_{\tau N})-f(V_0))\right]\\
&\quad= \frac{2}{N}\hat{\mathbb{E}}\left[g(V_0)\int_0^{\tau N}\nabla f(\overline{V}_t)\text{Hess}\Phi\left(\overline{X}_t-x^1_0-tv^1_0\right)\int_0^t\int_0^s\nabla\Phi(\overline{X}_u-x^1_0-uv^1_0)dudsdt\right] \nonumber\\
&\quad+\frac{1}{N}\hat{\mathbb{E}}\left[g(V_0)\int_0^{\tau N} \text{Hess}f(\overline{V}_t)\int_0^t\nabla\Phi(\overline{X}_s-x^1_0-sv^1_0)ds\nabla\Phi\left(\overline{X}_t-x^1_0-tv^1_0\right)dt\right]
+O\left(\frac{\tau}{N}\right). \nonumber
\end{align}
The second term, involving $\text{Hess}f$, will be identified as the diffusion coefficient $D$ introduced in \eqref{def:diffusion}.
It corresponds to the variance of the random force  alluded in \eqref{eq: TCL force} and takes into account the time correlations.
The term involving $\nabla f$ describes the friction $\Lambda$ arising from the collisions with the gas (see \eqref{def:drift}).
To understand intuitively the friction term, note that 
a gas particle will typically be deflected by an order $1/N$ while interacting with the tagged particle. This small deviation from a straight line induces a drag force as explained in Figure \ref{fig:illustration_friction}.

This justifies the identity  \eqref{eq:informal_mart} provided that 
a background particle interacts for a short time and only once. 
As explained below, the core of the proof is to control the correlations between the tagged particle and a typical background particle. 


\begin{figure}[H]
    	\centering
	\includegraphics[width=0.6\linewidth]{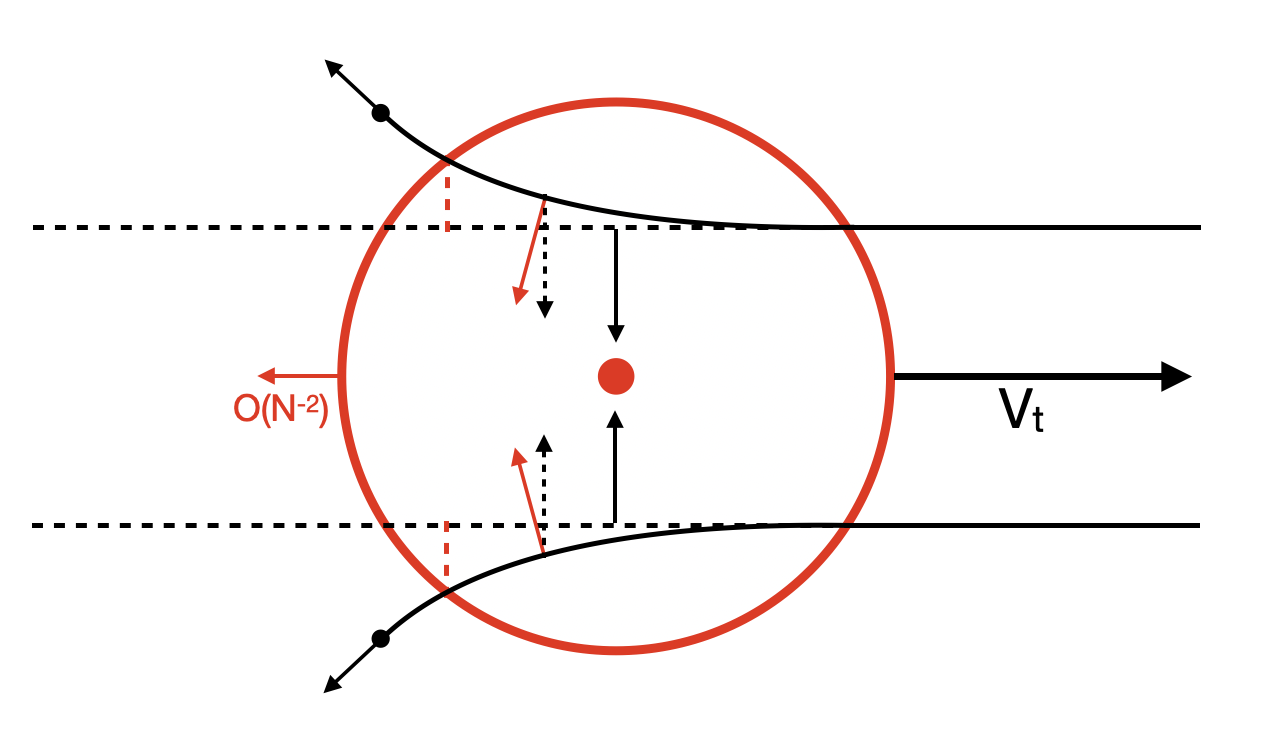}
	\caption{
    { \small \textbf{Appearance of a friction.}
    This figure depicts two symmetric trajectories  of background particles (in black) deflected by a repulsive interaction with the tagged particle (in red). 
    The main contribution of the force (solid black arrow, of order $N^{-1}$) created by a background particle cancels out on average. This leads to a random force of mean 0 and variance $1/N$ acting on the tagged particle. However, the deviation of order $N^{-1}$ of the background particles shifts slightly  this force (solid red arrow) and ultimately induces a friction of order $N^{-2}$. Since during a time interval $[0,t]$ the tagged particle meets $\sim Nt$ background particles (each yielding this tiny force of strength $N^{-2}$), for $t\simeq N$, a friction force of order $1$ appears. 
    }}
	\label{fig:illustration_friction} 
\end{figure}


\paragraph{Sources of correlation.} 

If the tagged particle encounters a background particle with a similar velocity, then correlations will build up between these two particles as they may interact for a very long time, leading to some resonance effect. 
We stress that this feature doesn't arise in  a Lorentz gas where the scatterers are fixed (see e.g. \cite{DGL87}).
The second source of correlations is due to particles recolliding after having left the interaction radius (see Figure~\ref{Fig:time_correl}).
A major difficulty in this paper is to show that these two types of correlations will not contribute to the evolution of the tagged particle in the large $N$ limit.


\begin{figure}[H]
    	\centering
    \includegraphics[width=0.9\linewidth]{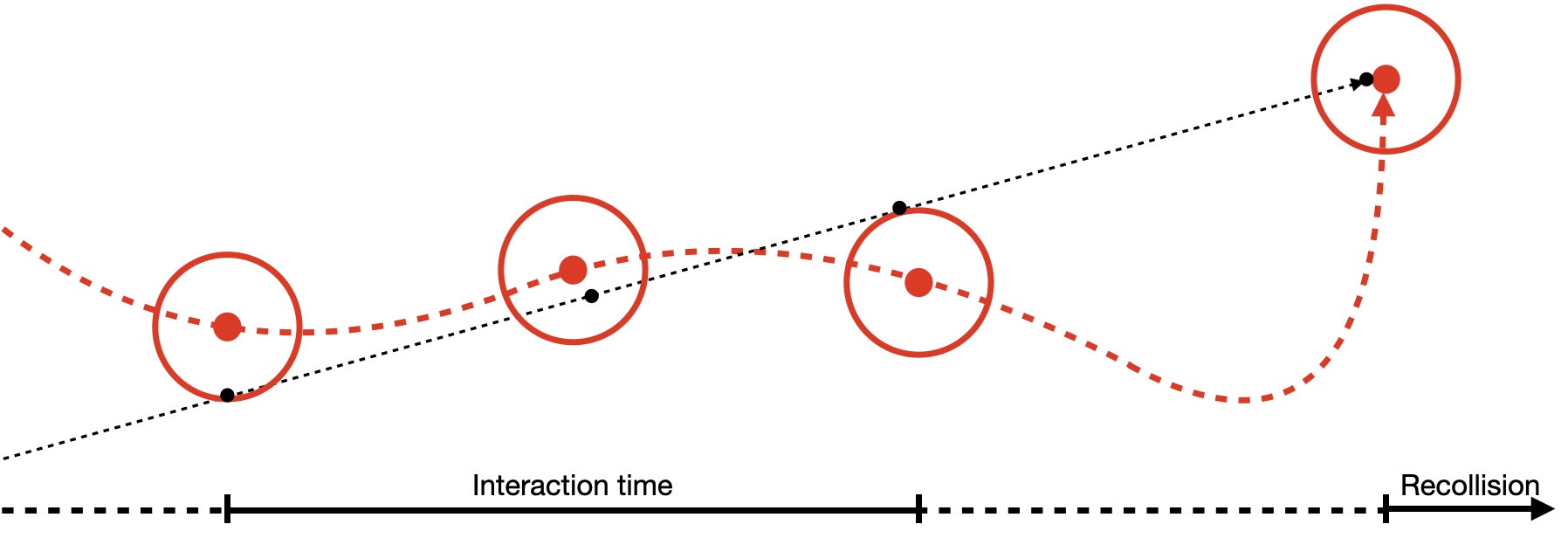}
    \caption{\small{\textbf{Sources of time correlation.} A background particle interacting for a long time, and a background particle recolliding. Both the interaction time and the first possible recollision time are controlled by the relative velocity.}}
    \label{Fig:time_correl}
\end{figure}

\paragraph{The timescale $N^{1/4}$.} 

To estimate the influence of the background particle $1$ on the tagged particle, we consider the modified process $(\overline{X}_t, \overline{V}_t)_t$ which evolves according to the same rule as the tagged particle without interacting with particle $1$.
As already mentioned in \eqref{eq:sketch_approx_back}, we will compare the original evolution and the one of the modified process. 
This difference obeys the following approximate equation
\begin{align}
\frac{d^2}{dt^2}(X_t-\overline{X}_t)=&-\frac{1}{N}\sum_{i=1}^\mathcal{N}\nabla\Phi(X_t-x^i_t)+\frac{1}{N}\sum_{i=2}^\mathcal{N}\nabla\Phi(\overline{X}_t-x^i_t)\nonumber\\
\simeq&\left(-\frac{1}{N}\sum_{i=1}^\mathcal{N}\text{Hess}\Phi(X_t-x^i_t)\right)(X_t-\overline{X}_t)-\frac{1}{N}\nabla\Phi(\overline{X}_t-x^1_t).
\label{eq:sketch_gronwall}
\end{align}
Note that the (indirect) influence of $x^1$ on the evolution of the other background particles has been neglected for simplicity. 
As for the evaluation of the force on the tagged particle \eqref{eq: TCL force}, one expects that at any time the background particles are uniformly distributed so that with high probability (see Lemma~\ref{lem:max_value_drift}):
\begin{equation}
\label{eq:sketch_control}
 \left\|\frac{1}{N}\sum_{i}\text{Hess}\Phi(X_t-x^i_t)\right\|=O\left( \frac{1}{\sqrt{N}}\right).
\end{equation}
Thus \eqref{eq:sketch_gronwall} is a second order equation of the form $\ddot{x}(t)=a(t)x(t)+b(t)$ where $\|a(t)\|\leq 1/\sqrt{N}$ and $|b(t)|\leq1/N$. Through a Grönwall-like result, we obtain the  bound $|X_t-\overline{X}_t|=O\left(t^2/N\right)$ provided 
$t\leq N^{1/4}$ (see Lemma~\ref{lem:Grönwall_precis_alpha}).
We deduce that 
\begin{center}
\textit{$(X_t)$ can be approximated  by $(\overline{X}_t)$  on timescales much smaller than $N^{1/4}$.}
\end{center}
This provides a first estimate on the timescale at which the influence of a given background particle can be neglected and the Taylor expansion \eqref{eq: rough Taylor estimate} can be justified.
Beyond the timescale $N^{1/4}$, the naive Gronwall estimate \eqref{eq:sketch_gronwall} needs to be improved.

\smallskip

 If the tagged particle and particle $1$ start interacting at time $s$, then the interaction time  can be roughly estimated by the inverse of the relative velocity $1/|v^1_s-V_s|$. The smaller the relative velocity is, the longer the background particle may interact.
Typically, under the invariant measure, the relative velocity has a Gaussian distribution in dimension $d$, so that 
\begin{align}
\label{eq: initial contrainte vitesse}
\mathbb{P}\left[|v^1_0-V_0| < N^{-\gamma}\right]\lesssim N^{-d\gamma}.
\end{align}
For $\gamma >1/d$, this upper bound is much smaller than $1/N$ and this will be  sufficient to neglect the influence of the {\it slow} background particles, i.e. those with relative velocities much less than $N^{ - 1/d}$.
Indeed for $\gamma >1/d$, the estimate \eqref{eq: initial contrainte vitesse} implies that the number of slow background particles interacting during the time interval $[0,T]$ is less than $T \times N \times N^{-d\gamma} \ll N$ for $T = O( N)$. As a background particle induces a force of order $1/N$, the total force from the slow particles will be negligible on the macroscopic timescale $T = O( N)$.

Thus, it is enough to consider only the background particles with relative velocity $|v^1_0-V_0|$ of order at least $N^{-1/d}$. In dimensions $d \geq 5$, the corresponding interaction time is at most $\frac{1}{|v^1_0-V_0|} \leq N^{1/5}$ which is much shorter than the timescale $N^{1/4}$ on which we know how to control the correlations between these two particles. 
In this way, the Taylor expansion \eqref{eq: rough Taylor estimate} can be justified provided the two particles do not recollide after this first interaction (see Figure \ref{Fig:time_correl}).

\smallskip

The last step is to show that the probability of a recollision between the two particles is much smaller than $1/N$ and can therefore be neglected. 
For this, we use that, for a relative velocity large enough, particle $1$ will not only cross the trajectory of the tagged particle, but it will also move away far enough before the tagged particle can catch it up. Indeed, as already mentioned in \eqref{eq: TCL force}, 
it is easy to deduce a priori estimates (with high probability) on the tagged particle velocity (see Definition \ref{def:hyp_good_set})
\begin{equation}
\label{eq: TCL force bis}
\frac{d}{dt}V_t=-\frac{1}{N}\sum_{i\in\mathbb{N}} \nabla\Phi(X_t-x^i_t) = O \left( \frac{1}{\sqrt{N}} \right)
\quad \Rightarrow \quad 
X_t= X_0+tV_0+O\left(\frac{t^2}{\sqrt{N}}\right).
\end{equation}
Thus for times $t \ll N^{1/4}$ the tagged particle evolves almost in a straight line. 
If the relative velocity is large enough during the first interaction, the two particles will drift far apart before the fluctuations of the tagged particle velocity can lead to a recollision.
Once this initial dispersion effect has occurred, we  show in Section \ref{sec:recollision} that a recollision is very unlikely.

The previous heuristic arguments are enough to control the small relative velocities and the recollisions in dimensions $d \geq 5$, but fall short in smaller dimensions as the correlations have been estimated only for times much smaller than $N^{1/4}$. The next step is to improve on the timescale $N^{1/4}$.


\paragraph{The timescale $N^{1/3}$.}

We expect that the tagged particle is well behaved on much larger  timescales than $N^{1/4}$ and in particular it should evolve in straight-line motion  up to times of order $N^{1/3}$.
Indeed the random force of order $1/\sqrt{N}$ acting on the tagged particle \eqref{eq: TCL force bis} is constantly renewed  in time by the flux of new background particles so that a perfect averaging of the fluctuations in space and time should lead to 
\begin{align}
\label{eq: ideal averaging}
V_T - V_0 =
-\frac{1}{N}\int_0^T\sum_{i}\nabla \Phi(X_t-x^i_t)dt \simeq \sqrt{\frac{T}{N}}.
\end{align}  
Integrating again in time would then imply that
\begin{align}
\label{eq: ideal averaging}
X_T = X_0 + T V_0 + O \left( \frac{T^{3/2}}{N^{1/2}} \right).
\end{align}
At times of order $N^{1/3}$, the fluctuations start to kick in and the tagged particle trajectory can no longer be approximated by a straight-line motion.
We cannot obtain such a precise control, but we can still gain from the time averaging.
More precisely, we will show in Propositions~\ref{prop:fin_bootstrap} and ~\ref{prop:init_bootstrap} that 
for $\alpha=\frac{1}{4(d+2)}$ and some $\beta > 1/2$, 
we have for $T\leq N^{\beta}$
\begin{equation}
\label{eq:sketch_boot_init}
|V_T-V_0|=\left|\frac{1}{N}\int_0^T\sum_{i}\nabla \Phi(X_t-x^i_t)dt\right|\leq \sqrt{\frac{T}{N}}N^{\alpha}.
\end{equation}
The same calculations hold for $\left\|\frac{1}{N}\int_0^T\sum_{i}\text{Hess} \Phi(X_t-x^i_t)dt\right\|$ and lead to an improved Gronwall estimate in \eqref{eq:sketch_gronwall}.
Similarly as \eqref{eq: ideal averaging}, the maximal timescale provided by \eqref{eq:sketch_boot_init} is now $N^{(1-2\alpha)/3}$.
This threshold appears in several instances in this paper, see e.g.
Lemma~\ref{Lem: Better set estimates}.
This is enough to control the correlations on  timescales much larger than $N^{1/4}$ and therefore to conclude also in dimension $d = 4$.

To derive \eqref{eq:sketch_boot_init}, we decompose the background particles involved in \eqref{eq: ideal averaging} into two classes: the fast particles which are interacting for short times and the slow ones. 
Restricting the sum \eqref{eq:sketch_boot_init} to fast particles allow us to recover enough independence in time and an approximated martingale structure for which concentration type estimates can be applied.
The loss factor  $N^\alpha$ in \eqref{eq:sketch_boot_init} comes from the slow particles whose contribution is hard to average out but which are much more rare.


\paragraph{The issues in dimension $d=3$.} Throughout the proof, the main point is that slow background particles create time correlations that are difficult to handle in dimension $d=3$ for the following reasons:

\begin{enumerate}
\item Only the relative velocities much smaller than $N^{-1/3}$ can be discarded by the argument sketched after \eqref{eq: initial contrainte vitesse}. This means that in $d=3$, one has to consider background particles with interaction times of order at least $N^{1/3}$. This is out of reach with our current estimates \eqref{eq:sketch_boot_init}. In fact for such small relative velocities, the decoupling mechanism will be different than the one we described where the background and the tagged particle are crossing essentially in straight-line motions. For times of order $N^{1/3}$ the fluctuations of the tagged particle become relevant and they will be the main force leading to the separation of the two particles. It is not the background particle which exits the tagged one, as we claim previously, but
rather the opposite.
\item Once a background particle has interacted with the tagged particle, we have also to show that  the probability of a recollision is very small (less than $1/N$). 
For this we use heavily, in Proposition \ref{prop:proba_recol},  the dispersion after the first interaction to make sure that the two particles are far apart so that a recollision will be extremely unlikely.
For the estimates of Proposition \ref{prop:proba_recol} to remain valid in dimension $d=3$, it would be necessary to restrict attention to interactions with background particles having relative velocities much larger than
 than $N^{-1/4}$.
\end{enumerate}
In summary, the present estimates are insufficient to handle, in dimension  $3$,  the background particles interacting with relative velocities in $[N^{-1/3}, N^{-1/4}]$. 
Nonetheless, we write the results assuming $d\geq3$ whenever possible.
Handling dimension $2$ would be even harder as the correlations are stronger.

%
%
%
%

\subsection{Notations}

The most peculiar notations the reader should get used to are the following:
\begin{center}
\fbox{
\begin{minipage}{0.9\textwidth}
\textbf{$\bullet$} For two functions $f$ and $g$ of the various parameters appearing in this article (denoted here and only here $\mathcal{Z}$), we write $f(\mathcal{Z})=O\left(g(\mathcal{Z})\right)$ if there exists a constant $C$, \textbf{depending only on $R$, $d$ and $\Phi$}, such that (possibly for $N$ large enough)
\begin{align*}
\forall \mathcal{Z},\qquad |f(\mathcal{Z})|\leq C|g(\mathcal{Z})|.
\end{align*}
This differs from the usual  notations in the fact that the control holds for any parameter $Z$, with an asymptotic that may only be taken in $N$. Similarly, for $h$ a parameter or a function, we write $f(\mathcal{Z})=O_h\left(g(\mathcal{Z})\right)$ if the constant $C$ given above depends on $h$. We also write $f(\mathcal{Z})=o\left(g(\mathcal{Z})\right)$ if both $f(\mathcal{Z})=O\left(g(\mathcal{Z})\right)$ and $\left|\frac{f(\mathcal{Z})}{g(\mathcal{Z})}\right|\xrightarrow[N\rightarrow\infty]{}0$.
\end{minipage}
}
\end{center}

\begin{itemize}
\item $\mathcal{N}$ is the (random) number of background particles in $\mathbb{T}=[-N^\eta,N^\eta]^d$, with $\eta>2$.
\item $X_t\in \mathbb{T}$ and $V_t\in\mathbb{R}^d$  are respectively the position and velocity of the tagged particle at time $t$, and we write $Z_t=(X_t,V_t)$.
\item $(x^i_t)_{i\in\{1,...,\mathcal{N}\}}\in \mathbb{T}^\mathcal{N}$ and  $(v^i_t)_{i\in\{1,...,\mathcal{N}\}}\in (\mathbb{R}^d)^\mathcal{N}$ are respectively the positions and velocities of the background particles at time $t$ when given the total number of particles $\mathcal{N}$. We may write $(x^i_t)_{i}$ and $(v^i_t)_{i}$ when $\mathcal{N}$ is not relevant. We also denote $x^{1:\mathcal{N}}_t=(x^i_t)_{i}$ and $v^{1:\mathcal{N}}_t=(v^i_t)_{i}$, as well as $x_{1:\mathcal{N}}=(x_1,..., x_\mathcal{N})$ and $v_{1:\mathcal{N}}=(v_1,..., v_\mathcal{N})$, usually for initial conditions. We finally denote $z_i=(x_i,v_i)$, and adapt all previous notations to $z$.
\item $(\overline{X}, \overline{V}, \overline{x}^{1:\mathcal{N}},  \overline{v}^{1:\mathcal{N}})$ are defined in \eqref{eq: dynamique barres} and studied Section~\ref{sec:influence_part}. They correspond to the processes for which we removed the influence of particle 1.
\item $\mathcal{B}(x,R)$ is the ball in $\mathbb{R}^d$ of center $x$ and radius $R$ (and thus also the ball in $\mathbb{T}$ since $R$ is significantly smaller than the size of $\mathbb{T}$).
\item $T_{m}:\mathbb{R}^d\times\mathbb{R}^d\mapsto\mathbb{R}$ is a function which gives a bound on the interaction time as a function of the velocities of the background particle and the tagged particle, defined in Definition~\ref{def:inter_time_0}.
\item $\mathcal{G}_N(\delta)$ (resp. $\mathcal{G}_N(\delta,\alpha, \beta)$) is the \textit{good} (resp. \textit{better}) set defined in Definition~\ref{def:hyp_good_set} (resp. Definition~\ref{def:hyp_boot}) and represents the set of initial conditions for which the various controls we require hold.
\item $\alpha$ and $\beta$ are given in Definition~\ref{def:hyp_boot} and characterize the control
\begin{align*}
|V_t-V_0|\leq \sqrt{\frac{t}{N}}N^\alpha\qquad \text{ for }\qquad t\leq N^{\beta}.
\end{align*}
 $\alpha^*, \beta^*$ will then be the best values of $\alpha, \beta$ we can obtain, given in Proposition~\ref{prop:fin_bootstrap}.
 \item $C^1_t$, defined in Section~\ref{sec:recollision}, is the recollision event.
\item For two d-dimensional vectors $X$ and $Y$, we denote $X\cdot Y:=\sum_{i=1}^d X_iY_i$ the standard Euclidean scalar product. We therefore also consider the usual Euclidean distance $|X|=\sqrt{X\cdot X}$. For a $d\times d$ real matrices $A,B$, we denote $\|A\|=\sup_{X\neq 0}\frac{|AX|}{|X|}$ and $A:B=\sum_{i,j=1}^dA_{i,j}B_{i,j}$.
\item For two d-dimensional vectors $X$ and $Y$, we denote $X\otimes Y$ the matrix defined by $(X\otimes Y)_{i,j}=X_iY_j$. We may also denote $X^{\otimes 2}=X\otimes X$. 
\item We denote, for $a,b\in\bbR$, $a\wedge b=\min(a,b)$ and $a\vee b=\max(a,b)$.
\item We define the first (positive) entry time of the first background particle as
\begin{align*}
\sigma^+_1:=\sigma^+_1(\mathcal{N}, Z_0, z_{1:\mathcal{N}}):=\inf \left\{t\geq0\text{ s.t. }x^1_t\in\mathcal{B}(X_t,R)\right\},
\end{align*}
and the first entry time in $[-N,N]$ as 
\begin{align*}
\sigma_1:=\sigma_1(\mathcal{N}, Z_0, z_{1:\mathcal{N}}):=\inf \left\{t\in[-N,N]\text{ s.t. }x^1_t\in\mathcal{B}(X_t,R)\right\}.
\end{align*}
\end{itemize}

%
%
%
%

\section{Convergence of the tagged particle}

In this section we prove both the tightness of the trajectories  $(V_{\tau N})_{\tau\in[0,1]}$ and the fact that the limit satisfies a martingale problem.

%
%
%
%

\subsection{Martingale problem}\label{sec:martingale_formulation}

Recall that Itô's formula applied to the limiting process \eqref{eq:eds_limit} reads for a sufficiently regular test function $f$ and times $\tau>\tau'$
\begin{align*}
f(\mathcal{V}_{\tau})-f(\mathcal{V}_{\tau'})=\int_{\tau'}^\tau \mathcal{L}f(V_s)ds+\mathcal{M}_{\tau',\tau},
\end{align*}
where
\begin{equation}\label{eq:def_gen_nl}
\mathcal{L}f(V)=2\nabla f(V)\cdot \Lambda(V)+\text{Hess}f(V):D(V),
\end{equation}
with $\Lambda(V)$ and $D(V)$ defined in \eqref{def:drift} and \eqref{def:diffusion} respectively, and $\mathcal{M}$ is a martingale independent of the trajectory before time $\tau'$. Our goal is to show that, in the limit $N\rightarrow\infty$, the process $(V_{\tau N})_{\tau\in[0,1]}$ satisfies the same property.

Below, we denote $\mathcal{C}^1_{bb}(\mathbb{R}^d,\mathbb{R})$ the set of bounded continuous functions with bounded continuous derivative.


\begin{proposition}\label{prop:martingale_formulation}
Assume $d\geq 4$. There exists $\omega\in]0,1]$ and $\gamma_r<\frac{1}{3}$ such that, for any $n\in\mathbb{N}$, any functions  $g_1,....,g_n,f:\mathbb{R}^d\mapsto \mathbb{R}$ satisfying
\begin{align*}
\forall i, g_i\in \mathcal{C}^1_{bb}(\mathbb{R}^d,\mathbb{R}),\qquad  f\in\mathcal{C}^3(\mathbb{R}^d,\mathbb{R}), \qquad \nabla f\in W^{2,\infty}(\mathbb{R}^d),
\end{align*}
 we have, for all $0\leq \tau_1<...<\tau_n\leq1$, for all $N\in\mathbb{N}$, and all $\tau_n\leq \tau_{n+1}\leq1$ (we assume, for the sake of simplicity, that only $\tau_{n+1}$ may depend on $N$)\footnote{Because we use this estimate later to prove equicontinuity results in Lemma~\ref{lem:equi}, we wish to be able to consider $(\tau_{n+1}-\tau_n)N$ to be as small as possible: when trying to use this proposition to prove that $\mathbb{E}[|V_{\tau N}-V_{\sigma N}|^4]=O\left(|\tau-\sigma|^{1+\epsilon}\right)$, errors of the form $(\tau-\sigma)/N^\omega$ cannot be bounded by $C|\tau-\sigma|^{1+\epsilon}$ with a constant $C$ that is independent of $N$ and $|\tau-\sigma|$. This therefore requires a more careful analysis, hence why we allow for $\tau_{n+1}$ to depend on $N$.}
\begin{align*}
\mathbb{E}&\left[\prod_{i=1}^ng_i(V_{\tau_iN})\left(f(V_{\tau_{n+1}N})-f(V_{\tau_nN})-\frac{1}{N}\int_{\tau_nN}^{\tau_{n+1}N}\mathcal{L}f(V_t)dt\right)\right]\\
&\qquad\qquad=O\left(\|\nabla f\|_{W^{2,\infty}}\prod_{i=1}^{n}( \|g_i\|_\infty+ \|\nabla g_i\|_\infty)\left(\frac{\tau_{n+1}-\tau_n}{N^\omega}+\left(\frac{N^{\gamma_r}}{N}\wedge (\tau_{n+1}-\tau_n)\right)\right)\right).
\end{align*}
\end{proposition}


\begin{remark}\label{rk:dep_norme}
Note that we could be more precise in the dependence of the error on the norms of the various test functions.
Indeed, the parameter $\omega$ comes, as explained below, from a control on recollisions, which is independent of the second and third derivatives of the test function $f$ (and only multiplies $\|\nabla f\|_\infty$ and $\|g_i\|_\infty$). 
On the contrary, the bound $\|\nabla f\|_{W^{2,\infty}}$ appears in a Taylor expansion, and multiplies an error term coming from the control on the interaction time (see Lemma~\ref{lem:abstract_approx}). It is  in fact divided by $N^{\Omega}$ for a constant $\Omega>\omega$ coming from \eqref{eq:erreur_T_m_k_negli}.
\end{remark}


\begin{proof}[Proof of Proposition~\ref{prop:martingale_formulation}]
 Up to adding a test function $g_0$ at time $\tau_0=0$, the initial velocity distribution $g_0$ can be replaced by $g_0=1$.
Fix $n,N\in\mathbb{N}$, $0\leq \tau_1<...<\tau_n<\tau_{n+1}\leq1$ and some functions $g_1,....,g_n,f$ satisfying the assumptions of the proposition. Assume without loss of generality that $\|\nabla f\|_{W^{2,\infty}}=1$, and that for all $i\in\{1,...,n\}$ we have $\|g_i\|_\infty+\|\nabla g_i\|_\infty =1$.

The first thing we do, for the sake of clarity, is shifting the time so as to set the $0$ in between the times $\tau_{n-1}N$ and $\tau_nN$. This way we clearly separate the negative times, i.e. the past of the trajectory (involved in the computations of $g_i(V_{\tau_iN})$ for $i<n$), from the martingale increment $g_n(V_{\tau_nN})(f(V_{\tau_{n+1}N})-f(V_{\tau_nN}))$ for positive times. As shown later, we are able to ensure that a background particle involved in the evolution of the function $f$ (i.e. the martingale increment) actually does not interact with the trajectory for negative times.

Denote for all $i\in\{1,...,n-1\}$ the microscopic time $t_i=-(\tau_i-\tau_{n}+\frac{\tau_{n}-\tau_{n-1}}{2}) N$ and, for $i\in\{n,n+1\}$, similarly $t_i=(\tau_i-\tau_{n}+\frac{\tau_{n}-\tau_{n-1}}{2}) N$ (we just change signs so that all $t_i$ are nonnegative, and thus we may more easily distinguish between positive and negative times below).  Using the time invariance of the process, we have
\begin{align*}
\mathbb{E}\left[\prod_{i=1}^ng_i(V_{\tau_iN})\left(f(V_{\tau_{n+1}N})-f(V_{\tau_nN})\right)\right]=\mathbb{E}\left[\left(\prod_{i=1}^{n-1}g_i(V_{-t_i})\right)g_n(V_{t_n})\left(f(V_{t_{n+1}})-f(V_{t_n})\right)\right].
\end{align*}
From the evolution equations \eqref{eq:def_tagged} and \eqref{eq:def_process_back} we get
\begin{align}
\mathbb{E}&\left[\left(\prod_{i=1}^{n-1}g_i(V_{-t_i})\right)g(V_{t_n})\left(f(V_{t_{n+1}})-f(V_{t_{n}})\right)\right]\nonumber\\
&\qquad=-\mathbb{E}\left[\left(\prod_{i=1}^{n-1}g_i(V_{-t_i})\right)g_n(V_{t_n})\int_{t_n}^{t_{n+1}}\nabla f(V_t)\left(\frac{1}{N}\sum_{i}\nabla\Phi(X_t-x^i_t)\right)dt\right]\nonumber\\
&\qquad=-\hat{\mathbb{E}}\left[\left(\prod_{i=1}^{n-1}g_i(V_{-t_i})\right)g_n(V_{t_n})\int_{t_n}^{t_{n+1}}\nabla f(V_t)\nabla\Phi(X_t-x^1_t)dt\right],
\label{eq:depart_calcul}
\end{align}
where the notation $\hat{\mathbb{E}}$ was introduced in \eqref{eq: definition E chapeau}.

In this expectation, the relevant contributions come from configurations such that the tagged particle and particle $1$ are interacting for a limited amount of time. More precisely, denoting by $\sigma^+_1$ the first (positive) interaction time 
\begin{align}\label{eq:def_tps_entree}
\sigma^+_1=\inf\{t\geq 0,\quad |X_t-x^1_t|\leq R\},
\end{align}
we will show that it is enough to consider microscopic configurations in a good set $\mathfrak{G}_N$ such that the interaction between these two particles up to time $t_{n+1}$ is limited to a time interval of the form $[\sigma^+_1, \sigma^+_1+ N^{a}]$ for some $a>0$.
The interaction time depends on the relative velocities, thus we introduce the following notations.


\begin{definition}[Interaction time]\label{def:inter_time_0}
Define for $v,V\in\mathbb{R}^d$
\begin{equation}
\label{eq:def_inter_time 0}
T_m(v,V)=T_m(|v-V|)=\frac{1}{|v-V|}.
\end{equation}
\end{definition}

The following lemma defines a set of good configurations for which the interaction is well controlled.


\begin{lemma}
\label{lem:abstract_tout_va_bien}
 Recall that $d \geq 4$. There exist $\omega'>0$ and sets $\mathfrak{G}_N(t)$ of initial conditions indexed by $t$ in $[0,N]$ such that
\begin{equation}
\label{eq:proba_abstract_good}
\hat{\mathbb{E}}\left[\left(1-\mathds{1}_{\mathfrak{G}_N(t)}\right)\mathds{1}_{|x^1_t-X_t|\leq \frac{3}{2}R}\right]=O\left(N^{-(1+\omega')}\right),
\end{equation}
and such that for $(\mathcal{N},Z,z_{1:\mathcal{N}})\in\mathfrak{G}_N(t)$, we have 
\begin{enumerate}[label=(\roman*)]
\item  particle $1$ is in the (slightly extended) interaction radius at time $t$, i.e. $|x^1_t-X_t|\leq \frac{3}{2}R$, 
\item  its relative speed is such that $|v^1_t-V_t|\geq N^{-\gamma_r}$, for $\gamma_r=\frac{1}{4}+\frac{1}{36}<\frac{1}{3}$ (thus satisfying $d\gamma_r>1$),
\item  particle $1$ interacts for a limited amount of time, in the sense that for $s\in[-N,N]$ such that 
\begin{equation}
\label{eq: point iii}
|s-t|\geq 9 R \, T_m(v^1_t,V_t) = 9R|v^1_t-V_t|^{-1} \quad \text{then} \quad  |x^1_s-X_s|>\frac{3}{2}R,
\end{equation}
\item the controls from Lemma~\ref{lem:abstract_approx} below hold true.
\end{enumerate}
\end{lemma}

 Note that Lemma~\ref{lem:abstract_tout_va_bien} (iii) ensures two things. First, as stated, it gives a bound on the interaction time. Second, and more importantly, it implies that particle $1$ does not recollide with the tagged particle before or after its interaction time frame. Proving this point is actually the main technical aspect of this work and most of this paper is devoted to the construction of the sets $\mathfrak{G}_N(t)$.
 The proof below of Lemma \ref{lem:abstract_tout_va_bien} relies on the results in Sections  \ref{sec: The Good and the Better set} and \ref{sec: Some useful controls}. It can be skipped in a first reading as the properties (i)-(iv) of the sets are enough to derive the martingale property.


\begin{proof}[Proof of Lemma~\ref{lem:abstract_tout_va_bien}]
Point (i) is an assumption on the configurations.
Note that if Point (ii) fails then \eqref{eq:proba_abstract_good} holds as $d\gamma_r>1$ and
\begin{align*}
\mathbb{P}\left[|v^1_t-V_t| \leq N^{-\gamma_r}\right]=\mathbb{P}\left[|v^1_0-V_0|\leq N^{-\gamma_r}\right]=O\left(N^{-d\gamma_r}\right).
\end{align*}
Assuming the condition on the relative velocity in Point (ii), then Point (iii) is the result of Lemma~\ref{lem:max_inter_time} and Proposition~\ref{prop:fin_bootstrap}~and~\ref{prop:proba_recol}  below (resp. proving the maximum interaction time, the probability of configurations belonging to a set satisfying enough controls, and the probability of recollision in said set)\footnote{In these results, we formally change $R$ into $\frac{3}{2}R$. The conclusions obviously remain the same.}. 
Point (iv) is the result of Lemma~\ref{Lem: Better set estimates}.
\end{proof}

For configurations in $\mathfrak{G}_N(t)$, the interaction time is short enough so that the weak forces between the tagged particle and particle 1 hardly modify their trajectories during the time interval ${[\sigma^+_1, \sigma^+_1+9R|v^1_t-V_t|^{-1}]}$.  In particular, the trajectory of particle 1 can be approximated by a straight line 
\begin{equation*}
\bar x^1_t=x_1+t v_1,
\end{equation*} 
and up to time $\sigma^+_1+ 9R|v^1_t-V_t|^{-1}$, the tagged particle remains close to the auxiliary process 
$(\bar Z_t)_t =(\overline{X}_t, \overline{V}_t)_t$ which is defined as  $(Z_t)_t $ but without particle 1, i.e.
\begin{align}
\label{eq: dynamique barres}
\begin{cases}
\overline{X}_0=X_0,\quad \overline{V}_0=V_0,\quad \frac{d}{dt}\overline{X}_t=\overline{V}_t\quad \text{ and }\quad \frac{d}{dt}\overline{V}_t=-\frac{1}{N}\sum_{i=2}^\mathcal{N} \nabla\Phi(\overline{X}_t-\overline{x}^i_t),\\
i \geq 2,\quad \overline{x}^i_0= x_i,\quad \overline{v}^i_0=v_i,\quad \frac{d}{dt}\overline{x}^i_t=\overline{v}^i_t\quad \text{ and }\quad \frac{d}{dt}\overline{v}^i_t=\frac{1}{N}\nabla\Phi(\overline{X}_t-\overline{x}^i_t).
\end{cases}
\end{align}
The trajectory $(\bar Z_t)_t$ depends only the initial data $(\mathcal{N}, Z_0,z_{2:\mathcal{N}})$ and coincides with $(Z_t)_t$ up to time $\sigma^+_1$. 
After $\sigma^+_1$, the tagged particle starts interacting with  particle 1 with a small force $\pm \frac{1}{N}\nabla\Phi(X_t-x^1_t)$ which will modify their trajectories and contribute to the expectation  \eqref{eq:depart_calcul}. The trajectory of particle 1 deviates from the  straight line as follows
\begin{align}
\label{eq:approx_trajectory_1}
x^1_t- \bar x^1_t =\int_0^t (v^1_s-v^1_0)ds=\frac{1}{N}\int_0^t ds \int_0^s du \, \nabla\Phi(X_u-x^1_u) .
\end{align}
The difference between $(Z_t)_t$ and $(\bar Z_t)_t$ is harder to control as the tagged particle induces also a retroaction on the other background particles. One can show that during the interaction time the leading contribution remains the one of the force $- \frac{1}{N}\nabla\Phi(X_t-x^1_t)$.

We can now state that the processes $(Z_t)_t$ and $(\bar Z_t)_t$ remain close in the following sense:


\begin{lemma}
\label{lem:abstract_approx}
Recall the notation \eqref{eq:def_inter_time 0} of the interaction time 
$T_m(v^1_t,V_t) =|v^1_t-V_t|^{-1}$ and 
fix $\delta>0$ arbitrarily small. 
Then for any  $t\in[-N,N]$ and an initial data in $\mathfrak{G}_N(t)$, for all $s\in[-N,N]$ such that $|s-t|\leq 9 R \, T_m(v^1_t,V_t)$,
the following holds 
\begin{align}
&\left|V_s-\overline{V}_s\right|=O\left(\frac{T_m(v^1_t,V_t)}{N}\right),\qquad \left|X_s-\overline{X}_s\right|=O\left(\frac{T_m(v^1_t,V_t)^2}{N}\right),\label{eq:abstract_approx}\\
&\left|V_s-\overline{V}_s+\frac{1}{N}\int_{0}^s\nabla\Phi(\overline{X}_u-\overline{x}^1_u)du  \right|
\leq \mathcal{E}_1 (N,v^1_t,V_t):= O\left(\frac{T_m(v^1_t,V_t)^3}{N^{\frac{3-\delta}{2}}}+\frac{T_m(v^1_t,V_t)^5}{N^{2}}\right),\label{eq:abstract_approx_V_prec}\\
&\left|X_s-\overline{X}_s+\frac{1}{N}\int_{0}^s\int_0^u\nabla\Phi(\overline{X}_w-\overline{x}^1_w)dwdu\right|
\leq \mathcal{E}_2 (N,v^1_t,V_t):= O\left(T_m(v^1_t,V_t)  \mathcal{E}_1 (N,v^1_t,V_t) \right).\label{eq:abstract_approx_X_prec}
\end{align}
Furthermore, for $\alpha=\frac{1}{4(d+2)}+\delta$, we have
\begin{align}
|V_t-V_s|=O\left(\sqrt{\frac{|t-s|}{N}}N^{\alpha}\right).
\label{eq:abstract_variation_vitesse}
\end{align}
\end{lemma}


\begin{proof}
Again, \eqref{eq:abstract_approx}-\eqref{eq:abstract_approx_V_prec}-\eqref{eq:abstract_approx_X_prec} are the results of Lemma~\ref{Lem: Better set estimates} below. Equation \eqref{eq:abstract_variation_vitesse} is the result of \eqref{eq:control_diff_velocite_alpha} and Proposition~\ref{prop:fin_bootstrap}.
\end{proof}

Finally, the following statement is the counterpart of Lemma \ref{lem:abstract_tout_va_bien} for the modified dynamics \eqref{eq: dynamique barres}. The proof can be found in Appendix~\ref{sec:preuves_lemmes_mart} and is to be skipped on first reading. 


\begin{lemma}\label{lem:mart_toutvabien_avec_bar}
We have,
\begin{equation}\label{eq:mart_toutvabien_avec_bar}
\hat{\mathbb{E}}\left[\left(1-\mathds{1}_{\good(t)}\right)\mathds{1}_{|t-\sigma_1^+|\leq 9RN^{\gamma_r}}\mathds{1}_{|\overline{x}^1_t-\overline{X}_t|\leq R}\right]=O\left(N^{-(1+\omega')}\right).
\end{equation}
Furthermore, defining $\widehat{\good(t)}$ similarly as $\good(t)$ as a set of initial configurations such that points (i)-(ii)-(iii) from Lemma~\ref{lem:abstract_tout_va_bien} hold for the processes $(\overline{Z}_t, \overline{z}^{1:\mathcal{N}}_t)_t$, we have
\begin{equation}\label{eq:mart_toutvabien_avec_bar_chez_bar}
\hat{\mathbb{E}}\left[\left(1-\mathds{1}_{\widehat{\good(t)}}\right)\mathds{1}_{|\overline{x}^1_t-\overline{X}_t|\leq \frac{3R}{2}}\right]=O\left(N^{-(1+\omega')}\right).
\end{equation}
\end{lemma}

Using the previous results, we turn now to the analysis of \eqref{eq:depart_calcul} and proceed in several steps.\\


\noindent\textbf{Decoupling the past.}  Let us first deal with the functions $g_i$ and show that 
\begin{align}
\left(\prod_{i=1}^{n-1}g_i(V_{-t_i})\right)&g_n(V_{t_n})\int_{t_n}^{t_{n+1}}\nabla f(V_t)\nabla\Phi(X_t-x^1_t)\mathds{1}_{\mathfrak{G}_N(t)}dt \nonumber\\
=&\left(\prod_{i=1}^{n-1}g_i(\overline{V}_{-t_i})\right)g_n(\overline{V}_{t_n})\int_{t_n}^{t_{n+1}}\nabla f(V_t)\nabla\Phi(X_t-x^1_t)\mathds{1}_{\mathfrak{G}_N(t)} \, dt
\label{eq: remplacement bar V gi} \\
&+O\left(\int_{t_n}^{t_{n+1}}  \frac{T_m(v^1_t,V_t)}{N}\mathds{1}_{|t-t_n|\leq 9R \, T_m(v^1_t,V_t)}\mathds{1}_{\mathfrak{G}_N(t)}dt\right), \nonumber
\end{align}
where the modified dynamics was introduced in \eqref{eq: dynamique barres} and remains well defined at negative times.

To prove \eqref{eq: remplacement bar V gi}, recall that for any configuration in $\mathfrak{G}_N(t)$ condition \eqref{eq: point iii} implies that the background particle $x^1$ and the tagged particle interact only at times $s$ such that 
$|t-s|\leq 9R \, T_m(v^1_t,V_t) \ll N$. In particular, since $t\geq t_n=O(N)$, they interact only at positive times.
By construction the functions $g_i(V_{-t_i})$ for $i \leq n-1$ are supported on negative times  so that $V_{-t_i}=\overline{V}_{-t_i}$ as 
 particle $1$ doesn't interact.
For $t$ close to $t_n$ the interaction can be controlled by \eqref{eq:abstract_approx} and since $g_n\in\mathcal{C}^1_{bb}$, we get for configurations in  $\mathfrak{G}_N(t)$
\begin{align*}
|t-t_n|\leq 9R \, T_m(v^1_t,V_t) \qquad \implies \qquad
g_n(V_{t_n})=g_n(\overline{V}_{t_n})+O\left( \frac{T_m(v^1_t,V_t)}{N}\right).
\end{align*}
If $|t-t_n|>9R \, T_m(v^1_t,V_t)$ we have $g_n(V_{t_n})=g_n(\overline{V}_{t_n})$ (as this implies that $t_n\leq \sigma^+_1$). Thus \eqref{eq: remplacement bar V gi} holds.

Using Lemma~\ref{lem:abstract_tout_va_bien} we deduce from \eqref{eq: remplacement bar V gi} that 
\begin{align}
\hat{\mathbb{E}}&\left[\left(\prod_{i=1}^{n-1}g_i(V_{-t_i})\right)g_n(V_{t_n})\int_{t_n}^{t_{n+1}}\nabla f(V_t)\nabla\Phi(X_t-x^1_t)dt\right]\nonumber\\
&\qquad=\hat{\mathbb{E}}\left[\left(\prod_{i=1}^{n-1}g_i(V_{-t_i})\right)g_n(V_{t_n})\int_{t_n}^{t_{n+1}}\nabla f(V_t)\nabla\Phi(X_t-x^1_t)\mathds{1}_{\mathfrak{G}_N(t)}dt\right]+O\left(\frac{t_{n+1}-t_n}{N^{1+\omega'}}\right)\nonumber\\
&\qquad=\hat{\mathbb{E}}\left[\left(\prod_{i=1}^{n-1}g_i(\overline{V}_{-t_i})\right)g_n(\overline{V}_{t_n})\int_{t_n}^{t_{n+1}}\nabla f(V_t)\nabla\Phi(X_t-x^1_t)\mathds{1}_{\mathfrak{G}_N(t)}dt\right]\label{eq:martingale_symmetry}\\
&\qquad\qquad+O\left(\frac{t_{n+1}-t_n}{N^{1+\omega'}}\right)+O\left(\hat{\mathbb{E}}\left[\int_{t_n}^{t_{n+1}}  \frac{T_m(v^1_t,V_t)}{N}\mathds{1}_{|t-t_n|\leq 9R \, T_m(v^1_t,V_t)}\mathds{1}_{\mathfrak{G}_N(t)}dt\right]\right) , \nonumber
\nonumber
\end{align}
Note that the error term in the  first equality is obtained by using  Fubini to exchange the expectation and the time  integral, and then \eqref{eq:proba_abstract_good}. 

\medskip


\noindent\textbf{Evaluating the interaction between $X_t$ and $x_t^1$.}  Using \eqref{eq:approx_trajectory_1} and Lemma \ref{lem:abstract_approx}, we will now take into account the mutual influence between the two particles in order to evaluate \eqref{eq:martingale_symmetry}. 
For this we write a Taylor expansion of the different terms in \eqref{eq:martingale_symmetry}. In all the calculations below, we assume that the configurations belong to $\mathfrak{G}_N(t)$.
Furthermore, all calculations are done while $|X_t-x^1_t|\leq R$ or, because of the Taylor expansion, $|\overline{X}_t-\overline{x}^1_t|\leq R$. Thanks to \eqref{eq:approx_trajectory_1} and  \eqref{eq:abstract_approx}.
Note that for configurations in $\mathfrak{G}_N(t)$ we have \begin{equation}\label{eq:borne_tps_inter_dans_G}
T_m(v^1_t,V_t)^2\leq N^{2\gamma_r}\ll N.
\end{equation}
Ensuring that $\overline{X}_t$ is close to $X_t$ and $\overline{x}^1_t$ to $x^1_t$, it is sufficient to do the calculations while $|X_t-x^1_t|\leq 3R/2$\footnote{Hence why we consider the extended interaction radius in \ref{lem:abstract_tout_va_bien}}. As a consequence, any time integral actually occurs on a time frame of size $O\left(T_m(v^1_t,V_t)\right)$ by Lemma~\ref{lem:abstract_tout_va_bien}.
Starting with $f$ and using both \eqref{eq:abstract_approx} and \eqref{eq:abstract_approx_V_prec}, we get 
\begin{align}
\nabla f(V_t)=&\nabla f(\overline{V}_t)+\text{Hess}f(\overline{V}_t)\left(V_t-\overline{V}_t\right)+O\left(\frac{T_m(v^1_t,V_t)^2}{N^2}\right)  \nonumber \\
=&\nabla f(\overline{V}_t)
- \frac{1}{N} \text{Hess}f(\overline{V}_t)
\int_{0}^t\nabla\Phi(\overline{X}_s-\overline{x}^1_s)ds
+O\left(\frac{T_m(v^1_t,V_t)^2}{N^2} + \cE_1 (N,v_t,V_t) \right).
\label{eq:DL_1_f}
\end{align}
Concerning the term $\nabla \Phi$, we have for configurations in $\mathfrak{G}_N(t)$ using  \eqref{eq:approx_trajectory_1}-\eqref{eq:abstract_approx}-\eqref{eq:abstract_approx_X_prec}
\begin{align}
\nabla\Phi(X_t-x^1_t)=&\nabla\Phi(\overline{X}_t - \overline x^1_t )
+\text{Hess}\Phi(\overline{X}_t -  \overline x^1_t )\left(\left(X_t-\overline{X}_t\right)-\left(x^1_t - \overline x^1_t \right)\right)+O\left(\frac{T_m(v^1_t,V_t)^4}{N^2}  \right) \nonumber\\
=&\nabla\Phi(\overline{X}_t - \bar x^1_t )
 - \frac{2}{N} \text{Hess}\Phi(\overline{X}_t -  \overline x^1_t ) \int_0^t ds \int_0^s du \, \nabla\Phi(\overline{X}_u-\overline{x}^1_u) \nonumber\\
&+O\left(\frac{T_m(v^1_t,V_t)^4}{N^2} +\mathcal{E}_2 (N,v_t,V_t) \right)\label{eq:DL_1_Phi}.
\end{align}
Our goal is now to approximate $\overline{X}_u$ by the straight line $\overline{X}_t-(t-u)\overline{V}_t$ in both \eqref{eq:DL_1_f} and \eqref{eq:DL_1_Phi}. From \eqref{eq:abstract_variation_vitesse}, the velocity of the tagged particle changes at most of $O\left(\sqrt{\frac{T_m(v^1_t,V_t)}{N}}N^{\alpha}\right)$ which, for configurations in  $\mathfrak{G}_N(t)$, is of order at most $N^{\frac{\gamma_r-1}{2}+\alpha}\ll 1$. 
Recalling $\sigma^+_1$ defined in \eqref{eq:def_tps_entree}, from \eqref{eq:abstract_approx}, we also have for $s$ such that $|s-t| \leq 9R  T_m(v^1_t,V_t)$
\begin{align*}
|\overline{V}_t-\overline{V_s}|=O\left(\sqrt{\frac{T_m(v^1_t,V_t)}{N}}N^{\alpha}+\frac{T_m(v^1_t,V_t)}{N}\right).
\end{align*}
As a consequence, for  $ |s - t | \leq   9R \, T_m(v^1_t,V_t)$ and in  $\mathfrak{G}_N(t)$, we may write 
\begin{align}
\overline{X}_s
&=  \overline{X}_t + \int_t^s du \overline V_u  du 
=  \overline{X}_t + (s-t)   \overline V_t 
+ \int_{t}^s (\overline V_u-\overline V_t) \nonumber\\
&=  \overline{X}_t - (t-s) \overline V_t 
+ O \left(\frac{T_m(v^1_t,V_t)^{3/2}}{N^{\frac{1}{2}-\alpha}}+\frac{ T_m(v^1_t,V_t)^2}{N}  \right),\label{eq:mart_compa_ligne_droite}
\end{align}
and thus  
\begin{align}
\nabla \Phi \Big( 
\overline{X}_s-\overline x^1_s \Big)
= \nabla \Phi \Big( \overline{X}_t-(t-s)\overline{V}_t  - \overline x^1_s \Big) 
+ \mathds{1}_{ t -  9R \, T_m(v^1_t,V_t) \leq s \leq t}
O \left( \frac{T_m(v^1_t,V_t)^{3/2}}{N^{\frac{1}{2}-\alpha}} +\frac{T_m(v^1_t,V_t)^2}{N}  \right).
\label{eq:particules_ecart_2}
\end{align}
Note that by Lemma~\ref{lem:abstract_tout_va_bien}, for configurations in $\good(t)$ and $s \leq t -  9R \, T_m(v^1_t,V_t)$, all the terms above are equal to 0. Indeed, the particles $\overline{X}_s,\overline x^1_s$ have not yet interacted and the relative velocity at time $t$ is large enough. As a consequence all the calculations can be done on a time frame of size $O\left(T_m(v^1_t,V_t)\right)$. 

Using \eqref{eq:particules_ecart_2} along with \eqref{eq:DL_1_f} and \eqref{eq:DL_1_Phi},  we finally get that for any configurations in $\good(t)$
\begin{align}
\nabla f(V_t)=&\nabla f(\overline{V}_t)
- \frac{1}{N} \text{Hess}f(\overline{V}_t)
\int_{0}^t\nabla\Phi(\overline{X}_t-(t-s)\overline{V}_t  - \overline x^1_s )ds\nonumber
\\
&+O\left(\frac{T_m(v^1_t,V_t)^2}{N^2}+\frac{T_m(v^1_t,V_t)^{5/2}}{N^{3/2-\alpha}} + \cE_1 (N,v_t,V_t)\right),
\label{eq:DL_2_f}
\end{align}
and 
\begin{align}
\nabla\Phi(X_t-x^1_t)=&\nabla\Phi(\overline{X}_t - \overline x^1_t )
 - \frac{2}{N} \text{Hess}\Phi(\overline{X}_t -   \overline x^1_t )\int_0^t\int_0^s \nabla\Phi(\overline{X}_t-(t-u)\overline{V}_t - \overline x^1_u )duds \nonumber\\
&+O\left(\frac{T_m(v^1_t,V_t)^4}{N^2}+\frac{T_m(v^1_t,V_t)^{7/2}}{N^{3/2-\alpha}}  +\mathcal{E}_2 (N,v_t,V_t) \right)\label{eq:DL_2_Phi}.
\end{align}
Plugging  \eqref{eq:DL_2_f} and \eqref{eq:DL_2_Phi} back into \eqref{eq:depart_calcul} and using \eqref{eq:martingale_symmetry}, we obtain 
\begin{align}
\mathbb{E}&\left[\left(\prod_{i=1}^{n-1}g_i(V_{-t_i})\right)g_n(V_{t_n})\left(f(V_{t_{n+1}})-f(V_{t_n})\right) \right]\nonumber\\
&\quad
= -\hat{\mathbb{E}} \left[ \left(\prod_{i=1}^{n-1}g_i( \overline V_{-t_i})\right)g_n( \overline V_{t_n})\int_{t_n}^{t_{n+1}}\nabla f(\overline{V}_t)\nabla\Phi(\overline{X}_t - \overline x^1_t )\mathds{1}_{\good(t)}dt\right]\label{eq:calcul_mart_int_1}\\
&\qquad
+\frac{2}{N}\hat{\mathbb{E}} \Bigg[\left( \prod_{i=1}^{n-1}g_i( \overline V_{-t_i})\right)g_n( \overline V_{t_n})\int_{t_n}^{t_{n+1}}\text{Hess}\Phi(\overline{X}_t -   \overline x^1_t )\left(\int_0^t\int_0^s \nabla\Phi(\overline{X}_t-(t-u)\overline{V}_t - \overline x^1_u )duds\right)\nonumber\\
&\hspace{7cm}\cdot\nabla f(\overline{V}_t) \mathds{1}_{\good(t)}dt\Bigg]\label{eq:calcul_mart_int_2}\\   
&\qquad
+\frac{1}{N}\hat{\mathbb{E}} \Bigg[ \left(\prod_{i=1}^{n-1}g_i( \overline V_{-t_i})\right)g_n( \overline V_{t_n})\int_{t_n}^{t_{n+1}}
\text{Hess}f(\overline{V}_t)\left(\int_0^t\nabla
 \Phi(\overline{X}_t-(t-s)\overline{V}_t  - \overline x^1_s )     ds\right)  \nonumber\\
&\hspace{7cm}\cdot\nabla
\Phi(\overline{X}_t - \overline x^1_t )\mathds{1}_{\good(t)}dt  
\Bigg] + error, \label{eq:calcul_mart_int_3}
\end{align}
with an error term given by 
\begin{align}
error = & O\left(\hat{\mathbb{E}}\left[ \int_{t_n}^{t_{n+1}}\left(\frac{T_m(v^1_t,V_t)}{N}\mathds{1}_{|t-t_n|\leq 9R \, T_m(v^1_t,V_t)}+\frac{T_m(v^1_t,V_t)^4}{N^2}+\frac{T_m(v^1_t,V_t)^{7/2}}{N^{3/2-\alpha}} \right.\right.\right.\nonumber\\ 
&\hspace{3cm}+\cE_1 (N,v^1_t,V_t)+\cE_2 (N,v^1_t,V_t)\Bigg)\mathds{1}_{\good(t)}dt\Bigg]\Bigg)+O\left(\frac{t_{n+1}-t_n}{N^{1+\omega'}}\right).\label{eq:error}
 \end{align}
We now deal with each term individually. We start by showing that the first order term \eqref{eq:calcul_mart_int_1} is in fact (almost) $0$ by integration over $x^1_0$ and the fact that $\nabla\Phi$ is of mean $0$. Then we deal with the  error term \eqref{eq:error} before turning our attention to the two most important contributions. The drift term will arise from \eqref{eq:calcul_mart_int_2} (see the term $\nabla f$), and the diffusion from \eqref{eq:calcul_mart_int_3} (see $\text{Hess}f$). \\
\vspace{0.3cm}


\noindent\textbf{First order term from \eqref{eq:calcul_mart_int_1}.} 
Our goal is to use an integration over $x^1_0$ and the fact that $\nabla \Phi$ is of mean $0$ to obtain that the first order term is negligible. To do so, we first need to get rid of the constraint $\mathds{1}_{\good(t)}$, as it also depends on $(x^1_0,v^1_0)$. We prove below, in the bullet point that can be skipped on first reading, that we can do so. Let us now show how we may conclude. We have
\begin{align}
\hat{\mathbb{E}} &\left[ \left(\prod_{i=1}^{n-1}g_i( \overline V_{-t_i})\right)g_n( \overline V_{t_n})\int_{t_n}^{t_{n+1}}\nabla f(\overline{V}_t)\nabla\Phi(\overline{X}_t - \overline x^1_t )\mathds{1}_{\good(t)}dt\right]\nonumber\\
&=\int_{t_n}^{t_{n+1}}\hat{\mathbb{E}} \left[ \left(\prod_{i=1}^{n-1}g_i( \overline V_{-t_i})\right)g_n( \overline V_{t_n})\nabla f(\overline{V}_t)\nabla\Phi(\overline{X}_t -\overline x^1_t )\right]dt+O\left(\frac{t_{n+1}-t_n}{N^{1+\omega'}}\right).\label{eq:on_a_retiré_la_contrainte}
\end{align}
In the expectation above, the trajectories $(\overline{Z}_t)_t$ and $(\overline{x}^1_t=x^1_0+tv^1_0)_t$ are decoupled. We therefore compute, using the grand canonical measure \eqref{eq:def_gce}
\begin{align*}
\hat{\mathbb{E}}& \left[ \left(\prod_{i=1}^{n-1}g_i( \overline V_{-t_i})\right)g_n( \overline V_{t_n})\nabla f(\overline{V}_t)\nabla\Phi(\overline{X}_t - x^1_0-tv^1_0)\right]\\
=&\frac{1}{\mathcal{Z}}\sum_{M=1}^\infty \frac{N^{M-1}}{(M-1)!}\int_{\mathbb{T}\times\mathbb{R}^d}dZ\int_{\mathbb{T}^{M-1}\times \mathbb{R}^{(M-1)d}}dz_{2:M} \rho_{M-1}(Z,z_{2:M}) \left(\prod_{i=1}^{n-1}g_i( \overline V_{-t_i})\right)g_n( \overline V_{t_n})\\
&\quad\quad\times\nabla f(\overline{V}_t)\cdot \left(\int_{\mathbb{R}^d}dv_1\gamma(v_1)\int_{\mathbb{T}}dx_1 e^{-\frac{\Phi(X-x_1)}{N}}\nabla\Phi(\overline{X}_t - x_1-tv_1)\right).
\end{align*}
The main obstacle to use the mean $0$ of $\nabla \Phi$ is the term $e^{-\frac{\Phi(X-x_1)}{N}}$ which implies that $x_1$ is not exactly uniformly distributed. We thus expand $e^{-\frac{\Phi(X-x_1)}{N}}=1+O\left(\frac{\mathds{1}_{x_1\in\mathcal{B}(X,R)}}{N}\right)$, and this last line yields
\begin{align*}
\int_{\mathbb{R}^d}dv_1\gamma(v_1)\int_{\mathbb{T}}&dx_1 e^{-\frac{\Phi(X-x_1)}{N}}\nabla\Phi(\overline{X}_t - x_1-tv_1)\\
=&O\left(\frac{1}{N}\int_{\mathbb{R}^d}dv_1\gamma(v_1)\int_{\mathbb{T}}dx_1 \mathds{1}_{x_1\in\mathcal{B}(X,R)}\mathds{1}_{x_1+tv_1\in\mathcal{B}(\overline{X}_t ,R)}\right)\\
=&O\left(\frac{1}{N}\int_{\mathbb{R}^d}dv_1\gamma(v_1)\int_{\mathbb{T}}dx_1 \mathds{1}_{x_1\in\mathcal{B}(X,R)}\mathds{1}_{v_1\in\mathcal{B}\left(\frac{\overline{X}_t-X}{t} ,\frac{2R}{t}\right)}\right)\\
=&O\left(\frac{1}{N}\left(\frac{1}{t^d}\wedge 1\right)\right).
\end{align*}
As a consequence
\begin{align}
\label{eq: negliger moyenne terme dominant}
\hat{\mathbb{E}} &\left[ \left(\prod_{i=1}^{n-1}g_i( \overline V_{-t_i})\right)g_n( \overline V_{t_n})\int_{t_n}^{t_{n+1}}\nabla f(\overline{V}_t)\nabla\Phi(\overline{X}_t - \overline x^1_t )\mathds{1}_{\good(t)}dt\right]\\
&=O\left(\frac{1}{N}\int_{t_n}^{t_{n+1}}\left(\frac{1}{t^d}\wedge 1\right)dt\right)+O\left(\frac{t_{n+1}-t_n}{N^{1+\omega}}\right)=O\left(\frac{1\wedge (t_{n+1}-t_n)}{N}\right)+O\left(\frac{t_{n+1}-t_n}{N^{1+\omega'}}\right). \nonumber
\end{align}
Since $t_{n+1}-t_n=O(N)$, it is therefore negligible.\\


$\bullet$ \underline{Getting rid of $\good(t)$} Let us now prove \eqref{eq:on_a_retiré_la_contrainte}. 
Note that configurations in the complement of $\good(t)$ do not prevent $\overline{X}_t$ and $\overline x^1_t$ to encounter. We need to estimate this possibility and 
the set $\widehat{\good(t)}$ has been introduced in Lemma~\ref{lem:mart_toutvabien_avec_bar} for this reason. 

In $\good(t)$, one necessarily has $|t-\sigma_1^+|\leq 9RN^{\gamma_r}$: the time $t$ must be in the first interaction time frame, as there are no recollision. Therefore \eqref{eq:on_a_retiré_la_contrainte} reads
\begin{align}
\label{eq: getting rid of good}
\hat{\mathbb{E}} &\left[ \left(\prod_{i=1}^{n-1}g_i( \overline V_{-t_i})\right)g_n( \overline V_{t_n})\int_{t_n}^{t_{n+1}}\nabla f(\overline{V}_t)\nabla\Phi(\overline{X}_t - \overline x^1_t )\mathds{1}_{\good(t)}dt\right]    \\
&=\hat{\mathbb{E}} \left[ \left(\prod_{i=1}^{n-1}g_i( \overline V_{-t_i})\right)g_n( \overline V_{t_n})\int_{t_n}^{t_{n+1}}\nabla f(\overline{V}_t)\nabla\Phi(\overline{X}_t - \overline x^1_t )\mathds{1}_{\good(t)}\mathds{1}_{|t-\sigma_1^+|\leq 9RN^{\gamma_r}}dt\right]
\nonumber \\
&=\hat{\mathbb{E}} \left[ \left(\prod_{i=1}^{n-1}g_i( \overline V_{-t_i})\right)g_n( \overline V_{t_n})\int_{t_n}^{t_{n+1}}\nabla f(\overline{V}_t)\nabla\Phi(\overline{X}_t - \overline x^1_t )\mathds{1}_{|t-\sigma_1^+|\leq 9RN^{\gamma_r}}dt\right]  \nonumber \\
&\quad-\hat{\mathbb{E}} \left[ \left(\prod_{i=1}^{n-1}g_i( \overline V_{-t_i})\right)g_n( \overline V_{t_n})\int_{t_n}^{t_{n+1}}\nabla f(\overline{V}_t)\nabla\Phi(\overline{X}_t - \overline x^1_t )\mathds{1}_{|t-\sigma_1^+|\leq 9RN^{\gamma_r}}\left(1-\mathds{1}_{\good(t)}\right)dt\right].
\nonumber
\end{align}
By \eqref{eq:mart_toutvabien_avec_bar}, the second term is  $O\left(\frac{t_{n+1}-t_n}{N^{1+\omega'}}\right)$. We now need to get rid of $\mathds{1}_{|t-\sigma_1^+|\leq 9RN^{\gamma_r}}$, again because $\sigma_1^+$ depends on $(x_1,v_1)$.  Inserting the set $\widehat{\good(t)}$, we get
\begin{align*}
\hat{\mathbb{E}}& \left[ \left(\prod_{i=1}^{n-1}g_i( \overline V_{-t_i})\right)g_n( \overline V_{t_n})\int_{t_n}^{t_{n+1}}\nabla f(\overline{V}_t)\nabla\Phi(\overline{X}_t - \overline x^1_t )\mathds{1}_{|t-\sigma_1^+|\leq 9RN^{\gamma_r}}dt\right]\\
&=\hat{\mathbb{E}} \left[ \left(\prod_{i=1}^{n-1}g_i( \overline V_{-t_i})\right)g_n( \overline V_{t_n})\int_{t_n}^{t_{n+1}}\nabla f(\overline{V}_t)\nabla\Phi(\overline{X}_t - \overline x^1_t )\mathds{1}_{|t-\sigma_1^+|\leq 9RN^{\gamma_r}}\left(\mathds{1}_{\widehat{\good(t)}}+\left(1-\mathds{1}_{\widehat{\good(t)}}\right)\right)dt\right].
\end{align*}
For the same reason, in $\widehat{\good(t)}$ and if $|\overline{x}^1_t-\overline{X}_t|\leq R$, one necessarily has $|t-\sigma_1^+|\leq 9RN^{\gamma_r}$: since there are no recollision with $\overline{X}$ for $\overline{x}^1$, the time $t$ must be in the first interaction time frame. We have
\begin{align}
\mathds{1}_{|t-\sigma_1^+|\leq 9RN^{\gamma_r}}\mathds{1}_{\widehat{\good(t)}}=&\mathds{1}_{\widehat{\good(t)}}\label{eq:mart_retirer_set_ordre_0_int_1}\\
\mathds{1}_{|\overline{x}^1_t-\overline{X}_t|\leq R}\mathds{1}_{|t-\sigma_1^+|\leq 9RN^{\gamma_r}}\left(1-\mathds{1}_{\widehat{\good(t)}}\right)\leq& \mathds{1}_{\overline{x}^1_t\in\mathcal{B}\left(\overline{X}_t,\frac{3}{2}R\right)}\left(1-\mathds{1}_{\widehat{\good(t)}}\right).\label{eq:mart_retirer_set_ordre_0_int_2}
\end{align}
Thus, this part is also $O\left(\frac{t_{n+1}-t_n}{N^{1+\omega'}}\right)$ by \eqref{eq:mart_toutvabien_avec_bar_chez_bar}. This gets rid of $\mathds{1}_{|t-\sigma_1^+|\leq 9RN^{\gamma_r}}$. 
Adding back ${1-\mathds{1}_{\widehat{\good(t)}}}$ to complete the remaining $\mathds{1}_{\widehat{\good(t)}}$ (again using \eqref{eq:mart_toutvabien_avec_bar_chez_bar}), we finally obtain \eqref{eq:on_a_retiré_la_contrainte}.\\
\vspace{0.3cm}


\noindent\textbf{The error \eqref{eq:error}.} 
Our goal here is to show that the error term \eqref{eq:error}, which involves terms of the form $T_m^k$ for $k \leq 6$, is of a smaller order than the main contributions. Recall the definition of $T_m$ given in Definition~\ref{def:inter_time_0}. 
The terms involving powers with $k\in\{1,\dots, 6\}$ are of the form
\begin{align*}
&\hat{\mathbb{E}}\left[\int_{t_n}^{t_{n+1}} \mathds{1}_{|v^1_t-V_t|\geq N^{-\gamma_r}}\mathds{1}_{x^1_t\in\mathcal{B}(X_t,3R/2)}|v^1_t-V_t|^{-k}\right]\\
&\quad=(t_{n+1}-t_n)\hat{\mathbb{E}}\left[\mathds{1}_{|v^1_0-V_0|\geq N^{-\gamma_r}}\mathds{1}_{x^1_0\in\mathcal{B}(X_0,3R/2)}|v^1_0-V_0|^{-k}\right],
\end{align*}
where we used the time invariance of the measure. 
Under the grand canonical measure  \eqref{eq:def_gce}, this reduces to a computation involving only two particles
\begin{align*}
&\hat{\mathbb{E}}\left[\int_{t_n}^{t_{n+1}} \mathds{1}_{|v^1_t-V_t|\geq N^{-\gamma_r}}\mathds{1}_{x^1_t\in\mathcal{B}(X_t,3R/2)}|v^1_t-V_t|^{-k}\right]\\
&\quad=(t_{n+1}-t_n)\int_{\mathcal{B}\left(0,3R/2\right)}dx e^{-\frac{\Phi(x)}{N}}\int_{\mathbb{R}^d}dV\gamma(V)\int_{\mathbb{R}^d}dv\gamma(v)\frac{\mathds{1}_{|v-V|\geq N^{-\gamma_r}}}{|v-V|^k}\\
&\quad=O\left((t_{n+1}-t_n)\int_{N^{-\gamma_r}/\sqrt{2}}^\infty\frac{e^{-\frac{w^2}{2}}w^{d-1}}{w^{k}}dw\right),
\end{align*}
where for this last line we used the fact that $V-v$ followed a Gaussian distribution.  In particular, we get (using $\delta>0$ to take into account the possible term $\ln(N)$)
\begin{equation}\label{eq:erreur_T_m_k_negli}
\hat{\mathbb{E}}\left[\int_{t_n}^{t_{n+1}} \mathds{1}_{|v^1_t-V_t|\geq N^{-\gamma_r}}\mathds{1}_{x^1_t\in\mathcal{B}(X_t,R)}|v^1_t-V_t|^{-k}\right]=O((t_{n+1}-t_n)N^{(k-d)_+\gamma_r+\delta \mathds{1}_{k=d}}).
\end{equation}
All the terms in \eqref{eq:error} with $k\in\{2,\dots, 6\}$ are divided by factors $N^a$ with $a\geq 3/2$.
This way, since $\gamma_r<\frac{1}{3}$ and $\mathds{1}_{\good(t)}\leq \mathds{1}_{|v^1_t-V_t|\geq N^{-\gamma_r}}\mathds{1}_{x^1_t\in\mathcal{B}(X_t,3R/2)}$, 
the contribution of these terms in  \eqref{eq:error} is $O\left( \frac{t_{n+1}-t_n}{N^{1+\tilde{\omega}}}\right)$, for some $\tilde{\omega}>0$, hence negligible for $t_{n+1}-t_n=O(N)$.
There only remains to bound the first term in \eqref{eq:error}
\begin{align*}
\hat{\mathbb{E}}\left[ \int_{t_n}^{t_{n+1}}\frac{T_m(v^1_t,V_t)}{N}\mathds{1}_{|t-t_n|\leq 9R \, T_m(v^1_t,V_t)}
\mathds{1}_{|v^1_t-V_t|\geq N^{-\gamma_r}}\mathds{1}_{x^1_t\in\mathcal{B}(X_t,3R/2)} dt\right]
=O\left(\frac{N^{ \gamma_r}\wedge (t_{n+1}-t_n)}{N}\right),
\end{align*}
where we simply used that $T_m(v^1_t,V_t) = O(N^{\gamma_r})$ and \eqref{eq:erreur_T_m_k_negli}. 
This completes the control of the error term \eqref{eq:error}.

\vspace{0.3cm}


\noindent\textbf{The drift from \eqref{eq:calcul_mart_int_2}.} First, let us get rid of $\mathds{1}_{\good(t)}$ in order to integrate without constraints on any variable. We justify in the bullet point below, in \eqref{eq:pk_on_retire_indicatrice}, why we may do so. Let us expand the expectation to separate the integration on $(x^1_0,v^1_0)$ from the rest in \eqref{eq:calcul_mart_int_2} and get
\begin{align}
\hat{\mathbb{E}}&\left[\left( \prod_{i=1}^{n-1}g_i( \overline V_{-t_i})\right)g_n( \overline V_{t_n})\int_{t_n}^{t_{n+1}}\nabla f(\overline{V}_t)\cdot \text{Hess}\Phi(\overline{X}_t-\overline{x}^1_t)\int_0^t\int_0^s \nabla\Phi(\overline{X}_t-(t-u)\overline{V}_t-\overline{x}^1_u)duds dt\right]\nonumber\\
&\quad=\frac{1}{\mathcal{Z}}\sum_{M=1}^\infty \frac{N^{M-1}}{(M-1)!}\int_{\mathbb{T}\times\mathbb{R}^d}dZ\int_{\mathbb{T}^M\times \mathbb{R}^{Md}}dz_{1:M} \, \rho_M(Z,z_{1:M}) \left(\prod_{i=1}^{n-1}g_i( \overline V_{-t_i})\right)g_n( \overline V_{t_n})\nonumber\\
&\quad\quad\times\int_{t_n}^{t_{n+1}}\nabla f(\overline{V}_t)\cdot \text{Hess}\Phi(\overline{X}_t-\overline{x}^1_t)\int_0^t\int_0^s \nabla\Phi(\overline{X}_t-(t-u)\overline{V}_t-\overline{x}^1_u)duds dt\nonumber\\
&\quad=\frac{1}{\mathcal{Z}}\sum_{M=1}^\infty \frac{N^{M-1}}{(M-1)!}\int_{\mathbb{T}\times\mathbb{R}^d}dZ\int_{\mathbb{T}^{M-1}\times \mathbb{R}^{(M-1)d}}dz_{2:M} \rho_{M-1}(Z,z_{2:M})\left(\prod_{i=1}^{n-1}g_i( \overline V_{-t_i})\right)g_n( \overline V_{t_n})\nonumber \\
&\hspace{7cm}\times\int_{t_n}^{t_{n+1}}\nabla f(\overline{V}_t)\cdot \tilde{\Lambda}(t,X,\overline{X}_t,\overline{V}_t)dt,\label{eq:calcul_int_drift_mart}
\end{align}
where we denote, insisting on the possible dependencies,
\begin{align*}
\tilde{\Lambda}(t,X,\overline{X}_t,\overline{V}_t)=&\int_{\mathbb{T}}e^{-\frac{\Phi(X-x_1)}{N}}\int_{\mathbb{R}^d}\gamma(v_1)\\
&\quad\text{Hess}\Phi(\overline{X}_t-x_1-tv_1)\int_0^t\int_0^s \nabla\Phi(\overline{X}_t-(t-u)\overline{V}_t-x_1-uv_1)dudsdv_1dx_1.
\end{align*}
We may in fact, using changes of variables and the symmetry of $\Phi$, rewrite $\tilde{\Lambda}$ as only a function of $\overline{V}_t$ and obtain the drift term.


\begin{lemma}\label{lem:reecrire_drift}
We have for all $t$
\begin{equation*}
\tilde{\Lambda}(t,X,\overline{X}_t,\overline{V}_t)=\left(\Lambda(\overline{V}_t)+O\left(\frac{1}{t^{d-2}}\wedge 1\right)\right)\left(1+O\left(\frac{1}{N}\right)\right),
\end{equation*}
where the drift coefficient $\Lambda$ was introduced in \eqref{def:drift}. 
\end{lemma}

The proof of Lemma~\ref{lem:reecrire_drift} is postponed to Appendix~\ref{sec:preuves_lemmes_mart}.
As we prove below in Lemma~\ref{lem:coeff_bornes}, the function $\Lambda$ is bounded. 
This way, the computation of  \eqref{eq:calcul_mart_int_2} yields 
\begin{align}
\hat{\mathbb{E}}&\left[\left( \prod_{i=1}^{n-1}g_i( \overline V_{-t_i})\right)g_n( \overline V_{t_n})\int_{t_n}^{t_{n+1}}\nabla f(\overline{V}_t)\cdot \text{Hess}\Phi(\overline{X}_t-\overline{x}^1_t)\int_0^t\int_0^s \nabla\Phi(\overline{X}_t-(t-u)\overline{V}_t-\overline{x}^1_u)duds dt\right]\nonumber\\
&\quad=\frac{1}{\mathcal{Z}}\sum_{M=1}^\infty \frac{N^{M-1}}{(M-1)!}\int_{\mathbb{T}\times \mathbb{R}^d}dZ\int_{\mathbb{T}^{M-1}\times \mathbb{R}^{(M-1)d}}dz_{2:M}\rho_{M-1}(Z,z_{2:M})\left( \prod_{i=1}^{n-1}g_i( \overline V_{-t_i})\right)g_n( \overline V_{t_n})\nonumber\\
&\hspace{4cm}\times\int_{t_n}^{t_{n+1}}\nabla f(\overline{V}_t)\cdot \Lambda(\overline{V}_t)dt+O\left(\frac{ t_{n+1}-t_n}{N}+(t_{n+1}-t_n)\wedge1\right)\nonumber\\
&\quad=\mathbb{E}\left[ \left(\prod_{i=1}^{n-1}g_i(V_{-t_i})\right)g_n(V_{t_n})\int_{t_n}^{t_{n+1}}\nabla f(V_t)\cdot \Lambda(V_t)dt\right]+O\left(\frac{ t_{n+1}-t_n}{N}+(t_{n+1}-t_n)\wedge1\right).\label{eq:final_drift}
\end{align}
This is the contribution yielding the drift term. Recall that it is multiplied by a factor $\frac{1}{N}$ in the computation of the martingale  \eqref{eq:calcul_mart_int_2}, which in particular implies that the error term is negligible.\\


$\bullet$ \underline{Getting rid of $\good(t)$}:  
We proceed as in \eqref{eq: getting rid of good}, however the test functions are more complicated and additional controls are needed.
Like previously, in $\good(t)$, we know that $|t-\sigma^+_1|\leq 9RN^{\gamma_r}$. Thus \eqref{eq:calcul_mart_int_2} is of the form
\begin{align*}
\eqref{eq:calcul_mart_int_2}
=  \frac{2}{N}\hat{\mathbb{E}} \Bigg[\tilde{g}(\overline{V}_{[-N;t_n]})\int_{t_n}^{t_{n+1}}\text{Hess}\Phi(\overline{X}_t -   \overline x^1_t )\tilde{f}(\overline{Z}_t,\overline{z}^1_t)\mathds{1}_{|t-\sigma^+_1|\leq 9RN^{\gamma_r}}\mathds{1}_{\good(t)}dt\Bigg],
\end{align*}
where we define 
\begin{align*}
\tilde{g}(V)=&g_n(V_{t_n})\prod_{i=1}^{n-1}g_i(V_{-t_i}),\quad\text{which is such that}\quad\|\tilde{g}\|_\infty\leq1,\\
\tilde{f}(\overline{Z}_t,\overline{z}^1_t)=&\nabla f(\overline{V}_t)\cdot \int_{0}^t\int_{0}^s\nabla\Phi\left(\overline{X}_t-\overline{x}^1_t-(t-u)(\overline{V}_t-\overline{v}^1_t)\right)duds.
\end{align*}
Notice that, if $\overline x^1_t=x^1_0+tv^1_0\in\mathcal{B}(\overline{X}_t,R)$, we have $\overline x^1_u=\overline x^1_t-(t-u)v^1_0\in\mathcal{B}(\overline{X}_t-(t-u)\overline{V}_t,R)$ only if $|t-u|\leq \frac{2R}{|\overline{V}_t-v^1_0|}$ (recall that the functions $u\mapsto \overline x^1_t-(t-u)v^1_0$ and $u\mapsto \overline{X}_t-(t-u)\overline{V}_t$ define two straight lines)\footnote{The periodic aspect of the torus can be neglected: the size of the torus is polynomial in $N$, we consider $t\leq N$ and, with high probability, $|\overline{V}_t-v^1_0|$ is smaller than any power of $N$.}. We therefore have
\begin{align*}
\left|\mathds{1}_{\overline x^1_t\in\mathcal{B}(\overline{X}_t,R)}\tilde{f}(\overline{Z}_t,\overline{z}^1_t)\right|=O\left(\frac{1}{|\overline{V}_t-v^1_0|^2}\right).
\end{align*}
This leads to 
\begin{align}
\eqref{eq:calcul_mart_int_2}
=&   \frac{2}{N}\hat{\mathbb{E}} \left[\tilde{g}(\overline{V}_{[-N;t_n]})\int_{t_n}^{t_{n+1}}\text{Hess}\Phi(\overline{X}_t -   \overline x^1_t )\tilde{f}(\overline{Z}_t,\overline{z}^1_t)\mathds{1}_{|t-\sigma^+_1|\leq 9RN^{\gamma_r}}dt\right]\label{eq:mart_int_retire_set_drift_1}\\
&+O\left(   \frac{1}{N} \int_{t_n}^{t_{n+1}}\hat{\mathbb{E}}\left[ \mathds{1}_{\overline x^1_t\in\mathcal{B}(\overline{X}_t,R)}\left(\frac{1}{|\overline{V}_t-v^1_0|}\right)^2\mathds{1}_{|t-\sigma^+_1|\leq 9RN^{\gamma_r}}(1- \mathds{1}_{\good(t)}) \right]dt\right).\label{eq:mart_int_retire_set_drift_2}
\end{align}
We have
\begin{align}
\eqref{eq:mart_int_retire_set_drift_2}
=O\left(  \frac{1}{N} \int_{t_n}^{t_{n+1}}\hat{\mathbb{E}}\left[ \mathds{1}_{\overline x^1_t\in\mathcal{B}(\overline{X}_t,R)}\mathds{1}_{|t-\sigma^+_1|\leq 9RN^{\gamma_r}}(1- \mathds{1}_{\good(t)})\right]^{\frac{p-1}{p}}\hat{\mathbb{E}}\left[\mathds{1}_{\overline x^1_t\in\mathcal{B}(\overline{X}_t,R)}\frac{1}{|\overline{V}_t-v^1_0|^{2p}}\right]^{\frac{1}{p}}dt\right),\label{eq:mart_int_retire_set_drift_3}
\end{align}
for any $p>1$. Since by \eqref{eq:mart_toutvabien_avec_bar} we have $\hat{\mathbb{E}}\left[ \mathds{1}_{\overline x^1_t\in\mathcal{B}(\overline{X}_t,R)}\mathds{1}_{|t-\sigma^+_1|\leq 9RN^{\gamma_r}}(1- \mathds{1}_{\good(t)})\right]^{\frac{p-1}{p}}=O\left(N^{-\frac{p-1}{p}(1+\omega')}\right)$, we only need to show that $\hat{\mathbb{E}}\left[\mathds{1}_{\overline x^1_t\in\mathcal{B}(\overline{X}_t,R)}\frac{1}{|\overline{V}_t-v^1_0|^{2p}}\right]^{\frac{1}{p}}=O(1)$. This can be done for $p=5/4$ for instance by direct computations, similarly to what has been done above to deal with \eqref{eq:error}, as $\overline{V}_t$ and $v^1_0$ are independent Gaussian variables\footnote{One can also use that $\int \gamma(v+V)\frac{1}{|v|^{5/2}}dv=O\left(\int \gamma(v)\frac{1}{|v|^{5/2}}dv\right)=O\left(1\right)$ uniformly in $V\in\bbR^d$.}. This takes care of \eqref{eq:mart_int_retire_set_drift_2}, and we now need to remove $\mathds{1}_{|t-\sigma^+_1|\leq 9RN^{\gamma_r}}$ from \eqref{eq:mart_int_retire_set_drift_1}. This can be done using $\widehat{\good(t)}$ defined in Lemma~\ref{lem:mart_toutvabien_avec_bar}. We use again \eqref{eq:mart_retirer_set_ordre_0_int_1}-\eqref{eq:mart_retirer_set_ordre_0_int_2} and get
\begin{align*}
\eqref{eq:mart_int_retire_set_drift_1}
=& \frac{2}{N} \hat{\mathbb{E}} \left[\tilde{g}(\overline{V}_{[-N;t_n]})\int_{t_n}^{t_{n+1}}\text{Hess}\Phi(\overline{X}_t -   \overline x^1_t )\tilde{f}(\overline{Z}_t,\overline{z}^1_t)\mathds{1}_{|t-\sigma^+_1|\leq 9RN^{\gamma_r}}\mathds{1}_{\widehat{\good(t)}}dt\right]\\
&+  \frac{2}{N}\hat{\mathbb{E}} \left[\tilde{g}(\overline{V}_{[-N;t_n]})\int_{t_n}^{t_{n+1}}\text{Hess}\Phi(\overline{X}_t -   \overline x^1_t )\tilde{f}(\overline{Z}_t,\overline{z}^1_t)\mathds{1}_{|t-\sigma^+_1|\leq 9RN^{\gamma_r}}\left(1-\mathds{1}_{\widehat{\good(t)}}\right)dt\right]\\
= &  \frac{2}{N}\hat{\mathbb{E}} \left[\tilde{g}(\overline{V}_{[-N;t_n]})\int_{t_n}^{t_{n+1}}\text{Hess}\Phi(\overline{X}_t -   \overline x^1_t )\tilde{f}(\overline{Z}_t,\overline{z}^1_t)\mathds{1}_{\widehat{\good(t)}}dt\right]\\
&+O\left(  \frac{1}{N} \int_{t_n}^{t_{n+1}}\hat{\mathbb{E}} \left[\mathds{1}_{\overline x^1_t\in\mathcal{B}(\overline{X}_t,3R/2)}\left(\frac{1}{|\overline{V}_t-v^1_0|}\right)^2\left(1-\mathds{1}_{\widehat{\good(t)}}\right)dt\right]\right)\\
=&  \frac{2}{N} \hat{\mathbb{E}} \left[\tilde{g}(\overline{V}_{[-N;t_n]})\int_{t_n}^{t_{n+1}}\text{Hess}\Phi(\overline{X}_t -   \overline x^1_t )\tilde{f}(\overline{Z}_t,\overline{z}^1_t)dt\right]\\
&+O\left(   \frac{1}{N} \int_{t_n}^{t_{n+1}}\hat{\mathbb{E}} \left[\mathds{1}_{\overline x^1_t\in\mathcal{B}(\overline{X}_t,3R/2)}\left(\frac{1}{|\overline{V}_t-v^1_0|}\right)^2\left(1-\mathds{1}_{\widehat{\good(t)}}\right)dt\right]\right)
\end{align*}
We again bound this last error by Hölder's inequality, in the same way \eqref{eq:mart_int_retire_set_drift_2} was bounded by \eqref{eq:mart_int_retire_set_drift_3}, and now use \eqref{eq:mart_toutvabien_avec_bar_chez_bar} to obtain $\hat{\mathbb{E}}\left[ \mathds{1}_{\overline x^1_t\in\mathcal{B}(\overline{X}_t,3R/2)}\left(1- \mathds{1}_{\widehat{\good(t)}}\right)\right]^{\frac{p-1}{p}}=O\left(N^{-\frac{p-1}{p}(1+\omega')}\right)$ for any $p$. 
We thus conclude
\begin{align}
\label{eq:pk_on_retire_indicatrice}
\eqref{eq:calcul_mart_int_2}
=&  \frac{2}{N} \hat{\mathbb{E}} \left[\left( \prod_{i=1}^{n-1}g_i( \overline V_{-t_i})\right)g_n( \overline V_{t_n})\int_{t_n}^{t_{n+1}}\nabla f(\overline{V}_t)\cdot \text{Hess}\Phi(\overline{X}_t -   \overline x^1_t )\right.\\
&\hspace{3cm}\left(\int_0^t\int_0^s \nabla\Phi(\overline{X}_t-(t-u)\overline{V}_t - \overline x^1_u )duds\right)\Bigg]
+ O\left(\frac{t_{n+1}-t_n}{N^{1 + \frac{p-1}{p}(1+\omega')}}\right). \nonumber
\end{align}

\vspace{0.3cm}


\noindent\textbf{The diffusion from \eqref{eq:calcul_mart_int_3}.} Proceeding as in \eqref{eq:pk_on_retire_indicatrice}, we can get rid of  the indicator function $\mathds{1}_{\good(t)}$. We now expand the expectation in \eqref{eq:calcul_mart_int_3} and get as in \eqref{eq:calcul_int_drift_mart}
\begin{align*}
\hat{\mathbb{E}}&\left[\left( \prod_{i=1}^{n-1}g_i( \overline V_{-t_i})\right)g_n( \overline V_{t_n})\int_{t_n}^{t_{n+1}}\text{Hess}f(\overline{V}_t)\left(\int_0^t\nabla \Phi(\overline{X}_t-(t-s)\overline{V}_t-\overline{x}^1_s)ds\right)\cdot\nabla\Phi(\overline{X}_t-\overline{x}^1_t)  dt\right]\\
&\quad=\frac{1}{\mathcal{Z}}\sum_{M=1}^\infty \frac{N^{M-1}}{(M-1)!}\int_{\mathbb{T}\times\mathbb{R}^d}dZ\int_{\mathbb{T}^{M-1}\times \mathbb{R}^{(M-1)d}}dz_{2:M} \rho_M(Z,z_{2:M}) \left(\prod_{i=1}^{n-1}g_i( \overline V_{-t_i})\right)g_n( \overline V_{t_n})\\
&\qquad\qquad\times\int_{t_n}^{t_{n+1}}\text{Hess}f(\overline{V}_t):\tilde{D}(t,X,\overline{X}_t,\overline{V}_t)dt,
\end{align*}
where for two matrices $A,B$ we denote $A:B=\sum_{i,j}A_{i,j}B_{i,j}$ and $\tilde{D}$ is given by the coefficients
\begin{align*}
\tilde{D}_{i,j}(t,X,\overline{X}_t,\overline{V}_t)=&\int_{\mathbb{T}}dx_1 e^{-\frac{\Phi(X-x_1)}{N}}\int_{\mathbb{R}^d}dv_1\gamma(v_1)\left(\int_0^t\partial_j \Phi(\overline{X}_t-(t-s)\overline{V}_t-x_1-sv_1)ds\right)\\
&\qquad\cdot\partial_i\Phi(\overline{X}_t-x_1-tv_1)
\end{align*}
Again, with similar computations as in Lemma~\ref{lem:reecrire_drift}, we may rewrite:


\begin{lemma}\label{lem:reecrire_diffusion}
We have for all $t$
\begin{equation}\label{eq:on_a_la_diffusion}
\tilde{D}_{i,j}(t,X,\overline{X}_t,\overline{V}_t)=\left(D_{i,j}(\overline{V}_t)+O\left(\frac{1}{t^{d-1}}\wedge 1\right)\right)\left(1+O\left(\frac{1}{N}\right)\right),
\end{equation}
where the diffusion coefficient $D$ was introduced in \eqref{def:diffusion}.
\end{lemma}

The proof of this lemma is postponed to Appendix~\ref{sec:preuves_lemmes_mart}. As we prove below in Lemma~\ref{lem:coeff_bornes}, the function $D_{i,j}$ is bounded. This way, the computation of \eqref{eq:calcul_mart_int_3} yields
\begin{align}
\eqref{eq:calcul_mart_int_3}&
=\frac{1}{\mathcal{Z}}\sum_{M=1}^\infty \frac{N^{M-1}}{(M-1)!}\int_{\mathbb{T}\times\mathbb{R}^d}dZ\int_{\mathbb{T}^{M-1}\times \mathbb{R}^{(M-1)d}}dz_{2:M} \rho_M(Z,z_{2:M}) \left(\prod_{i=1}^{n-1}g_i( \overline V_{-t_i})\right)g_n( \overline V_{t_n})\nonumber\\
&\qquad\qquad\times
 \frac{1}{N}\int_{t_n}^{t_{n+1}}\sum_{i,j}\partial_{i,j}f(\overline{V}_t)D_{i,j}(\overline{V}_t)dt+O\left(\frac{t_{n+1}-t_n}{N^2}+\frac{(t_{n+1}-t_n)\wedge1}{N} \right)\nonumber\\
&\quad=
 \frac{1}{N} \mathbb{E}\left[\left( \prod_{i=1}^{n-1}g_i( V_{-t_i})\right)g_n(V_{t_n})\int_{t_n}^{t_{n+1}}\text{Hess}f(V_t):D(V_t)dt\right]+O\left(\frac{t_{n+1}-t_n}{N^2}+\frac{(t_{n+1}-t_n)\wedge1}{N} \right) .
\label{eq:final_diffusion}
\end{align}
\vspace{0.3cm}

\noindent
\textbf{Conclusion of the proof of Proposition~\ref{prop:martingale_formulation}.}
Plugging \eqref{eq:final_drift} and \eqref{eq:final_diffusion}, as well as the results concerning the other terms, back into \eqref{eq:calcul_mart_int_1}-\eqref{eq:calcul_mart_int_3} yields that there exists $\omega>0$ such that
\begin{align*}
\mathbb{E}\left[\prod_{i=1}^ng_i(V_{\tau_iN})\left(f(V_{\tau_{n+1}N})-f(V_{\tau_nN})\right)\right]=&\frac{2}{N}\mathbb{E}\left[\prod_{i=1}^ng_i(V_{\tau_iN})\int_{\tau_nN}^{\tau_{n+1}N}\nabla f(V_t)\cdot \Lambda(V_t)dt\right]\\
&+\frac{1}{N}\mathbb{E}\left[\prod_{i=1}^ng_i(V_{\tau_iN})\int_{\tau_nN}^{\tau_{n+1}N}\text{Hess}f(V_t):D(V_t)dt\right]\\
&+O\left( \frac{t_{n+1}-t_n}{N^{1+\omega}}\right)+O\left(\frac{N^{\gamma_r}\wedge (t_{n+1}-t_n)}{N}\right).
\end{align*}
\end{proof}

%
%
%
%

\subsection{Study of the macroscopic coefficients}

Let us give some results concerning the limiting drift and diffusion coefficients.


\begin{lemma}\label{lem:prop_coeff}
Let $\Lambda:\mathbb{R}^d\mapsto \mathbb{R}^d$ and $D:\mathbb{R}^d\mapsto \mathbb{R}^{d\times d}$ be defined respectively in \eqref{def:drift} and \eqref{def:diffusion}. For all $V\in\mathbb{R}^d$, we have that
\begin{enumerate}[label=(\roman*)]
\item $\Lambda(V)$ and $D(V)$ are well-defined in dimensions $d\geq3$,
\item
$D(V)$ is a symmetric positive definite matrix. There therefore exists a unique positive definite and symmetric matrix $\Sigma:\mathbb{R}^d\mapsto \mathbb{R}^{d\times d}$ such that for all $V\in\mathbb{R}^d$ we have $D(V)=\Sigma(V)\Sigma(V)$,
\item  $\Lambda(V)=\nabla\cdot D(V)$, where the divergence is taken line by line, i.e. for all $i\in\{1,...,d\}$ we have $\Lambda_i (V)=\sum_{j=1}^d\partial_jD_{i,j}(V)$,
\item we can write $\Lambda (V)= -D(V)V$,
\item and finally, we  have
\begin{align}
D(V)=&\frac{\pi}{(2\pi)^{d}}\int_{\mathbb{R}^d}dv\int_{\mathbb{R}^d}dk \gamma(v+V)(k\otimes k)|\hat{\Phi}(k)|^2\delta_{k\cdot v},\\
\Lambda(V)=&\frac{\pi}{(2\pi)^{d}}\int_{\mathbb{R}^d}dv\int_{\mathbb{R}^d}dk \nabla\gamma(v+V)(k\otimes k)|\hat{\Phi}(k)|^2\delta_{k\cdot v}.
\end{align}
\end{enumerate}
\end{lemma}


\begin{remark}
Let us give a few remarks concerning the implications of Lemma~\ref{lem:prop_coeff}.

(iv) ensures that the standard Gaussian distribution is the invariant measure for the macroscopic SDE \eqref{eq:eds_limit}, which is coherent with the invariant measure for the microscopic process $(V_{t})_t$.

(v) states that the coefficients $\Lambda$ and $D$ we obtain are the same as those identified in  \cite[Theorem 4.2]{NVW22} (up to multiplicative universal constants). 
\end{remark}

The proof is postponed to Appendix~\ref{sec:tech_coeff}. We also have the following regularity results on the coefficients.


\begin{lemma}\label{lem:coeff_bornes}
Let $\Lambda:\mathbb{R}^d\mapsto \mathbb{R}^d$ and $D:\mathbb{R}^d\mapsto \mathbb{R}^{d\times d}$ be defined respectively in \eqref{def:drift} and \eqref{def:diffusion}. We have 
\begin{itemize}
\item $\Lambda \in L^\infty(\mathbb{R}^d,\mathbb{R}^d)$ and $D \in L^\infty(\mathbb{R}^d,\mathbb{R}^{d\times d})$,
\item there exists $C>0$ such that for all $V,V'\in\mathbb{R}^d$, we have
\begin{align*}
\left|\Lambda(V)-\Lambda(V')\right|\leq C|V-V'|,\qquad \left\|D(V)-D(V')\right\|\leq C|V-V'|.
\end{align*}
\end{itemize}
\end{lemma}

Likewise, the proof is postponed to Appendix~\ref{sec:tech_coeff}. From this last lemma, let us now conclude on the well-posedness of the SDE \eqref{eq:eds_limit} and the associated martingale problem.


\begin{lemma}\label{lem:existence_unicite_sde}
There is weak and strong existence, as well as weak and pathwise uniqueness, for the macroscopic SDE \eqref{eq:eds_limit}. In particular, there exists a unique solution to the martingale problem associated with \eqref{eq:eds_limit} with initial condition given by the standard Gaussian distribution, and this solution corresponds to the law of the process defined by \eqref{eq:eds_limit}.
\end{lemma}


\begin{proof}
Recall SDE \eqref{eq:eds_limit}
\begin{equation*}
d\mathcal{V}_\tau=2\Lambda(\mathcal{V}_\tau)d\tau+\sqrt{2}\Sigma(\mathcal{V}_\tau)dB_\tau,
\end{equation*}
where $B$ is a Brownian motion in $\mathbb{R}^d$ and $\Lambda, \Sigma$ are defined in Theorem~\ref{thm:main}.
\begin{description}
\item[$\bullet$ Weak existence:] By Lemma~\ref{lem:coeff_bornes}, $\Lambda$ and $D$ are bounded continuous functions. Let us show that this implies a similar property for $\Sigma$. First we have, for all $V,Z\in\mathbb{R}^d$ 
\begin{align*}
|\Sigma(V)Z|^2=D(V)Z\cdot Z\leq |D(V)Z\|Z|\leq \|D\|_\infty |Z|^2,
\end{align*}
which implies that $\|\Sigma\|_\infty\leq \|D\|_\infty^{1/2}$. Furthermore, we have for all $V,V'\in\mathbb{R}^d$
\begin{align*}
\|\Sigma(V)-\Sigma(V')\|\leq \sqrt{\|D(V)-D(V')\|}\leq C^{1/2}|V-V'|^{1/2}.
\end{align*}
Note that the first inequality above is a simple technical argument \cite{Phi87}.
The coefficients of the SDE are therefore bounded and continous, which implies existence of a solution to the martingale problem by Skorokhod (see \cite[Theorem 21.9]{Kal02}), and thus weak existence of a solution to the SDE by Stroock and Varadhan (see \cite[Theorem 21.7]{Kal02}).
\item[$\bullet$  Pathwise uniqueness:] Let us use \cite[Theorem 4]{WY71}. We have $\|\Sigma(V)-\Sigma(V')\|\leq \rho(|V-V'|)$ and $|\Lambda(V)-\Lambda(V')|\leq \overline{\rho}(|V-V'|)$ where $\rho:x\mapsto C^{1/2}\sqrt{x}$ and $\overline{\rho}:x\mapsto Cx$ satisfy
\begin{align*}
\int_0^\infty\frac{dx}{\frac{\rho(x)^2}{x}+\overline{\rho}(x)}=\infty,\qquad\text{ and }\qquad\frac{\rho(x)^2}{x}+\overline{\rho}(x)=C(1+x)\quad\text{ is concave.}
\end{align*}
By \cite[Theorem 4]{WY71}, there is therefore pathwise uniqueness for the SDE.
\item[$\bullet$ Uniqueness in law and strong existence:] By Watanabe and Yamada (see \cite[Lemma 21.17]{Kal02}), weak existence and pathwise uniqueness imply uniqueness in law and strong existence for the solutions of the SDE.
\item[$\bullet$ Martingale problem:] Again, by equivalence of the solutions to the martingale problem and weak solutions to the SDE (\cite[Theorem 21.7]{Kal02}), we obtain that there exists a unique solution to the martingale problem which coincides with the law of the SDE.
\end{description}
\end{proof}

%
%
%
%

\subsection{Tightness and conclusion}\label{sec:tight}

The goal of this section is conclude the proof of the main result. To do so, we rely on Kolmogorov's criterion, which can for instance be found in  \cite[Corollary 16.9]{Kal02}, and we prove, using Proposition~\ref{prop:martingale_formulation}, the following equicontinuity estimate


\begin{lemma}\label{lem:equi}
There exists $\epsilon>0$ such that for any $t,s\in[0,T]$, with $T\leq N$, we have
\begin{equation}\label{eq:equi_pow_4}
\mathbb{E}\left[\left|V_t-V_s\right|^4\right]=O\left(\left(\frac{|t-s|}{N}\right)^{1+\epsilon}\right).
\end{equation}
\end{lemma}

The proof of this lemma is deferred to Appendix~\ref{sec:proof_tight} 
and relies on Proposition~\ref{prop:martingale_formulation}.
We now have all the ingredients to prove the main theorem.


\begin{proof}[Proof of Theorem~\ref{thm:main}]
Using Lemma~\ref{lem:equi}, the sequence $(V_{\cdot N})_{N\geq 1}$ satisfies the assumptions  of Kolmogorov's criterion: 
\begin{itemize}
\item the sequence $(V_0)_N$ is trivially tight in $\mathbb{R}^d$ (since for all $N$ we have $V_0\sim g_0\gamma$),
\item there exists $c>0$ such that for all $\tau,\sigma\in[0,1]$
\begin{align*}
\sup_{N\geq 1}\mathbb{E}\left[|V_{\tau N}-V_{\sigma N}|^4\right]\leq c |\tau-\sigma|^{1+\epsilon}.
\end{align*}
\end{itemize}
The sequence of processes is therefore tight in $\mathcal{C}([0,1],\mathbb{R}^d)$. Any limit point of a converging subsequence has support in $\mathcal{C}([0,1],\mathbb{R}^d)$. Furthermore, for any test function $f\in\mathcal{C}^\infty_c(\mathbb{R}^d, \mathbb{R})$, any $n\geq1$, any $0\leq \tau_1<...<\tau_n<\tau_{n+1}\leq 1$ and any functions $g_i\in\mathcal{C}^1_{bb}(\mathbb{R}^d,\mathbb{R})$, the function
\begin{align*}
F:\ &\mathcal{C}([0,1],\mathbb{R}^d)\mapsto \mathbb{R},\\
&v\longrightarrow \prod_{i=1}^n g_i\left(v_{\tau_i }\right)\left(f(v_{\tau_{n+1}})-f(v_{\tau_{n}})-\int_{\tau_{n}}^{\tau_{n+1}} \mathcal{L}f(v_u)du\right),
\end{align*}
is bounded and continuous. By definition of convergence in law and Proposition~\ref{prop:martingale_formulation}, all limit points $\mathcal{V}$ must satisfy
\begin{align*}
\mathbb{E}&\left[\prod_{i=1}^n g_i(\mathcal{V}_{\tau_i})\left(f(\mathcal{V}_{\tau_{n+1}})-f(\mathcal{V}_{\tau_n})-\int_{\tau_n}^{\tau_{n+1}} \mathcal{L}f(\mathcal{V}_t)dt\right)\right]=0.
\end{align*}
By a straightforward approximation argument, the same equality holds for $g_i\in\mathcal{C}^0_b$ (instead of $g_i\in\mathcal{C}^1_{bb}$) and $f\in\mathcal{C}^2_c$. By \cite[Section 4.3]{EK86}, this characterizes the fact that any limit point is a solution to the martingale problem associated with \eqref{eq:eds_limit}, for which we have uniqueness by Lemma~\ref{lem:existence_unicite_sde}. Hence the result.
\end{proof}

%
%
%
%

\section{The Good and the Better set}
\label{sec: The Good and the Better set}

From now on, we give an upper bound on the probability of pathological sets. As a consequence, it is enough consider an initial velocity distribution such that $g_0=1$. The goal of this section is to formalize and prove a few natural results. Namely, we show 
\begin{itemize}
\item in Lemma~\ref{lem:max_part}  that the number of background particles interacting with the tagged particle is always of order at most $N$,
\item  in Lemma~\ref{lem:max_value_drift} that the drift on the tagged particle stays at all time of order $N^{-\frac{1}{2}}$ (up to a small correction $\delta$),
\item in Section~\ref{sec:max_inter} that we can control how long a background particle interacts with the tagged particle,
\item in Lemma~\ref{lem:max_sum_inter_time} that there are bounds on the empirical moments of the maximum interaction time,
\item finally in Proposition~\ref{prop:good_set_good} that the set of initial configurations which display pathological behaviors (i.e. contrary to the ones described above) has exponentially small probability  in $N$.
\end{itemize}

We gather the initial conditions satisfying these various results into a single set, called the \textit{good set} and denoted $\mathcal{G}_N$.


\begin{definition}[The good set $\mathcal{G}_N(\delta)$]
\label{def:hyp_good_set}
An initial configuration $(\mathcal{N}, X_0, V_0, x_{1:\mathcal{N}}, v_{1:\mathcal{N}})$ is said to belong to the set $\mathcal{G}_N(\delta)$ if $\mathcal{N}>1$ and if the resulting processes defined in \eqref{eq:def_tagged} and \eqref{eq:def_process_back} satisfy for all $t\in[-2N,2N]$
\begin{align}
&\mathds{1}_{|V_t|\geq N^{1/2}}+\sum_{i}\mathds{1}_{|v^i_{t}|\geq N^{1/2}}=0,\label{eq:G_N_0}\\
&\sum_{i}\mathds{1}_{x^i_{t}\in\mathcal{B}(X_{t},R)}\leq C_{int}N.\label{eq:G_N_1}
\end{align}
For all multi-indices $\mathcal{I}\in\{1,...,d\}^{|\mathcal{I}|}$ with $|\mathcal{I}|\in\{1,2,3,4\}$  and for $N$ large enough
\begin{equation}
\label{eq:borne_der_Phi} 
\left|\frac{1}{\sqrt{N}}\sum_{i}\partial_{\mathcal{I}}\Phi(X_t-x^i_t)\right|^2 \leq N^\delta .
\end{equation}
 In particular \eqref{eq:borne_der_Phi} implies
\begin{align}
\left\|\frac{1}{\sqrt{N}}\sum_{i}\text{Hess}\Phi(X_t-x^i_t)\right\|^2  \leq N^\delta,
\qquad 
\left|\frac{1}{\sqrt{N}}\sum_{i}\nabla\Phi(X_t-x^i_t)\right|^2\leq  N^{\delta}.
\label{eq:G_N_3}
\end{align}
The constants $C_{int}, C_{vel}$ depend only on $d$, $R$ and $\Phi$. Note that $C_{int}$ will allow us to control the number of interacting particles.

Furthermore, recalling $T_m(v,V)=\frac{1}{|v-V|}$ from Definition~\ref{def:inter_time_0}, we get 
\begin{align}
\forall k\in\{1,...,6\}&,\ \sum_{i}\left(T_m(v^i_t,V_t)\wedge N^{1/3}\right)^k \mathds{1}_{x^i_t\in\mathcal{B}(X_t,R)}=\left\{\begin{array}{ll}O\left( N^{1+\delta}\right)&\text{ if }k\leq 3\\O\left( N^{\frac{k}{3}+\delta}\right)&\text{ if }k> 3.\end{array}\right.,\label{eq:G_N_4}
\end{align}
Finally, for a given sequence $(M_N)_N\geq0$, we have 
\begin{align}
\left|\frac{1}{N}\sum_{i}\partial_{\mathcal{I}}\Phi(X_t-x^i_t)\mathds{1}_{|v^i_t-V_t|\geq M_N}du\right|=&O\left(\frac{1}{N^{\frac{1-\delta}{2}}}\right)\label{eq:G_N_5}.
\end{align}
\end{definition}


\begin{remark}
The cutoff \eqref{eq:G_N_0} at $N^{1/2}$ could be improved into a more classical cutoff at a power of $\log(N)$. 
\end{remark}

Throughout this document, for the sake of simplicity, we will often write 
\begin{align*}
(\mathcal{N}, Z_0,z_{1:\mathcal{N}})\in\mathcal{G}_N(\delta)\quad \implies\quad \forall s\in[-N,N], \ (\mathcal{N}, X_s,V_s,x^{1:\mathcal{N}}_s,v^{1:\mathcal{N}}_s)\in\mathcal{G}_N(\delta).
\end{align*}
In other words, we will consider, for simplicity, that the set $\mathcal{G}_N(\delta)$ is invariant under time shift. Note that $\mathcal{G}_N(\delta)$ a priori requires \eqref{eq:G_N_1}-\eqref{eq:G_N_4} to hold for all $t\in[-2N,2N]$.
However, we are only interested in $t\in[-N,N]$ and we will only ever shift time by $s\in[-N,N]$, hence why we will us this abuse of notations.

The goal of this section is  to prove the following proposition.


\begin{proposition}
\label{prop:good_set_good}
Let $\delta\in]0,1/4[$. We have $\mathbb{P}\left[\overline{\mathcal{G}_N(\delta)}\right]=O\left(e^{-N^{\delta/4}}\right).$
\end{proposition}


\begin{proof}
Let $\delta\in]0,\frac{1}{4}[$. The fact that \eqref{eq:G_N_0}, \eqref{eq:G_N_1}, \eqref{eq:borne_der_Phi},  \eqref{eq:G_N_4} and \eqref{eq:G_N_5} hold with probability greater than $1-e^{-N^{\delta/4}}$ are direct consequences of respectively Lemma~\ref{lem:lip_emp}, Lemma~\ref{lem:max_part}, Lemma~\ref{lem:max_value_drift}, Lemma~\ref{lem:max_sum_inter_time} and \eqref{eq:control_drift_avec_contrainte_sup} below. 
\end{proof}

In parallel, because we require better controls on the trajectories than those provided in $\mathcal{G}_N(\delta)$, let us define for $\alpha, \beta\in[0,1]$ the \textit{better set} $\mathcal{G}_N(\delta,\alpha, \beta)$:


\begin{definition}[The better set  $\mathcal{G}_N(\delta,\alpha, \beta)$]
\label{def:hyp_boot}
An initial configuration $(\mathcal{N}, X_0, V_0, x_{1:\mathcal{N}}, v_{1:\mathcal{N}})$ is said to belong to the set $\mathcal{G}_N(\delta,\alpha, \beta)$ if it belongs to the set $\mathcal{G}_N(\delta)$ and if the resulting processes defined in \eqref{eq:def_tagged} and \eqref{eq:def_process_back} satisfy the following: there exist $C_{vel}\geq 1$ such that, for all $s,t\in[-2N,2N]$ satisfying $|t-s|\leq N^\beta$, we have for all multi-indices $\mathcal{I}\in\{1,...,d\}^{|\mathcal{I}|}$ with $|\mathcal{I}|\in\{1,2,3\}$
\begin{equation}\label{eq: borne sqrt Phi} 
\left|\frac{1}{N}\int_s^t\sum_{i}\partial_{\mathcal{I}}\Phi(X_u-x^i_u)du\right|=O\left( \sqrt{\frac{|t-s|}{N}}N^{\alpha}\right).
\end{equation}
with in particular  
\begin{align}
\left\|\frac{1}{N}\int_s^t\sum_{i}\text{Hess}\Phi(X_u-x^i_u)du\right\|\leq \sqrt{\frac{|t-s|}{N}}N^{\alpha}
,& \quad 
\left|\frac{1}{N}\int_s^t\sum_{i}\nabla\Phi(X_u-x^i_u)du\right|\leq \sqrt{\frac{|t-s|}{N}}N^{\alpha}\label{eq:control_drift_hyp_alpha},\\
\text{and }\quad |V_t-V_s|=O\left(\sqrt{\frac{|t-s|}{N}}N^\alpha\right).\label{eq:control_diff_velocite_alpha}
\end{align}
Furthermore, for a given sequence $(M_N)_N\geq0$,  we have 
\begin{align}
\left|\frac{1}{N}\int_s^t\sum_{i}\partial_{\mathcal{I}}\Phi(X_u-x^i_u)\mathds{1}_{|v^i_u-V_u|\geq M_N}du\right|=&O\left(\sqrt{\frac{|t-s|}{N}}N^{\alpha}\right)\label{eq:control_boot_avec_contrainte_sup}.
\end{align}
\end{definition}

Proving an analog of Proposition~\ref{prop:good_set_good} for $\mathcal{G}_N(\delta,\alpha, \beta)$ actually requires more work, as well as a bootstrap argument. This is the topic of Proposition~\ref{prop:fin_bootstrap}. We carry out the main calculations latter, but since some of technical points are similar to those necessary for $\mathcal{G}_N(\delta)$, we prove in the lemmas of this section some estimates both in the case the initial configuration belongs to $\mathcal{G}_N(\delta)$ and in the case it belongs to $\mathcal{G}_N(\delta,\alpha, \beta)$ for some $\alpha, \beta\in[0,1]$.

%
%
%
%

\subsection{Exponential controls on the empirical measure}\label{sec:emp_measure}

For any bounded test function $(x, v,V) \mapsto F(x,v,V)$, we
define the empirical measure at time $t$ as 
\begin{equation}\label{eq:empirical_measure}
\pi_t (F) :=  \sum_i F \left( X_t-x^i_t,v^i_t, V_t \right) .
\end{equation} 
Up to a small correction, exponential moments under the equilibrium measure are the same as those of a Poisson point process with parameter $N$. 

Our goal is to show pointwise in time controls in Definition~\ref{def:hyp_good_set} that hold for all $t\in[-2N,2N]$, i.e. of the form $\mathbb{P}\left[\exists t\in[-2N,2N],\ \pi_t(F)\geq \lambda \right]=O(e^{-N^{\delta}})$. To do so, we rely on the fact that the system is at equilibrium and we prove two lemmas:
\begin{itemize}
\item To bound $\mathbb{P}\left[\pi_t(F)\geq \lambda \right]$ for any given $t\in\mathbb{R}$, we bound $\mathbb{E}\left[\exp(\pi_t(F))\right]$, which by time-invariance is equal to $\mathbb{E}\left[\exp(\pi_0(F))\right]$. This is Lemma~\ref{lem:emp_exp}.
\item We can then control, for any subdivision $(t_j)_j$ of $[-2N,2N]$, the probability $$\mathbb{P}\left[\exists t_j,\ \pi_{t_j}(F)\geq \lambda \right]\leq \sum_j\mathbb{P}\left[\pi_{t_j}(F)\geq \lambda \right].$$
\item It now remains to show that, provided two consecutive times $t_j<t_{j+1}$ are close enough to each other, that if $\pi_{t_j}(F)$ and $\pi_{t_{j+1}}(F)$ are bounded, then so is $\pi_{t}(F)$ for all $t\in[t_j,t_{j+1}]$. This comes from some continuity estimate on $t\mapsto\pi_t(F)$ proved in Lemma~\ref{lem:lip_emp}, which holds because we can bound the velocities of the particles.
\end{itemize}
We thus start with the control of the exponential moment of the empirical measure at any time $t\in\mathbb{R}$.


\begin{lemma}\label{lem:emp_exp}
For any bounded function $(x, v,V) \mapsto F(x,v,V)$ and all $t\in\mathbb{R}$, we get 
\begin{align}
\frac{1}{N}\ln \mathbb{E} \left[ \exp \left(  \pi_t (F) \right) \Big| V_t\right]=&\int_{\mathbb{T}\times\mathbb{R}^{d}}dxdv\gamma(v)e^{-\frac{\Phi(x)}{N}}\left(e^{F \left( x,v, V_t \right)}-1\right)
\label{eq:exp_moment_0} \\
=&\int_{\mathbb{T}\times\mathbb{R}^{d}}dxdv\gamma(v)F \left( x,v, V_t \right)\nonumber\\
&+O\left(\int_{\mathbb{T}\times\mathbb{R}^{d}}dxdv\gamma(v)e^{|F \left( x,v, V_t \right)|}\left[\frac{|\Phi(x)|}{N}|F \left( x,v, V_t \right)|+|F \left( x,v, V_t \right)|^2\right]\right)
\label{eq:exp_moment}.
\end{align} 
\end{lemma}


\begin{proof}
Because we consider the stationary distribution, we have for all $t\in\mathbb{R}$
\begin{align*}
\mathbb{E} \left[ \exp \left(  \pi_t (F) \right) \right]=\mathbb{E} \left[ \exp \left(  \pi_0 (F) \right) \right].
\end{align*}
Therefore
\begin{align*}
\mathbb{E} \left[ \exp \left(  \pi_t (F) \right) \right]=&\frac{1}{\mathcal{Z}}\sum_{M=0}^\infty\frac{N^M}{M!}\int_{\mathbb{T}\times\mathbb{R}^d}dZ\int_{\mathbb{T}^M\times\mathbb{R}^{Md}}dz_{1:M}\rho_M(Z,z_{1:M}) \exp \left(  \sum_{i=1}^M F \left( X-x_i,v_i, V \right) \right)\\
=&\frac{1}{\mathcal{Z}}\sum_{M=0}^\infty\frac{N^M}{M!}\int_{\mathbb{T}\times\mathbb{R}^d}dXdV\gamma(V)\prod_{i=1}^M\left(\int_{\mathbb{T}\times\mathbb{R}^{d}}dx_idv_i\gamma(v_i)e^{-\frac{\Phi(X-x_i)}{N}}e^{F \left( X-x_i,v_i, V \right)}\right)\\
=&\frac{1}{\mathcal{Z}}\sum_{M=0}^\infty\frac{N^M}{M!}\int_{\mathbb{T}\times\mathbb{R}^d}dXdV\gamma(V)\left(\int_{\mathbb{T}\times\mathbb{R}^{d}}dxdv\gamma(v)e^{-\frac{\Phi(x)}{N}}e^{F \left( x,v, V \right)}\right)^M\\
=&\frac{|\mathbb{T}|}{\mathcal{Z}}\int_{\mathbb{R}^d}dV\gamma(V)\exp\left(N\int_{\mathbb{T}\times\mathbb{R}^{d}}dxdv\gamma(v)e^{-\frac{\Phi(x)}{N}}e^{F \left( x,v, V \right)}\right).
\end{align*}
Recall $\mathcal{Z}=|\mathbb{T}|\exp\left(N\left(\int_{\mathbb{T}}dxe^{-\frac{\Phi(x)}{N}}\right)\right)$, which then directly implies \eqref{eq:exp_moment_0}. Since we have for all $X\in\mathbb{R}$
\begin{align*}
 \left|e^{X}-1\right|=O\left(|X|e^{|X|}\right), \qquad \left|e^{X}-1-X\right|=O\left(|X|^2e^{|X|}\right),
\end{align*}
we obtain
\begin{align*}
\int_{\mathbb{T}\times\mathbb{R}^{d}}dxdv\gamma(v)&e^{-\frac{\Phi(x)}{N}}\left(e^{F \left( x,v, V \right)}-1\right)\\
=&\int_{\mathbb{T}\times\mathbb{R}^{d}}dxdv\gamma(v)\left(e^{F \left( x,v, V \right)}-1\right)+O\left(\int_{\mathbb{T}\times\mathbb{R}^{d}}dxdv\gamma(v)\frac{|\Phi(x)|}{N}|F \left( x,v, V \right)|e^{|F \left( x,v, V \right)|}\right)\\
=&\int_{\mathbb{T}\times\mathbb{R}^{d}}dxdv\gamma(v)F \left( x,v, V \right)\\
&+O\left(\int_{\mathbb{T}\times\mathbb{R}^{d}}dxdv\gamma(v)e^{|F \left( x,v, V \right)|}\left[\frac{|\Phi(x)|}{N}|F \left( x,v, V \right)|+|F \left( x,v, V \right)|^2\right]\right).
\end{align*}
Hence the result.
\end{proof}

The following result shows that on very short timescales the process is continuous.


\begin{lemma}\label{lem:lip_emp}
For any $\ell$-Lipschitz function $(x,v,V) \mapsto F(x,v,V) $, with $\ell\leq N$, such that
\begin{align*}
|x|\geq 2R\implies F(x,v,V) =0.
\end{align*}
We get 
\begin{equation}\label{eq:time_increment}
\mathbb{P} \left[ \exists t \in [-2N,2N], \qquad 
 \sup_{s \in [t, t + N^{-10}]}  | \pi_t (F) - \pi_s (F) | \geq \frac{1}{N} \right] = O\left(e^{-N^{1/4}}\right).
\end{equation} 
Furthermore, 
\begin{equation}\label{eq:control_vitesse}
\mathbb{P} \left[ \exists t \in [-2N,2N], \qquad 
\mathds{1}_{|V_t|\geq N^{1/2}}+\sum_{i}\mathds{1}_{|v^i_{t}|\geq N^{1/2}}\geq 1 \right] = O\left(e^{-N^{1/4}}\right).
 \end{equation} 
\end{lemma}


\begin{proof}
As the domain $\mathbb{T} = [-N^\eta, N^\eta]^d$ is polynomially bounded, one can impose a cut-off on all the velocities of the particles. Once this cutoff is proved, we can show the continuity of the process on short timescales as we can control the change in velocities.\\

\noindent\textbf{Cutoff: fixed $t$ control.} Fix $t\in\mathbb{R}$ and let  us show
\begin{equation}\label{eq:borne_vitesse_fixed_time}
\mathbb{P} \left[ |V_t| \geq \frac{\sqrt{N}}{2} \right]  + \mathbb{P} \left[ \exists i, \quad |v^i_t| \geq \frac{\sqrt{N}}{2}\right]=O\left( e^{-\sqrt{N}}\right).
\end{equation} 
Recall that, for $Z$ distributed according to a Gaussian law, we have $\mathbb{P}\left[|Z|\geq \sqrt{N}/2\right]=O\left(\exp(-N/8)\right)$. By applying \eqref{eq:exp_moment_0} with $F(x,v,V) = \lambda \mathds{1}_{|v| \geq \sqrt{N}/2}$, we get
\begin{align*}
\frac{1}{N}\ln \mathbb{E} \left[ \exp \left( \sum_i \mathds{1}_{|v^i_t| \geq \sqrt{N}/2} \right) \right]=&\left(\int_{\mathds{T}}dxe^{-\frac{\Phi(x)}{N}}\right)\left(\int_{\mathbb{R}^d}dv\gamma(v)\left(e^{ \mathds{1}_{|v| \geq \sqrt{N}/2}}-1\right)\right)\\
=&O\left(|\mathds{T}|\int_{\mathbb{R}^d}dv\gamma(v)(e-1)\mathds{1}_{|v| \geq \sqrt{N}/2}\right)
=O\left(N^{d \eta}\exp\left(-\frac{N}{8}\right)\right).
\end{align*} 
Applying this and Markov's inequality with $\lambda = N^{1/2}$, we deduce
\begin{align*}
\mathbb{P} \left[ \exists i, \quad |v^i_t| \geq \frac{\sqrt{N}}{2} \right]
\leq \mathbb{P} \left[  \sum_i 1_{|v^i_t| \geq \frac{\sqrt{N}}{2}}  \geq 1 \right]\leq e^{-1}\mathbb{E} \left[ \exp \left( \sum_i 1_{|v_i| \geq \frac{\sqrt{N}}{2}} \right) \right]
=O\left(e^{-\sqrt{N}}\right).
\end{align*} 
This completes \eqref{eq:borne_vitesse_fixed_time}. Note that the suboptimal bound $O\left(e^{-\sqrt{N}}\right)$ is enough for our purposes.
\\

\noindent\textbf{Cutoff: on a small time interval.} The next step is to extend this result
to a small time interval. Let us now show that there exists $c>0$ such that 
\begin{align}
& \mathbb{P} \Big[ \exists s\in [0,N^{-10}], \quad   \exists i, \quad |v^i_s |\geq \sqrt{N} \Big] =O\left(e^{-\sqrt{N}}\right)
\label{eq:borne_vitesse_small} \\
& \mathbb{P} \Big[ \exists s\in [0,N^{-10}], \quad |\dot V_s | \geq c,\quad  |V_s | \geq \sqrt{N} \Big]  =O\left(e^{-\sqrt{N}}\right). 
\label{eq:borne_vitesse_tagged} 
\end{align} 
First of all, notice $\dot v_i$ is at most of order $1/N$ and thus if $v_i$ is initially less than $\sqrt{N}/2$, then it remains less than $\sqrt{N}$ on a small time interval (for $N$ large enough). Therefore, \eqref{eq:borne_vitesse_fixed_time} implies  \eqref{eq:borne_vitesse_small}.

Then, the change of velocity $\dot V_s$ depends on the number of interacting background particles. By \eqref{eq:exp_moment_0} applied to $F(x,v,V)=\mathds{1}_{|x|\leq 2R}$, we know that for all $C>0$
\begin{align*}
\mathbb{P}\left[\sum_{i=1}^\mathcal{N}\mathds{1}_{x^i_0\in\mathcal{B}(X_0,2R)}\geq CN\right]\leq e^{-CN}\exp\left(N(e-1)C_d(2R)^de^{\|\Phi\|_\infty}\right),
\end{align*}
and thus, for $C=2(e-1)C_d(2R)^de^{\|\Phi\|_\infty}$, 
\begin{equation}\label{eq:nb_part_temps_0}
\mathbb{P}\left[\sum_{i=1}^\mathcal{N}\mathds{1}_{x^i_0\in\mathcal{B}(X_0,2R)}\geq CN\right]=O\left(e^{-CN/2}\right).
\end{equation}
 In other words, the number of background particles in a neighborhood $\mathcal{B}(X_0, 2R)$ of the tagged particle is bounded by (a constant times) $N$ with high probability. While this ensures $\sum_{i=1}^\mathcal{N}\mathds{1}_{x^i_s\in\mathcal{B}(X_0,R)}$ is bounded, we in fact wish to show $\sum_{i=1}^\mathcal{N}\mathds{1}_{x^i_s\in\mathcal{B}(X_s,R)}$ remains bounded on a small time frame. Because all particles move, we use the controls on the velocities. Assume $|V_0|\leq \sqrt{N}/2$ (which, thanks to \eqref{eq:borne_vitesse_fixed_time}, holds with high probability) and consider a stopping time
\begin{align*}
s=\inf\left\{u\in[0,N^{-10}]\quad \text{ s.t. }\quad \sum_{i=1}^\mathcal{N}\mathds{1}_{x^i_u\in\mathcal{B}(X_u,R)}\geq 2CN \right\}.
\end{align*} 
For $u\in[0,s]$, $|\dot V_u | \leq 2C\|\nabla\Phi\|_\infty$ and thus (for $N$ large enough) $|V_u| < N^{1/2}$ and $|X_u-X_0|\leq N^{1/2-10}\ll 1$. Furthermore, by \eqref{eq:borne_vitesse_small}, for all $i$ we have $\sup_{[0,N^{-10}]}|v^i_s|\leq N^{1/2}$ and thus $|x^i_u-x^i_0|\ll 1$ with high probability. Both ensure that for all $u\leq s$, we have $x^i_u\in\mathcal{B}(X_u,R)$ only if $x^i_0\in\mathcal{B}(X_0,2R)$. Therefore, if $s<N^{-10}$, then 
\begin{align*}
\sum_{i=1}^\mathcal{N}\mathds{1}_{x^i_0\in\mathcal{B}(X_0,2R)}\geq \sum_{i=1}^\mathcal{N}\mathds{1}_{x^i_s\in\mathcal{B}(X_s,R)}\geq 2CN,
\end{align*}
and thus by \eqref{eq:borne_vitesse_small} and \eqref{eq:nb_part_temps_0} we obtain $\mathbb{P}\left[s<N^{-10}\right]=O\left(e^{-\sqrt{N}}\right)$, i.e.
\begin{equation}
\label{eq:nb_part_interval}
\mathbb{P}\left[\exists s\in [0,N^{-10}], \quad\sum_{i=1}^\mathcal{N} \mathds{1}_{x^i_s\in\mathcal{B}(X_s,R)}\geq 2CN \right]=O\left(e^{-\sqrt{N}}\right).
\end{equation}
Thus, with high probability, the number of interacting particles (i.e. background particles in the ball $\mathcal{B}(X_s, R)$) during the time interval $[0,N^{-10}]$ remains of order $N$. As a consequence,  $\dot V_s$ remains of order 1, and this implies
\eqref{eq:borne_vitesse_tagged}.\\

\noindent\textbf{Lipschitz continuity of the empirical measure.} Let $s,t\in[0, N^{-10}]$. By the controls \eqref{eq:borne_vitesse_small}, \eqref{eq:borne_vitesse_tagged} and \eqref{eq:nb_part_interval} on the particles' velocities and the number of interacting particles, we deduce that with high probability 
\begin{align*}
\left| \pi_s (F) - \pi_t (F) \right|=&\left|\sum_{i}\left(F \left( X_s-x^i_s,v^i_s, V_s \right)-F \left( X_t-x^i_t,v^i_t, V_t \right)\right)\right|\\
\leq &\ell \sum_{i}\left(\mathds{1}_{|x^i_s-X_s|\leq 2R}+\mathds{1}_{|x^i_t-X_t|\leq 2R}\right)\left(\left|X_s-X_t\right|+\left|V_s-V_t\right|+\left|x^i_s-x^i_t\right|+\left|v^i_s-v^i_t\right|\right)\\
=&O\left(\ell |t-s| N^{1+1/2}\right).
\end{align*}
 Since by assumption $\ell \leq N$ then
\begin{align*}
\mathbb{P} \left[ \sup_{s \in [0, N^{-10}]}  | \pi_s (F) - \pi_0 (F) | \geq \frac{1}{N^2} \right] = O\left(e^{-\sqrt{N}}\right).\\
\end{align*} 

\noindent\textbf{Subdividing the time interval.}
Finally \eqref{eq:time_increment} is derived by using the time invariance and subdividing $[-2N,2N]$ into small time intervals of size $N^{-10}$. Let us write at least once this argument.  Consider, for $j \leq 4N^{11}$, the times $t_j=-2N+j N^{-10}$. 
By time invariance,  all the events $\{\sup_{s\in[t_j,t_{j+1}]}|\pi_s(F)-\pi_{t_j}(F)|\geq 1/N^2\}$ have the same (small) probability.  Therefore
\begin{align}
\label{eq: adding time intervals}
\mathbb{P} &\left[ \exists t \in [-2N,2N], \quad 
 \sup_{s \in [t, t + N^{-10}]}  | \pi_t (F) - \pi_s (F) | \geq \frac{1}{N} \right] \\
 &\leq \sum_{j =0}^{4N^{11}-1} \mathbb{P} \left[ 
 \sup_{s \in [t_j, t_{j+1}]}  | \pi_{t_j} (F) - \pi_s (F) | \geq \frac{1}{N^2} \right] =O\left(N^{11}e^{-N^{1/2}}\right)=O\left(e^{-N^{1/4}}\right). \nonumber
\end{align}
Hence \eqref{eq:time_increment}. The control \eqref{eq:control_vitesse} is likewise a direct consequence of \eqref{eq:borne_vitesse_small} and \eqref{eq:borne_vitesse_tagged}.
\end{proof}

%
%
%
%

\subsection{Proving the various controls}

In this section we now use Lemmas~\ref{lem:emp_exp}~and~\ref{lem:lip_emp} to prove that the various controls of Definition~\ref{def:hyp_good_set} hold with high probability. First, the number of background particles interacting with the tagged particle remains of order $N$.


\begin{lemma}[On the maximum number of interacting particles]
\label{lem:max_part}
There exists $C>0$ such that for all $N\geq 1$
\begin{equation}\label{eq:proba_trop_part_traj}
\mathbb{P}\left[\exists t\in[-2N,2N],\ \sum_{i=1}^\mathcal{N}\mathds{1}_{x^i_{t}\in\mathcal{B}(X_{t},R)}\geq CN\right]=O\left(  e^{-N^{1/4}}\right).
\end{equation}
\end{lemma}
The proof is a direct consequence of \eqref{eq:nb_part_interval} by summing as in \eqref{eq: adding time intervals}. Next, we show that the empirical moment of a function of mean $0$ remains of order $\sqrt{N}$ up to a small correction.


\begin{lemma}\label{lem:max_value_drift}
For any bounded and Lipschitz continuous function $F:\mathbb{R}^d\mapsto \mathbb{R}$ with a compact support included in $\mathcal{B}(0,R)$, assuming that $F$ has mean $0$, for  all $N\in\mathbb{N}\setminus\{0,1\}$ and all $\delta\in]0,1/4[$
\begin{equation}\label{eq:control_drift}
\mathbb{P}\left[\exists t\in[-2N,2N],\ \left|\frac{1}{\sqrt{N}}\sum_{i=1}^\mathcal{N}F(X_t-x^i_t)\right|\geq N^{\delta/2}\right]=O\left(e^{-N^{\delta/4}}\right).
\end{equation}
In particular, statement \eqref{eq:borne_der_Phi} follows by choosing $F = \partial_{\mathcal{I}} \Phi$.

Furthermore for any $\delta_v\in[N^{-\frac{1}{d}},1]$ 
\begin{equation}\label{eq:control_drift_avec_contrainte}
\mathbb{P}\left[\exists t\in[-2N,2N],\ \left|\frac{1}{\sqrt{N\delta_v^d}}\sum_{i=1}^\mathcal{N}F(X_t-x^i_t)\mathds{1}_{|v^i_t-V_t|\leq \delta_v}\right|\geq N^{\delta/2}\right]=O\left(e^{-N^{\delta/4}}\right),
\end{equation}
as well as for any $\delta_v\in[0,1]$ 
\begin{equation}\label{eq:control_drift_avec_contrainte_sup}
\mathbb{P}\left[\exists t\in[-2N,2N],\ \left|\frac{1}{\sqrt{N}}\sum_{i=1}^\mathcal{N}F(X_t-x^i_t)\mathds{1}_{|v^i_t-V_t|\geq \delta_v}\right|\geq N^{\delta/2}\right]=O\left(e^{-N^{\delta/4}}\right).
\end{equation}
\end{lemma}


\begin{proof}
By Lemma~\ref{lem:emp_exp}, since $F$ is of mean $0$, there exists $C>0$ such that for all $t\in[-2N,2N]$
\begin{align*}
\mathbb{P}\left[\frac{1}{\sqrt{N}}\pi_{t}(F)\geq \frac{1}{2}N^{\delta/2}\right]\leq Ce^{-\frac{1}{2}N^{\delta/2}},\quad \mathbb{P}\left[\frac{1}{\sqrt{N}}\pi_{t}(-F)\geq \frac{1}{2}N^{\delta/2}\right]\leq Ce^{-\frac{1}{2}N^{\delta/2}}.
\end{align*}
Like previously, by subdividing $[-2N,2N]$ into subintervals of size $N^{-10}$ and using Lemma~\ref{lem:lip_emp}, we conclude on \eqref{eq:control_drift}. 

Consider now a function $h:\mathbb{R}^+\mapsto \mathbb{R}$ such that $h(v)=1$ for $v\leq \delta_v$, $h(v)=0$ for $v\geq \delta_v+1/N$, and $h$ is in between a linear interpolation. This implies that $h$ is Lipschitz continuous, of Lipschitz constant of order $N$. We define $G(x,v,V):= \frac{1}{\sqrt{N\delta_v^d}}F(x)h(|v-V|)$ (which is therefore Lipschitz continuous, of Lipschitz constant of order $\delta_v^{-\frac{d}{2}}\sqrt{N}\leq N$). By Lemma~\ref{lem:emp_exp}
\begin{align}
\frac{1}{N}\ln \mathbb{E} \left[ \exp \left(  \pi_t (G) \right) \Big| V_t\right]=&\int_{\mathbb{T}}\int_{\mathbb{R}^{d}}dxdv\gamma(v)G \left( x,v, V_t \right)\nonumber\\
&+O\left(e^{\frac{\|F\|_\infty}{\sqrt{N\delta_v^d}}}\left[\frac{\|\Phi\|_\infty\|F\|_\infty}{N^{3/2}\delta_v^{d/2}}+\frac{\|F\|_\infty^2}{N\delta_v^d}\right]\int_{\mathbb{R}^d}dv\gamma(v)\mathds{1}_{|v-V_t|\leq 2\delta_v}\right)\nonumber\\
=&O\left(e^{\|F\|_\infty}\left[\frac{\|\Phi\|_\infty\|F\|_\infty}{N^{3/2}\delta_v^{d/2}}+\frac{\|F\|_\infty^2}{N\delta_v^d}\right]\delta_v^d\right)\label{eq:int_calcul_control_drift}\\
=&O\left(\frac{1}{N}e^{\|F\|_\infty}\left[\|\Phi\|_\infty\|F\|_\infty+\|F\|_\infty^2\right]\right)\nonumber,
\end{align}
where we used $\delta_v\geq N^{-1/d}$ in order to bound $\frac{\|F\|_\infty}{\sqrt{N\delta_v^d}}\leq \|F\|_\infty$ (and the right-hand side is uniform in $V_t$). Again, by subdividing $[-2N,2N]$ into small intervals of size $N^{-10}$ and using Lemma~\ref{lem:lip_emp} (since $G$ is Lipschitz continuous of Lipschitz constant smaller than $N$), we obtain 
\begin{align}
\bbP\left[\exists t\in[-2N,2N],\quad \left|\frac{1}{\sqrt{N\delta_v^d}}\sum_iF(X_t-x^i_t)h(|v^i_t-V_t|)\right|\geq \frac{N^{\delta/2}}{2}\right]=O\left(e^{-N^{\delta/4}}\right).\label{eq:pi_t_G}
\end{align}

We now wish to replace $h$ by $\mathds{1}_{|v|\leq \delta_v}$ in the control above. Consider a function $\tilde{h}:\mathbb{R}^+\mapsto \mathbb{R}$ such that $\tilde{h}(v)=1$ for $v\in [\delta_v,\delta_v+1/N]$, $\tilde{h}(v)=0$ for $v\leq \delta_v-1/N$ and $v\geq \delta_v+2/N$, and $\tilde{h}$ is in between a linear interpolation. This implies that $\tilde{h}$ is Lipschitz continuous, of Lipschitz constant of order $N$. We define $\tilde{G}(x,v,V):= \frac{1}{\sqrt{N\delta_v^d}}|F(x)|\tilde{h}(|v-V|)$ (which is therefore Lipschitz continuous, of Lipschitz constant of order $\delta_v^{-\frac{d}{2}}\sqrt{N}\leq N$). This way
\begin{align}
\left|\frac{1}{\sqrt{N\delta_v^d}}\sum_iF(X_t-x^i_t)\mathds{1}_{|v^i_t-V_t|\leq \delta_v}-\frac{1}{\sqrt{N\delta_v^d}}\sum_iF(X_t-x^i_t)h(|v^i_t-V_t|)\right|\leq \pi_t(\tilde{G}).\label{eq:pi_t_G_tilde_0}
\end{align}
By Lemma~\ref{lem:emp_exp}, since $\tilde{h}(v)\leq \mathds{1}_{v\in[\delta_v-1/N, \delta_v+2/N]}$
\begin{align}
\frac{1}{N}\ln \bbE\left[\exp\pi_t(\tilde{G})\big |V_t\right]=O\left(\left(e^{\frac{\|F\|_\infty}{\sqrt{N\delta_v^d}}}-1\right)e^{\|\Phi\|_\infty}\int dv\gamma(v)\mathds{1}_{|v-V_t|\in[\delta_v-1/N, \delta_v+2/N]}\right).\label{eq:pi_t_G_tilde_1}
\end{align}
Note that 
\begin{align*}
\int dv\gamma(v)\mathds{1}_{|v-V_t|\in[\delta_v-1/N, \delta_v+2/N]}=&O\left( \int dv\mathds{1}_{|v|\in[\delta_v-1/N, \delta_v+2/N]}\right)
=O\left(\frac{\delta_v^{d-1}}{N}\right).
\end{align*}
Plugging this estimate back in \eqref{eq:pi_t_G_tilde_1}, and using $e^{\frac{\|F\|_\infty}{\sqrt{N\delta_v^d}}}-1\leq \frac{\|F\|_\infty}{\sqrt{N\delta_v^d}}e^{\|F\|_\infty}$, we get, uniformly in $V_t$
\begin{align*}
\frac{1}{N}\ln \bbE\left[\pi_t(\tilde{G})\big |V_t\right]=O\left(e^{\|F\|_\infty}\|F\|_\infty\frac{\delta_v^{\frac{d}{2}-1}}{N^{3/2}}\right)=O\left(\frac{e^{\|F\|_\infty}\|F\|_\infty}{N^{3/2}}\right),
\end{align*}
which yields for all $t\in\bbR$
\begin{align*}
\bbP\left[\pi_t(\tilde{G})\geq \frac{1}{4}N^{\delta/2}\right]=O\left(e^{-N^{\delta/2}/4}\right).
\end{align*}
Since $\tilde{G}$ is Lipschitz continuous of Lipschitz constant smaller than $N$, we may now use Lemma~\ref{lem:lip_emp} and obtain like previously that
\begin{align}
\bbP\left[\exists t\in[-2N,2N],\ \pi_t(\tilde{G})\geq \frac{1}{2}N^{\delta/2}\right]=O\left(e^{-N^{\delta/4}}\right).\label{eq:pi_t_G_tilde_2}
\end{align}
Combining \eqref{eq:pi_t_G}, \eqref{eq:pi_t_G_tilde_0} and \eqref{eq:pi_t_G_tilde_2}, we conclude on \eqref{eq:control_drift}. Finally, \eqref{eq:control_drift_avec_contrainte_sup} is obtained similarly, without the possibility to integrate on the velocity as in \eqref{eq:int_calcul_control_drift}.
\end{proof}

Finally, because we show in Section~\ref{sec:max_inter} below that $T_m$ given by \eqref{eq:def_inter_time 0} is an upper bound on the interaction time of a single background particle, we also give here some estimates on its empirical moments.


\begin{lemma}[Sum of maximum interaction times]\label{lem:max_sum_inter_time}
Let $k\in\{1,...,6\}$, $\delta>0$ and set
\begin{align}
\label{eq: def CN}
C(N)=N^\delta\left(N^{\frac{k-3}{3}}\mathds{1}_{k>3}+1\right).
\end{align}
We have, for all $N\in\mathbb{N}\setminus\{0\}$ 
\begin{align}
\mathbb{P}\left[\exists t\in[-2N,2N],\quad \sum_{i}\left(T_m(v^i_t,V_t)\wedge N^{\frac{1}{3}}\right)^k\mathds{1}_{x^i_t\in\mathcal{B}(X_t,R)}\geq 3N^{1+\delta} C(N)\right]=O\left(e^{-N^{\delta/2}}\right).\label{eq:proba_sum_max_time_traj}
\end{align}
\end{lemma}


\begin{proof}
Consider a function $h:\mathbb{R}^+\mapsto \mathbb{R}$ such that $h(x)=1$ for $x\leq R$, $h(x)=0$ for $x\geq 2R$, and $h$ is a linear interpolation in between. Likewise, we consider $\tilde{h}:\mathbb{R}^+\mapsto \mathbb{R}$ defined by $\tilde{h}(v):=\left(\frac{1}{v}\wedge N^{1/3}\right)^k$. This way, $\tilde{h}$ is continuous of Lipschitz constant of order $N^{\frac{k+1}{3}}$. Recalling $C(N)$ in \eqref{eq: def CN}, we set
\begin{align*}
G(x,v,V):=\frac{1}{C(N)N}h(|x|)\tilde{h}(|v-V|),
\end{align*}
which is continuous of Lipschitz constant of order $N^{\frac{1}{3}-\delta}$ and satisfies for all $(x,v,V)$
\begin{align*}
1\geq \frac{N^{\frac{k-3}{3}}}{C(N)}\geq G(x,v,V)\geq \frac{1}{C(N)N}\left(T_m(v,V)\wedge N^{\frac{1}{3}}\right)^k\mathds{1}_{x\in\mathcal{B}(X,R)}.
\end{align*}
Let us now prove that for all $t\in[-2N,2N]$
\begin{equation}\label{eq:int_mes_emp_tps_inter}
\ln\mathbb{E}\left[\exp(\pi_t(G))\Big| V_t\right]=O\left(1\right).
\end{equation}
We have
\begin{align*}
\int_{\mathds{T}\times \mathbb{R}^d}dxdv\gamma(v)G(x,v,V_t)=\frac{1}{C(N)N}O\left(\int_{\mathbb{R}^d}dv\gamma(v)\left(T_m(v,V_t)\wedge N^{\frac{1}{3}}\right)^k\right).
\end{align*}
We have
\begin{align}
\int_{\mathbb{R}^{d}} dv \gamma(v)&\left(T_m(v,V_t)\wedge N^{1/3}\right)^k 
\leq \int_{\mathbb{R}^{d}} dv \frac{\gamma(v)}{|v-V_t|^k}\mathds{1}_{|v-V|\geq N^{-1/3}}+\int_{\mathbb{R}^{d}} dv \gamma(v) N^{k/3}\mathds{1}_{|v-V_t|\leq N^{-1/3}}.
\label{eq: Tm moment borne}
\end{align}
First we bound the last term by
\begin{align*}
\int_{\mathbb{R}^{d}} dv \gamma(v+V)N^{k/3}\mathds{1}_{|v|\leq N^{-1/3}}=O\left(N^{\frac{k-d}{3}}\right).
\end{align*}
The first term in \eqref{eq: Tm moment borne} is estimated by the following technical result: 
\begin{align}
\label{lem:technique_retirer_V_moment_inverse}
\forall V \in\mathbb{R}^d, \qquad 
\int_{\mathbb{R}^{d}} dv \gamma(v+V) \frac{1}{|v|^k}\mathds{1}_{|v|\geq N^{-1/3}}=O_k\left( \int_{\mathbb{R}^{d}} dv \gamma(v) \frac{1}{|v|^k}\mathds{1}_{|v|\geq N^{-1/3}}\right).
\end{align}
Then, we have
\begin{align*}
\int_{\mathbb{R}^{d}} dv \gamma(v) \frac{\mathds{1}_{|v|\geq N^{-1/3}}}{|v|^k}=&O\left(\int_{0}^\infty e^{-\frac{x^2}{2}}\frac{x^{d-1}}{x^k}\mathds{1}_{|x|\geq N^{-1/3}}\right)\\
=&O\left(N^{\frac{k-d}{3}}\mathds{1}_{k>d}+\ln(N)\mathds{1}_{k=d}+\mathds{1}_{k<d}\right).
\end{align*}
Therefore
\begin{align*}
\int_{\mathbb{R}^{d}} dv \gamma(v)&\left(T_m(v,V)\wedge N^{1/3}\right)^k=O\left(N^{\frac{k-d}{3}}\mathds{1}_{k>d}+\ln(N)\mathds{1}_{k=d}+\mathds{1}_{k<d}\right).
\end{align*}
We thus finally obtain
\begin{align*}
\int_{\mathds{T}\times \mathbb{R}^d}dxdv\gamma(v)G(x,v,V)=O\left(\frac{N^{\frac{k-d}{3}}\mathds{1}_{k>d}+\ln(N)\mathds{1}_{k=d}+\mathds{1}_{k<d}}{C(N)N}\right)=O\left(\frac{1}{N}\right).
\end{align*}
By Lemma~\ref{lem:emp_exp}, we obtain \eqref{eq:int_mes_emp_tps_inter}. Again, by subdividing $[-2N,2N]$ into small intervals of size $N^{-10}$ and using Lemma~\ref{lem:lip_emp}, we get the result.
\end{proof}

%
%
%
%

\subsection{On the maximum interaction time}
\label{sec:max_inter}

Assume given the initial velocity $V_0$ of the tagged particle, and $v_0^1$ of the background particle $i=1$. We wish to bound the time particle $1$ stays in the interaction sphere $\mathcal{B}(X_t,R)$ (or rather, to be "sure", in the sphere $\mathcal{B}(X_t,2R)$). Lemma~\ref{lem:max_inter_time} below gives an almost sure bound on the maximum interaction time as a function of $|V_0-v_0^1|$, and provides a lower bound for the time of re-entry both in the good set (see Definition \ref{def:hyp_good_set}) and the better set (see Definition \ref{def:hyp_boot}).


\begin{lemma}\label{lem:max_inter_time}
Let $s\in[-N,N]$, and assume that the particle $i=1$ is in $\mathcal{B}(X_s,R)$, i.e. $|X_s-x^1_s|\leq R$.

\noindent 
{\bf In the good set.} Consider a set of initial conditions in $\mathcal{G}_N(\delta)$. If $|V_s-v_s^1|\geq N^{-\frac{1}{4}+\frac{\delta}{2}}$, then
\begin{align}
\forall t\quad \text{ s.t. }\quad \frac{6R}{|v^1_s-V_s|} \leq |t-s|
 \leq |v_s^1-V_s| \, N^{\frac{1}{2} -\delta},
\qquad  |x^1_t-X_t|\geq 2R.
\label{eq: exit/return times good set}
\end{align}

\noindent 
{\bf In the better set.} Consider a set of initial conditions in $\mathcal{G}_N(\delta,\alpha, \beta)$ and assume $\beta\geq \frac{1-2\alpha}{3}+\delta$. If ${|V_s-v^1_s|\geq N^{-\frac{1-2\alpha}{3}+\delta}}$, then
\begin{align}
\forall t\  \text{ s.t. }\  \frac{6R}{|v^1_s-V_s|} \leq |t-s|\leq \frac{1}{5}\min\left(N^{1-2\alpha}|v^1_s-V_s|^2,N^{\frac{1+\beta}{2}-\alpha}|v^1_s-V_s|\right),\quad  |x^1_t-X_t|\geq 2R.
\label{eq: exit/return times better set}
\end{align}
\end{lemma}
This lemma and the idea of its proof are illustrated in Figure~\ref{fig:inter_time}. In particular, notice that Lemma~\ref{lem:max_inter_time} ensures that, maybe under some restrictions on $|v^1_s-V_s|$ or provided the observation window is not too large, the interaction time is bounded by $T_m$ given in Definition~\ref{def:inter_time_0}. 


\begin{proposition}
Let $t,s\in[-N,N]$ such that $|t-s|\leq N^{\frac{1-2\delta}{4}}$ and consider a set of initial conditions in $\mathcal{G}_N(\delta)$. Define
\begin{align}
\label{eq: def sigma i}
    \sigma_i=\inf\{u\geq -N,\quad x^i_u\in\mathcal{B}(X_u,R)\}.
\end{align}
Assuming $x^i_s\in\mathcal{B}(X_s,R)$ and $(s-\sigma_i)_+\leq N^{\frac{1-2\delta}{4}}$ (i.e. we consider a time $s$ in the first interaction time frame), we have
\begin{align*}
\mathds{1}_{x^i_t\in\mathcal{B}(X_t,R)}(t-\sigma_i)_+=O\left(T_m(v^i_s,V_s)\wedge N^{\frac{1-2\delta}{4}}\right).
\end{align*}
Note that this control holds for any choice of $s,t$ satisfying the above conditions. In other words, the maximum interaction time is bounded by $T_m(v^1_s,V_s)$, for any $s$ belonging to the interaction time frame\footnote{We write the result with respect to the first interaction time, but the same would hold for the second, third, etc. We just require $s$ and $t$ to belong to the same time frame.}.

Likewise, let $t,s\in[-N,N]$ such that $|t-s|\leq N^{\frac{1-2\alpha}{3}-\delta}$, consider a set of initial conditions in $\mathcal{G}_N(\delta,\alpha, \beta)$ and assume $x^i_s\in\mathcal{B}(X_s,R)$ and $(s-\sigma_i)_+\leq N^{\frac{1-2\alpha}{3}-\delta}$. We have
\begin{equation}\label{eq: interaction time better}
\mathds{1}_{x^i_t\in\mathcal{B}(X_t,R)}(t-\sigma_i)_+=O\left(T_m(v^i_s,V_s)\wedge N^{\frac{1-2\alpha}{3}-\delta}\right).
\end{equation}
\end{proposition}


\begin{proof}
This is a direct consequence of Lemma~\ref{lem:max_inter_time}
\end{proof}


\begin{figure}
	\centering
	\includegraphics[width=0.7\linewidth]{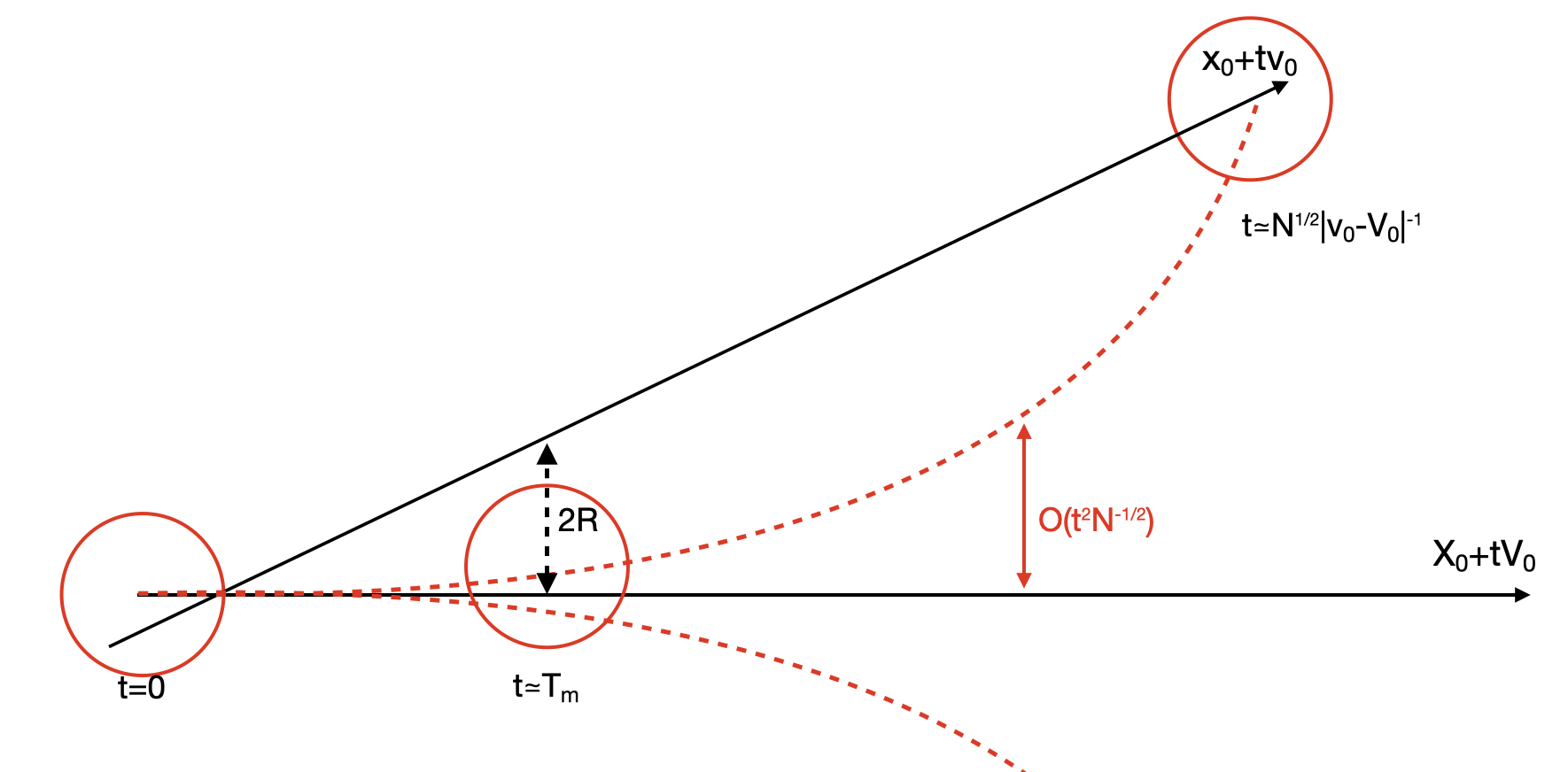}
	\caption{{\small Proof of the maximum interaction time and minimum recollision time. Assuming $x^1_0\in\mathcal{B}(X_0,R)$, we compare $X_t$ to $X_0+tV_0$ and $x^1_t$ to $x^1_0+tv^1_0$. In the case of $\mathcal{G}_N(\delta)$, the possible values for $X_t$ must be in between the two red dashed lines, which represent the bound on the error. We are interested in the time at which, despite the error, the comparison with straight lines ensures that the background particle exits the interaction radius, and the time at which the error becomes so big that the tagged particle may re-interact with particle $1$. For an initial difference of velocities which is too small, the known error does not allow us to ensure that the difference of positions becomes larger than $2R$.}}
	\label{fig:inter_time}
\end{figure}


\begin{proof}[Proof of Lemma~\ref{lem:max_inter_time}]
Using the time-invariance and time-reversibility of the process, it is sufficient to prove the result for $s=0$, and to consider times $t\geq0$. In both cases, we compare $X_t-x^1_t$ to $t(V_0-v_0^1)$, control the error on their difference, and consider the first time this error becomes so big that this approximation no longer ensures that the background particle has exited the radius of interaction of the tagged particle.  

\medskip

\noindent
{\bf In the good set $\mathcal{G}_N(\delta)$.} Let us start by considering initial configurations in $\mathcal{G}_N(\delta)$. We know that
\begin{align*}
\left|v_t^1-v_0^1\right|\leq \frac{\|\nabla \Phi\|_\infty t}{N}
\quad \text{ and } \quad 
\left|V_t-V_0\right| \leq  \frac{ t}{N^{\frac{1-\delta}{2}}}.
\end{align*}
For $N$ large enough to ensure  $1\geq \|\nabla\Phi\|_\infty N^{-\frac{1+\delta}{2}}$, we get
\begin{align*}
\left|(X_t-x^1_t)-(X_0-x^1_0)-\int_0^t(V_0-v_0^1)ds\right|\leq \int_0^t\left(\left|V_s-V_0\right|+\left|v_s^1-v_0^1\right|\right)ds\leq \frac{  t^2}{N^{\frac{1-\delta}{2}}},
\end{align*}
and thus
\begin{align*}
\left|(X_0-x^1_0 )+t(V_0-v_0^1)\right| -\frac{  t^2}{N^{\frac{1-\delta}{2}}} \leq |X_t-x^1_t|\leq\left|(X_0-x^1_0 )+t(V_0-v_0^1)\right|+\frac{ t^2}{N^{\frac{1-\delta}{2}}},
\end{align*}
since  $x^1_0$ belongs to $\mathcal{B}(X_0,R)$
\begin{equation}\label{eq:control_diff_pos}
-R+t|V_0-v_0^1| -\frac{  t^2}{N^{\frac{1-\delta}{2}}} \leq |X_t-x^1_t|\leq R+ t|V_0-v_0^1|+\frac{ t^2}{N^{\frac{1-\delta}{2}}}.
\end{equation}
We are now looking for all the values of  $t$ such that $ |X_t-x^1_t|\geq 2R$. It is therefore sufficient to have
\begin{align}
\label{eq: polynome ordre 2}
-R+t|V_0-v_0^1| -\frac{  t^2}{N^{\frac{1-\delta}{2}}} \geq 2R.
\end{align}
Note that, since for $N$ sufficiently large we have $N^{\frac{1-\delta}{2}}|V_0-v_0^1|^2\geq N^{\frac{\delta}{2}}\geq 12R$, the inequality is satisfied for $t\in[T_-,T_+]$ 
where the roots of the polynomial \eqref{eq: polynome ordre 2} are
\begin{align*}
T_\pm=&\frac{1}{2}N^{\frac{1-\delta}{2}}|V_0-v_0^1|\left(1\pm\sqrt{1-\frac{12R }{N^{\frac{1-\delta}{2}}|V_0-v_0^1|^2}}\right).
\end{align*} 
Since for $x\in[0,1]$ we have $\sqrt{1-x}\geq 1-x$, we get
\begin{equation}
\label{eq:tps_inter_minus}
T_- 
\leq \frac{6R}{|V_0-v_0^1|}
\quad \text{and} \quad 
T_+\geq \frac{|V_0-v_0^1|}{2 }N^{\frac{1-\delta}{2}}
\geq |V_0-v_0^1| \, N^{\frac{1}{2} -\delta}
.
\end{equation}
Hence we obtain the first result.

\medskip

\noindent
{\bf In the better set $\mathcal{G}_N(\delta,\alpha, \beta)$.}
We deal separately with the cases $t\leq N^\beta$ and $t\geq N^\beta$.

 For $t\leq N^\beta$ and configurations in $\mathcal{G}_N(\delta,\alpha, \beta)$, the velocity is well controlled (see Definition \ref{def:hyp_boot})
\begin{equation}
\label{eq:tps_inter_control_vitesse}
t\leq N^\beta, \qquad
\left|V_t-V_0\right| \leq  \sqrt{\frac{t}{N}} N^\alpha.
\end{equation}
Proceeding as \eqref{eq:control_diff_pos} we get
\begin{equation}\label{eq:control_diff_pos_boot}
-R+t|V_0-v_0^1| -\frac{  t^{\frac{3}{2}}N^\alpha}{N^{\frac{1}{2}}} \leq |X_t-x^1_t|\leq R+ t|V_0-v_0^1|+\frac{ t^{\frac{3}{2}}N^\alpha}{N^{\frac{1}{2}}}.
\end{equation}
We are looking for $t \leq N^\beta$ such that $ |X_t-x^1_t|\geq 2R$. It is therefore sufficient to have
\begin{equation}
\label{eq:cdt_dehors}
-R+t|V_0-v_0^1| -\frac{  t^{\frac{3}{2}}}{N^{\frac{1}{2}-\alpha}} \geq 2R.
\end{equation} 
Since $\beta\geq \frac{1-2\alpha}{3}+\delta$ and ${|V_0 -v^1_0|\geq N^{-\frac{1-2\alpha}{3}+\delta}}$,
we easily obtain that, denoting
\begin{align}
T_-=\frac{6R}{|V_0-v_0^1|}  \ll N^\beta
\qquad\text{ and }\qquad T_{+,1}=\frac{N^{1-2\alpha}|V_0-v_0^1|^2}{5}  \wedge N^\beta,
\label{eq: time estimate T1}
\end{align}
we have  (for $N$ large enough) $0<T_-<T_{+,1}$ and a simple function study ensures that (again, for $N$ large enough) for all $t\in[T_-,T_{+,1}]$, \eqref{eq:cdt_dehors} is satisfied.

\medskip 

For $t\geq N^\beta$, the control \eqref{eq:tps_inter_control_vitesse} is no longer valid, but it can  be applied on subintervals of size $N^\beta$
\begin{align*}
\left|V_t-V_0\right| \leq & \Big|V_t-V_{\left\lfloor\frac{t}{N^\beta}\right\rfloor N^\beta}\Big|+...+\left|V_{N^\beta}-V_0\right|
\leq 
\Big( \left\lfloor\frac{t}{N^\beta}\right\rfloor + 1 \Big) 
\sqrt{\frac{N^\beta}{N}} N^\alpha
\leq 2\frac{t N^\alpha}{N^{\frac{1+\beta}{2}}}.
\end{align*}
 Thus for $t\geq N^\beta$
\begin{align*}
\left|(X_t-x^1_t)-(X_0-x^1_0)-\int_0^t(V_0-v_0^1)ds\right|\leq& 
 \int_0^t \left(\left|V_s-V_0\right|+\left|v_s^1-v_0^1\right|\right)ds
\leq  \frac{5}{2}\frac{  t^{2}}{N^{\frac{1+\beta}{2}}}N^\alpha.
\end{align*}
This way
\begin{equation}\label{eq:cdt_dehors_2}
-R+t|V_0-v_0^1| - \frac{5}{2}\frac{  t^{2}N^\alpha}{N^{\frac{1+\beta}{2}}}\leq |X_t-x^1_t|\leq R+ t|V_0-v_0^1|+ \frac{5}{2}\frac{ t^{2}N^\alpha}{N^{\frac{1+\beta}{2}}}.
\end{equation}
Note that this is a control similar to \eqref{eq:control_diff_pos} thus we get  
\begin{align}
T_-=\frac{6R}{|V_0-v_0^1|}\qquad\text{ and }\qquad T_{+,2}=\frac{N^{\frac{1+\beta}{2}-\alpha}|V_0-v_0^1|}{5},
\label{eq: time estimate T2}
\end{align}
with $0<T_-<T_{+,2}$ (for $N$ large enough)  and a simple function study ensures that (again for $N$ large enough) for all $t\in[T_-,T_{+,2}]$, we ensure that $|X_t-x^1_t|\geq 2R$.\\

To conclude  \eqref{eq: exit/return times better set}, we need to combine the estimates  
\eqref{eq: time estimate T1}, \eqref{eq: time estimate T2}.
First note that, for $|V_0-v_0^1|<N^{-\frac{1-\beta}{2}+\alpha}$, we have $T_{+,1}\leq N^\beta$. Hence $T_{+,1}$ is a valid recollision time satisfying the a priori assumption that $t\leq N^\beta$. For $|V_0-v_0^1|\geq  N^{-\frac{1-\beta}{2}+\alpha}$ we have $T_{+,2}\leq T_{+,1}$. Hence \eqref{eq: exit/return times better set} follows.
\end{proof}

%
%
%
%

\section{Some useful controls}
\label{sec: Some useful controls}

%
%
%
%

\subsection{Influence of a single particle}\label{sec:influence_part}

We are going to derive the controls on the influence of one background particle on the tagged particle stated in Lemma \ref{lem:abstract_approx}. Given an initial data $\mathcal{N}, Z_0,z_{1:\mathcal{N}}$, recall the definition of the auxiliary process $\overline{Z}_t$ given in \eqref{eq: dynamique barres}, which consists in removing the influence of the background particle $1$. We can write
\begin{equation}\label{eq:lien_tagged_process}
\overline{X}_\cdot(Z_0,z_{1:\mathcal{N}})=X_\cdot(Z_0,z_{2:\mathcal{N}}),\quad \overline{V}_\cdot(Z_0,z_{1:\mathcal{N}})=V_\cdot(Z_0,z_{2:\mathcal{N}}).
\end{equation}
We first study the stability of the microscopic dynamics under the perturbation \eqref{eq: dynamique barres} for initial configuration in the better set. 
Analogous estimates will be stated in Lemma \ref{lem:influence_une_particule} for initial data in the good set.


\begin{lemma}[Better set estimates] 
\label{Lem: Better set estimates}
Given $\alpha \in ]0,1/2[$ and $\delta >0$ a small parameter, choose 
\begin{equation}
\label{eq: def gamma beta}
\gamma=\frac{1-2\alpha}{3}-\delta
\quad  \text{and} \quad 
\beta \geq \gamma + \delta.
\end{equation}
Recall $\sigma^+_1$ is the first positive interaction time \eqref{eq:def_tps_entree}.
For a configuration $(\mathcal{N},Z,z_{1:\mathcal{N}})$ in the better set $\mathcal{G}_N(\delta,\alpha, \beta)$ (see  Definition \ref{def:hyp_boot}),
the distance between the trajectories  $(Z_t)_t$ and $(\overline{Z}_t)_t$ defined in \eqref{eq:def_tagged}-\eqref{eq: dynamique barres} can be controlled as follows  for all  $t\in[\sigma^+_1,\sigma^+_1+N^{\gamma}],$ 
\begin{align} 
&
\left|X_t-\overline{X}_t\right|+(t-\sigma^+_1)_+ \left|V_t-\overline{V}_t\right| =O\left(\frac{(t-\sigma^+_1)_+^2}{N}\right) .
\label{eq:est_diff_ordre_1_boot}
\end{align}
One has also
\begin{align} 
&\left|X_t-\overline{X}_t+\frac{1}{N}\int_{0}^t\int_0^s\nabla\Phi(\overline{X}_u-\overline{x}^1_u)duds\right|+(t-\sigma^+_1)_+ \left|V_t-\overline{V}_t+\frac{1}{N}\int_{0}^t\nabla\Phi(\overline{X}_s-\overline{x}^1_s)ds\right|\nonumber\\
&\hspace{3cm}=O\left(\frac{(t-\sigma^+_1)_+^4}{N^{\frac{3-\delta}{2}}}+\frac{(t-\sigma^+_1)_+^6}{N^{2}}\right),
\label{eq:est_diff_ordre_2_boot}\\
&\left|X_t-\overline{X}_t+\frac{1}{N}\int_{0}^t\int_0^s\nabla\Phi(X_u-x^1_u)duds\right|+(t-\sigma^+_1)_+ \left|V_t-\overline{V}_t+\frac{1}{N}\int_{0}^t\nabla\Phi(X_s-x^1_s)ds\right|\nonumber\\
&\hspace{3cm}=O\left(\frac{(t-\sigma^+_1)_+^4}{N^{\frac{3-\delta}{2}}}+\frac{(t-\sigma^+_1)_+^6}{N^{2}}\right).
\label{eq:est_diff_ordre_2_V2_boot}
\end{align}
\end{lemma}


\begin{proof}
Consider an initial configuration $(\mathcal{N},Z,z_{1:\mathcal{N}})$ in the better set $\mathcal{G}_N(\delta,\alpha, \beta)$ with $\beta>\gamma$.
For times $0\leq t \leq \sigma^+_1$, the tagged particle has not interacted with particle $1$ so that $Z_t = \overline{Z}_t$. It is enough to assume that  the interaction starts at time $\sigma^+_1=0$ and consider the shifts in the trajectories for $t \geq 0$.

We introduce the function $C(t)=\sup_{s\in[0,t]}|X_s-\overline{X}_s|$. This function is continuous, nondecreasing and satisfies $C(0)=0$. Define the stopping time 
\begin{equation}
\label{eq: Def T}
T = \sup \left \{ t \leq N^\gamma, \quad C(t)\leq \frac{1}{N^{1-2\gamma-\delta}}
= \frac{1}{N^{\frac{1}{3} + \frac{4 \alpha}{3}  + \delta}}
 \right \}.
\end{equation}
We first analyse the dynamics up to   time $T$  and then check that $T =  N^\gamma$.\\


\noindent
{\bf Step 1: Controlling the shifts on the background particles.}

The evolution equations \eqref{eq:def_tagged}-\eqref{eq: dynamique barres} lead to 
\begin{align}
\label{eq: shift background}
\frac{d^2}{dt^2}(x^1_t-\overline{x}^1_t)=\frac{1}{N}\nabla\Phi(X_t-x^1_t), 
\qquad
\forall i \geq 2,\ \frac{d^2}{dt^2}(x^i_t-\overline{x}^i_t)=\frac{1}{N}\left( \nabla\Phi(X_t-x^i_t)-\nabla\Phi(\overline{X}_t-\overline{x}^i_t)\right). 
\end{align}
This implies the following crude estimate
\begin{align}
\label{eq: shift small particles}
\forall i \leq \mathcal{N}, \qquad 
|x^i_t-\overline{x}^i_t|=O\left(\frac{(t - \sigma^+_i)_+^2}{N}\right),
\end{align}
where $\sigma^+_i = \inf\{t\geq 0,\quad |X_t-x^i_t|\leq 3 R/2\}$.

For $i \geq 2$, this upper bound can be improved as removing particle $1$ from the dynamics induces a small shift on the background particles only from the retroaction with the tagged particle (see \eqref{eq: shift background}). 
Note that  by a Taylor expansion
\begin{align}
\label{eq: Taylor expansion Phi}
\nabla\Phi(\overline{X}_t-\overline{x}^i_t)=&\nabla\Phi(X_t-x^i_t)+\text{Hess}\Phi(X_t-x^i_t)\left(\overline{X}_t-X_t\right)+\text{Hess}\Phi(X_t-x^i_t)\left(x^i_t-\overline{x}^i_t\right)\\
&+
F_{\Phi}(x^i_t,\overline{x}^i_t,X_t,\overline{X}_t), \nonumber
\end{align}
where the remainder $F_{\Phi}(x^i_t,\overline{x}^i_t,X_t,\overline{X}_t)$  satisfies 
\begin{equation}
\label{eq: def FPhi}
\forall x,y,X,Y\in\mathbb{R}^d,  \qquad F_{\Phi}(x,y,X,Y)=O\left( |x-y|^2+|X-Y|^2\right) 
\mathds{1}_{x \in\mathcal{B}(X,R)\text{ or } y \in\mathcal{B}(Y,R)}.
\end{equation}
Thus for $i\in\{2,...,\mathcal{N}\}$ and  $\sigma^+_i \leq t \leq T\wedge N^\gamma$, \eqref{eq: shift background} leads to 
\begin{align}
\label{eq: borne xi C(t)}
\frac{d^2}{dt^2}(x^i_t-\overline{x}^i_t) 
= 
&a_i(t) \left(x^i_t-\overline{x}^i_t\right)+b_i(t),
\end{align}
where
\begin{align}
\left\|a_i(t)\right\|=O\left(\frac{1}{N}\right),\qquad
\left|b_i(t)\right|=O\left(\frac{C(t)}{N} + \frac{t^4}{N^3}+\frac{C(t)^2}{N}\right),\label{eq: borne xi C(t) a et b}
\end{align}
and where we used 
\begin{align*}
&\left\|\frac{1}{N}\text{Hess}\Phi(X_t-x^i_t)\right\|=O\left(\frac{1}{N}\right), 
\quad 
\left|\frac{1}{N}\text{Hess}\Phi(X_t-x^i_t)\left(\overline{X}_t-X_t\right)\right|=O\left(\frac{C(t)}{N}\right),
\end{align*}
and by estimate \eqref{eq: shift small particles}
\begin{align}
\label{eq: borne F}
\left| F_{\Phi}(x^i_t,\overline{x}^i_t,X_t,\overline{X}_t)\right|=O\left(\frac{t^4}{N^3}+\frac{C(t)^2}{N}\right) .
\end{align}
For $t \leq T$ we have $t \leq N^{1/3}$ (as $\gamma \leq 1/3$). Furthermore, the a priori estimate \eqref{eq: Def T} implies that $C(t)\leq  1/N^{1/3}$ and  thus \eqref{eq: borne xi C(t) a et b} becomes
\begin{align*}
\forall i\geq 2,\quad \forall t\in[0, T],\quad \left\|a_i(t)\right\|=O\left(\frac{1}{N}\right),\qquad
\left|b_i(t)\right|=O\left(\frac{1}{N^{2-2\gamma-\delta}}\right) .
\end{align*}
Using \eqref{eq:gron_formule_simple_cas_simple} from Lemma~\ref{lem:Grönwall_precis_alpha}, we obtain from \eqref{eq: borne xi C(t)} and for 
$ t \leq T \leq N^\gamma$
\begin{align}
\label{eq: shift particule i}
\forall i\geq 2,\qquad|x^i_t-\overline{x}^i_t|
=O\left(\frac{(t - \sigma^+_i)_+^2}{N^{2-2\gamma-\delta}}\right)
=O\left(\frac{(t - \sigma^+_i)_+^2}{N^{\frac{4}{3} + \frac{4 \alpha}{3}  + \delta}}\right) .
\end{align}
This estimate will be used below to control the shift on the tagged particle.\\


\noindent
{\bf Step 2: Derivation of  \eqref{eq:est_diff_ordre_1_boot}.}

Using the evolution equations \eqref{eq:def_tagged}-\eqref{eq: dynamique barres}, we get
\begin{align}
\frac{d^2}{dt^2}(X_t-\overline{X}_t)=-\frac{1}{N}&\sum_{i=1}^\mathcal{N}\left( \nabla\Phi(X_t-x^i_t)-\nabla\Phi(\overline{X}_t-\overline{x}^i_t)\right)
- \frac{\nabla\Phi(\overline{X}_t-\overline{x}^1_t)}{N}.
\label{eq:evol_error_part_en_moins_boot_1}
\end{align}
Applying again the Taylor expansion \eqref{eq: Taylor expansion Phi}, we get 
\begin{align}
\frac{d^2}{dt^2}(X_t-\overline{X}_t) 
=&-\left(\frac{1}{N}\sum_{i=1}^\mathcal{N}\text{Hess}\Phi(X_t-x^i_t)\right)(X_t-\overline{X}_t)  
- \frac{\nabla\Phi(\overline{X}_t-\overline{x}^1_t)}{N}  \nonumber \\
& +  \frac{1}{N}\sum_{i=1}^\mathcal{N}\text{Hess}\Phi(X_t-x^i_t)(x^i_t-\overline{x}^i_t) 
 -\frac{1}{N}\sum_{i=1}^\mathcal{N}
F_{\Phi}(x^i_t,\overline{x}^i_t,X_t,\overline{X}_t) \nonumber \\
=:&
a(t)(X_t-\overline{X}_t)+b_1(t) +b_2(t) -\frac{1}{N}\sum_{i=1}^\mathcal{N}
F_{\Phi}(x^i_t,\overline{x}^i_t,X_t,\overline{X}_t).
\label{eq: X-barX}
\end{align}
For configurations in the better set $\mathcal{G}_N(\delta,\alpha, \beta)$, the coefficients are bounded as follows by assumption \eqref{eq: borne sqrt Phi}
\begin{align}
\label{eq: a(t) integrale}
t \leq N^\gamma, \quad 
a(t):=-\left(\frac{1}{N}\sum_{i=1}^\mathcal{N}\text{Hess}\Phi(X_t-x^i_t)\right)
\quad\text{ satisfies }\quad 
\Big\| \int_0^t a(u) \, du \Big\| \leq C_{vel}\sqrt{\frac{t}{N}}N^{\alpha},
\end{align}
with $\beta > \gamma$ \eqref{eq: def gamma beta}.
The main contribution of the source term is given by
\begin{align}
\label{eq: b1 (t)}
b_1(t):=  \frac{\nabla\Phi(\overline{X}_t-\overline{x}^1_t)}{N} = O \left( \frac{1}{N} \right).
\end{align}
We are going to show that the other terms are smaller.
First of all using the control \eqref{eq: shift particule i} on the shifts of the small particles, we get 
\begin{align}
\label{eq: b2 (t) 0}
b_2(t) &:=
\frac{1}{N}\sum_{i=1}^\mathcal{N}\text{Hess}\Phi(X_t-x^i_t)(x^i_t-\overline{x}^i_t)\\
&=\frac{1}{N}\sum_{i=2}^\mathcal{N}\text{Hess}\Phi(X_t-x^i_t)(x^i_t-\overline{x}^i_t)+\frac{1}{N}\text{Hess}\Phi(X_t-x^1_t)(x^1_t-\overline{x}^1_t) \nonumber\\
& =   \frac{1}{N^{\frac{4}{3} + \frac{4 \alpha}{3}  + \delta}}  O\left( \frac{1}{N}\sum_{i=2}^\mathcal{N} (t - \sigma^+_i)_+^2 \; \mathds{1}_{x^i_t\in\mathcal{B}\left(X_t,\frac{3}{2}R\right)} \right)+O\left(\frac{t^2}{N^2}\right). \nonumber
\end{align}
Using a uniform bound $t \leq N^\gamma$ would be too crude, therefore we rely on the 
interaction time estimate \eqref{eq: interaction time better} to obtain a sharper upper bound
\begin{align}
\label{eq: b2 (t)}
|b_2(t)|  =    \frac{1}{N^{\frac{4}{3} + \frac{4 \alpha}{3}  + \delta}}  
O\left( \frac{1}{N}\sum_{i=2}^\mathcal{N}  \big( T_m(v^i_t,V_t)\wedge N^{\gamma} \big)^2\mathds{1}_{x^i_t\in\mathcal{B}\left(X_t,\frac{3}{2}R\right)} \right)+O\left(\frac{t^2}{N^2}\right)
= O\left(\frac{1}{N^{\frac{4}{3} + \frac{4 \alpha}{3}}}\right),
\end{align}
where the second estimate follows from \eqref{eq:G_N_4} as the particle configuration belong to the better set.

The last term to estimate in \eqref{eq: X-barX} involves $F_\Phi$. 
The 2nd order Taylor expansion in \eqref{eq: def FPhi} lead to errors of order $C(t)^2$
which are not sharp enough (by the a priori bound \eqref{eq: Def T}) to deduce an upper bound of order $1/N$. 
Thus we will need to improve \eqref{eq: def FPhi} by expanding $\nabla \Phi$ in  \eqref{eq: Taylor expansion Phi} to the third order\footnote{Where we denote $\left(\nabla^3 \Phi(X-x)\left(Y-X + x- y \right)^{\otimes 2}\right)_i=\sum_{j,k=1}^d\partial_{i,j,k}\Phi(X-x)(Y-X + x- y)_k(Y-X + x- y)_j$}
\begin{align}
\label{eq: def FPhi 3rd}
F_{\Phi}(x,y,X,Y)=
\frac{1}{2}\nabla^3 \Phi(X-x)\left(Y-X + x- y \right)^{\otimes 2}
+ O\left( |x-y|^3+|X-Y|^3\right)\mathds{1}_{x \in\mathcal{B}(X,R)\text{ or } y \in\mathcal{B}(Y,R)}.
\end{align}
Splitting the different contributions in $F_{\Phi}$ by expanding the quadratic term, we obtain
\begin{align}
-\frac{1}{N}\sum_{i=1}^\mathcal{N}
F_{\Phi}(x^i_t,\overline{x}^i_t,X_t,\overline{X}_t)
= \left( \frac{1}{N}\sum_{i=1}^\mathcal{N} \nabla^3 \Phi(X_t-x^i_t) \right) \; \left(X_t - \overline{X}_t  \right)^{\otimes 2}
+ b_3 (t),
\label{eq: F error 3}
\end{align}
where the last term $b_3$ involves terms of order $2$ containing $x^i_t - \overline{x}^i_t$ or order 3 terms:
\begin{align}
\big| b_3 (t) \big| \leq &
 \frac{C \, | X_t - \overline{X}_t |}{N}\sum_{i=1}^\mathcal{N} |x^i_t - \overline{x}^i_t|\mathds{1}_{x^i_t \in\mathcal{B}\left(X_t,R\right)}
+ \frac{C}{N}\sum_{i=1}^\mathcal{N} |x^i_t - \overline{x}^i_t|^2\mathds{1}_{x^i_t \in\mathcal{B}\left(X_t,R\right)}\nonumber
\\
&+ C \, | X_t - \overline{X}_t |^3\left(\frac{1}{N}\sum_{i=1}^\mathcal{N}\mathds{1}_{x^i_t \in\mathcal{B}\left(X_t,\frac{3}{2}R\right)}\right)+\frac{C}{N}\sum_{i=1}^\mathcal{N}| x^i_t - \overline{x}^i_t|^3\mathds{1}_{x^i_t \in\mathcal{B}\left(X_t,\frac{3}{2}R\right)}.
\label{eq: b3 error}
\end{align}
for some constant $C>0$ depending only on $\Phi$. By the a priori estimate \eqref{eq: Def T}, for $t \leq T$ the following bound holds
\begin{equation}
\label{eq: b3 X - bar X|}
| X_t - \overline{X}_t | \leq C(t)\leq  \frac{1}{N^{\frac{1}{3} + \frac{4 \alpha}{3}  + \delta}}
\quad \text{so that} \quad | X_t - \overline{X}_t |^3 = o \left( \frac{1}{N} \right).
\end{equation}
Furthermore for $t \leq N^\gamma$ (with $\gamma \leq 1/3$) we get by
\eqref{eq: shift particule i}
\begin{equation}
\label{eq: b3 x - bar x|}
|x^i_t-\overline{x}^i_t|\mathds{1}_{x^i_t \in\mathcal{B}\left(X_t,\frac{3}{2}R\right)}
=O\left(\frac{(t - \sigma^+_i)_+^2}{N^{\frac{4}{3} + \frac{4 \alpha}{3}  + \delta}}\mathds{1}_{x^i_t \in\mathcal{B}\left(X_t,\frac{3}{2}R\right)}\right) =O\left(\frac{1}{N^{\frac{2}{3} + \frac{8 \alpha}{3}  + 3\delta}}\mathds{1}_{x^i_t \in\mathcal{B}\left(X_t,\frac{3}{2}R\right)}\right) .
\end{equation}
Combining \eqref{eq: b3 X - bar X|} and \eqref{eq: b3 x - bar x|}, we deduce that 
\begin{align}
 b_3 (t) =o\left(\frac{1}{N} \right),
\label{eq: b3 error finale}
\end{align}
as the number of interacting particles is bounded by $O\left(N\right)$ (see \eqref{eq:G_N_3}) for configurations in the better set.
Finally the first term in \eqref{eq: F error 3} is also negligible as the prefactor is uniformly controlled by 
\eqref{eq:G_N_3} for configurations in the better set:
\begin{align*}
\left | \left( \frac{1}{N}\sum_{i=1}^\mathcal{N} \nabla^3 \Phi(X_t-x^i_t) \right) \; \left(X_t - \overline{X}_t  \right)^{\otimes 2}\right|
=O\left( \frac{1}{N^{\frac{1- \delta}{2} }} C(t)^2 \right)= o \left(\frac{1}{N} \right),
\end{align*}
with $C(t)$ bounded by \eqref{eq: b3 X - bar X|}. This implies that 
\begin{align}
\label{eq: def FPhi 3rd final}
\frac{1}{N}\sum_{i=1}^\mathcal{N}
F_{\Phi}(x^i_t,\overline{x}^i_t,X_t,\overline{X}_t)
 =o\left(\frac{1}{N} \right).
\end{align}
Combining the different contributions \eqref{eq: b1 (t)}, \eqref{eq: b2 (t)} and \eqref{eq: def FPhi 3rd final}, we deduce from \eqref{eq: X-barX} that 
\begin{align}
\frac{d^2}{dt^2}(X_t-\overline{X}_t) 
=
a(t)h(t)+ b(t) 
\qquad \text{with $a(t)$ as in \eqref{eq: a(t) integrale} and 
$b(t) = O  \left(\frac{1}{N} \right)$} .
\label{eq: X-barX finale}
\end{align}
Recall that $\gamma < \frac{1-2\alpha}{3}$. 
 Lemma~\ref{lem:gronw_opt_racine} implies
\begin{align*}
\forall t\leq T\wedge N^{\gamma}, \qquad 
|X_t-\overline{X}_t|+t|V_t-\overline{V}_t|=O\left(\frac{t^2}{N}\right).
\end{align*}
For $t\leq N^\gamma$, this implies that $|X_t-\overline{X}_t| \leq \frac{1}{N^{1-2\gamma}}$ thus the 
stopping time $T$ introduced in \eqref{eq: Def T} coincides with  $N^\gamma$. 
This completes the proof of \eqref{eq:est_diff_ordre_1_boot}.\\


\noindent
{\bf Step 3: Derivation of  \eqref{eq:est_diff_ordre_2_boot} and \eqref{eq:est_diff_ordre_2_V2_boot}.}

By symmetry it is enough to prove  \eqref{eq:est_diff_ordre_2_boot}.
To improve the previous computations, we are going to use estimate \eqref{eq:est_diff_ordre_1_boot} 
\begin{equation}
\label{eq: enhanced tagged particle}
\forall t \leq N^\gamma, \qquad |X_t-\overline{X}_t| \leq C(t) = O\left(\frac{t^2}{N}\right).
\end{equation}
Let us first improve the estimate  \eqref{eq: shift particule i} on the shift of the small particles. 
In particular, the estimates \eqref{eq: borne xi C(t)}-\eqref{eq: borne xi C(t) a et b} on the evolution of the background particles for $i\in\{2,...,\mathcal{N}\}$ implies now  for $\sigma^+_i \leq  t \leq N^\gamma$
\begin{align*}
\frac{d^2}{dt^2}(x^i_t-\overline{x}^i_t) 
= 
&a_i(t) \left(x^i_t-\overline{x}^i_t\right)+b_i(t),
\end{align*}
with
\begin{align*}
\left\|a_i(t)\right\|=O\left(\frac{1}{N}\right),\qquad
\left|b_i(t)\right|=O\left(\frac{t^2}{N^2} + \frac{t^4}{N^3}\right)=O\left(\frac{t^2}{N^2}\right).
\end{align*}
Using \eqref{eq:gron_formule_simple_cas_simple} from Lemma~\ref{lem:Grönwall_precis_alpha},
estimate \eqref{eq: shift particule i} can therefore be enhanced for  
$t  \leq N^\gamma$ and $i \geq 2$ to
\begin{align}
\label{eq: shift particule i improved}
|x^i_t-\overline{x}^i_t|
=O\left(\frac{(t - \sigma^+_i)_+^4}{N^2}\right) = O\left(\frac{t^4}{N^2}\right).
\end{align}
Let $h(t):=X_t-\overline{X}_t+\frac{1}{N}\int_{0}^t\int_0^s\nabla\Phi(\overline{X}_u-\overline{x}^1_u)duds$. 
We have
\begin{align*}
h''(t)=&\frac{d}{dt}V_t-\frac{d}{dt}\overline{V}_t+\frac{1}{N}\nabla\Phi(\overline{X}_t-x^1_t)
=-\frac{1}{N}\sum_{i=1}^\mathcal{N}\left( \nabla\Phi(X_t-x^i_t)-\nabla\Phi(\overline{X}_t-\overline{x}^i_t)\right).
\end{align*}
Applying again the Taylor expansion \eqref{eq: Taylor expansion Phi}, we get 
\begin{align*}
h''(t)=&-\left(\frac{1}{N}\sum_{i=1}^\mathcal{N}\text{Hess}\Phi(X_t-x^i_t)\right)h(t)
+  \frac{1}{N}\sum_{i=1}^\mathcal{N}\text{Hess}\Phi(X_t-x^i_t)(x^i_t-\overline{x}^i_t)\\
&+\left(\frac{1}{N}\sum_{i=1}^\mathcal{N}\text{Hess}\Phi(X_t-x^i_t)\right)\left(\frac{1}{N}\int_{0}^t\int_0^s\nabla\Phi(\overline{X}_u-\overline{x}^1_u)duds\right)\\
&+\frac{1}{N}\sum_{i=1}^\mathcal{N} F_{\Phi}(x^i_t,\overline{x}^i_t,X_t,\overline{X}_t).
\end{align*}
This can be rewritten as 
\begin{align}
\label{eq: h}
h''(t) =& a(t)h(t)+b_2(t)+b_4(t) +\frac{1}{N}\sum_{i=1}^\mathcal{N} F_{\Phi}(x^i_t,\overline{x}^i_t,X_t,\overline{X}_t),
\end{align}
where $a(t)$ was defined in \eqref{eq: a(t) integrale}, $b_2(t)$
as in \eqref{eq: b2 (t) 0} 
 \begin{align}
b_2(t):=\frac{1}{N} \text{Hess}\Phi(X_t-x^1_t)(x^1_t-\overline{x}^1_t) +
\frac{1}{N}\sum_{i=2}^\mathcal{N}\text{Hess}\Phi(X_t-x^i_t)(x^i_t-\overline{x}^i_t) 
= O\left(\frac{t^2}{N^2}+\frac{t^4}{N^2}\right),
\end{align} 
where we used \eqref{eq: shift particule i improved} (and \eqref{eq: shift small particles} for $i=1$) to improve on estimate \eqref{eq: b2 (t)}.
The term $b_4$ can be controlled easily thanks to assumption  \eqref{eq:G_N_3}
\begin{align*}
b_4(t):= \left(\frac{1}{N}\sum_{i=1}^\mathcal{N}\text{Hess}\Phi(X_t-x^i_t)\right)\left(\frac{1}{N}\int_{0}^t\int_0^s\nabla\Phi(\overline{X}_u-\overline{x}^1_u)duds\right)
= O\left(\frac{t^2}{N^{\frac{3-\delta}{2}}}\right).
\end{align*}
Finally to control $F_{\Phi}$, it is enough to combine the quadratic bound \eqref{eq: def FPhi} with the enhanced estimates \eqref{eq: enhanced tagged particle} and \eqref{eq: shift particule i improved}
\begin{align*}
\frac{1}{N}\sum_{i=1}^\mathcal{N} F_{\Phi}(x^i_t,\overline{x}^i_t,X_t,\overline{X}_t)
= O\left(\frac{t^8}{N^4} + \frac{t^4}{N^2}\right)=O\left(\frac{t^4}{N^2}\right).
\end{align*}
From the previous estimates, \eqref{eq: h} becomes, for $t\leq N^\gamma$
\begin{align}
h''(t) = a(t)h(t)+ O \left(\frac{t^4}{N^2}+\frac{t^2}{N^{\frac{3-\delta}{2}}}\right).
\end{align}
Using Lemma~\ref{lem:gronw_opt_racine}, we obtain for $t\leq N^\gamma$
\begin{align*}
|h(t)|+t|h'(t)|=O\left(\left(\frac{t^4}{N^2}+\frac{t^2}{N^{\frac{3-\delta}{2}}}\right)t^2\right).
\end{align*}
This completes the derivation of \eqref{eq:est_diff_ordre_2_boot}.
\end{proof}


Using similar (but easier) methods as in the derivation of Lemma \ref{Lem: Better set estimates}, one can also obtain bounds for configurations in the good set.
\begin{lemma}[Good set estimates]
\label{lem:influence_une_particule}
Consider a configuration $(\mathcal{N},Z,z_{1:\mathcal{N}})$ in the good set $\mathcal{G}_N(\delta)$
introduced in Definition \ref{def:hyp_good_set}. 
The distance between the trajectories  $(Z_t)_t$ and $(\overline{Z}_t)_t$ defined in \eqref{eq:def_tagged}-\eqref{eq: dynamique barres} can be controlled as follows  for all $t\in[\sigma^+_1,\sigma^+_1+N^{\frac{1-\delta}{4}}]$
\begin{align}
&\left|X_t-\overline{X}_t\right|+(t-\sigma^+_1)_+ \left|V_t-\overline{V}_t\right| =O\left(\frac{(t-\sigma^+_1)_+^2}{N}\right).
\label{eq:est_diff_ordre_1}
\end{align}
\end{lemma}


\begin{proof}
The proof follows the argument in Lemma \ref{Lem: Better set estimates} and we will use the same notations. Again, we assume without loss of generality $\sigma^+_1=0$.
The main difference is the shorter timescale which makes the proof easier and allows to use 
only the cruder structure of the good set. The  stopping time in \eqref{eq: Def T} is now given by 
\begin{equation}
\label{eq: Def T 2}
T = \sup \Big \{ t \leq N^{\frac{1-\delta}{4}}, \quad C(t)\leq 1/ \sqrt{N} \Big \}.
\end{equation}
The evolution equations are similar so that the shifts on the background particles obey equations
\eqref{eq: borne xi C(t)}-\eqref{eq: borne xi C(t) a et b} 
\begin{align*}
i \geq 2, t \leq T,  \qquad 
\frac{d^2}{dt^2}(x^i_t-\overline{x}^i_t) 
= 
&a_i(t) \left(x^i_t-\overline{x}^i_t\right)+b_i(t),
\end{align*}
with
\begin{align*}
\left\|a_i(t)\right\|=O\left(\frac{1}{N}\right),\qquad
\left|b_i(t)\right|=O\left(\frac{C(t)}{N} + \frac{t^4}{N^3}+\frac{C(t)^2}{N}\right)=O\left(\frac{1}{N^{3/2}}\right),
\end{align*}
as for $t \leq T$ then $C(t)\leq \frac{1}{\sqrt{N}}$ and $t \leq N^{\frac{1-\delta}{4}}$.
Using \eqref{eq:gron_formule_simple_cas_simple} from Lemma~\ref{lem:Grönwall_precis_alpha}, we obtain 
\begin{align}
\label{eq: shift particule i 2}
i\in\{2,.., \mathcal{N} \}, \qquad 
|x^i_t-\overline{x}^i_t|
=O\left(\frac{t^2}{N^{\frac{3}{2}}}\right).
\end{align}
To control the shift \eqref{eq:est_diff_ordre_1} of the tagged particle, we recall  \eqref{eq: X-barX}
\begin{align}
\frac{d^2}{dt^2}(X_t-\overline{X}_t) 
=&-\left(\frac{1}{N}\sum_{i=1}^\mathcal{N}\text{Hess}\Phi(X_t-x^i_t)\right)(X_t-\overline{X}_t)  
- \frac{\nabla\Phi(\overline{X}_t-\overline{x}^1_t)}{N}  \nonumber \\
& +  \frac{1}{N}\sum_{i=1}^\mathcal{N}\text{Hess}\Phi(X_t-x^i_t)(x^i_t-\overline{x}^i_t) 
 -\frac{1}{N}\sum_{i=1}^\mathcal{N}
F_{\Phi}(x^i_t,\overline{x}^i_t,X_t,\overline{X}_t) \nonumber \\
=:&
a(t)h(t)+b_1(t) +b_2(t) -\frac{1}{N}\sum_{i=1}^\mathcal{N}
F_{\Phi}(x^i_t,\overline{x}^i_t,X_t,\overline{X}_t).
\label{eq: X-barX bis}
\end{align} 
For configurations in the good set $\mathcal{G}_N(\delta)$, the coefficients are bounded as follows 
\begin{align*}
a(t):=-\left(\frac{1}{N}\sum_{i=1}^\mathcal{N}\text{Hess}\Phi(X_t-x^i_t)\right)\quad\text{ satisfies }\quad \|a(t)\|=O\left(N^{-\frac{1-\delta}{2}}\right).
\end{align*}
A direct estimate shows that the larger source term is $b_1(t) = O(1/N)$.
By the new bound \eqref{eq: shift particule i 2} on the shifts of the background particles 
$|x^i_t-\overline{x}^i_t| = O\left(\frac{t^2}{N^{\frac{3}{2}}}\right)$ for $i \geq 2$, we deduce that 
for $t \leq T\leq N^{\frac{1-\delta}{4}}$
\begin{align*}
b_2(t)=O\left(\frac{t^2}{N^{\frac{3}{2}}} \right) =o\left(\frac{1}{N}\right) ,
\end{align*}
as the number of particles interacting at time $t$ is always of order $N$.
Finally the error term $F_{\Phi}$ is controlled as in \eqref{eq: borne F} so that by the a a priori estimate
$C(t)\leq 1/ \sqrt{N}$ in \eqref{eq: Def T 2}, we get for $t \leq T$
\begin{align*}
\left| F_{\Phi}(x^i_t,\overline{x}^i_t,X_t,\overline{X}_t)\right|
=O\left(\frac{t^4}{N^3}+C(t)^2\right) =O\left(\frac{1}{N}\right) .
\end{align*}
From the previous estimates, we deduce from 
Lemma~\ref{lem:Grönwall_precis_alpha} that for $t\leq T$
\begin{align}
\label{eq: borne h(t)}
\forall t \leq T, \qquad
\left|X_t-\overline{X}_t\right|+ t \left|V_t-\overline{V}_t\right| =O\left(t^2/N \right).
\end{align}
For $t \leq N^{\frac{1-\delta}{4}}$, this is compatible with the a priori estimate  \eqref{eq: Def T 2} thus 
$T = N^{\frac{1-\delta}{4}}$. 
This completes the proof of Lemma \ref{lem:influence_une_particule}.
\end{proof}

%
%
%
%

\subsection{Being in the better set}

The goal of this section is to improve Proposition \ref{prop:good_set_good} and show that with high probability,  the initial configurations belong to a better set (with well chosen parameters).


\begin{proposition}
\label{prop:fin_bootstrap}
Let $\delta>0$ small enough, assume $d\geq 3$ and consider 
\begin{align}
\label{eq: optimal beta*, alpha*}
\beta^*=\frac{2d^2+5d+4}{2(d+2)(d+4)}-\delta 
 > \frac{1}{2},
\qquad \alpha^*=\frac{1}{4(d+2)}+\delta.
\end{align}
There exists $\delta'>0$ such that 
the better set (see Definition \ref{def:hyp_boot}) occurs with high probability 
\begin{equation}
\mathbb{P}\left[ \overline{\mathcal{G}_N(\delta,\alpha^*,\beta^*)} \right]
= O\left(e^{-N^{\delta'}}\right).
\end{equation}
\end{proposition}


\begin{remark}
 As the statement above holds also in dimension 3, we kept it for completeness. 
\end{remark}

To prove Proposition~\ref{prop:fin_bootstrap} above, we first show that $\mathcal{G}_N(\delta,\alpha, \beta)$ holds for a certain $\beta$, smaller than the one claimed  in \eqref{eq: optimal beta*, alpha*}, and then bootstrap the proof in order to reach a final value for $\beta$. We thus start proving:


\begin{proposition}\label{prop:init_bootstrap}
Assume $d\geq3$. Given $\delta>0$ small enough, we consider the better set $\mathcal{G}_N(\delta,\alpha,\beta)$ with the following parameters
\begin{align}
\label{eq: sub-optimal beta, alpha}
\alpha=\frac{1}{4(d+2)}+\delta,  \quad \beta=\frac{1}{2} - \frac{1}{2(d+2)}-\delta.
\end{align}
There exists $\delta'>0$ such that the better set occurs with high probability 
\begin{equation}
\mathbb{P}\left[ \overline{\mathcal{G}_N(\delta,\alpha,\beta)} \right]
= O\left(e^{-N^{\delta'}}\right).
\label{eq: better set niveau 1}
\end{equation}
\end{proposition}

The bootstrap estimates are obtained as follows:


\begin{proposition}\label{prop:boot_beta}
Assume  $d\geq 3$ and that there exists $\alpha, \beta, \delta>0$ such that $\mathbb{P}\left[\overline{\mathcal{G}_N(\delta,\alpha, \beta)}\right]=O\left(e^{-N^{\delta_2}}\right)$ for some $\delta_2>0$ with 
\begin{equation}\label{eq:cond_beta}
0<\alpha<\frac{d+1}{2(d+4)}-\frac{3}{2}\delta,\qquad \frac{1-2\alpha}{3}+\delta\leq\beta \leq\frac{d+2}{d+4}-2\alpha-2\delta.
\end{equation}
Then, provided $\delta$ is small enough, defining 
\begin{align}
\tilde{\alpha}=\alpha+\frac{\delta}{2},\qquad \tilde{\beta}=
\Psi_\alpha(\beta) - \delta
\quad \text{with} \quad  
\Psi_\alpha(\beta) = 
\left\{\begin{array}{ll}\frac{d}{d+4}+2\alpha &\text{ if }\beta\geq\frac{d-2}{d+4}+6\alpha\\ \frac{\beta}{2}+\frac{d+2}{2(d+4)}-\alpha &\text{ else.}\end{array}\right.
\label{eq: recurrrence Psi}
\end{align} 
we obtain $\mathbb{P}\left[\overline{\mathcal{G}_N(\delta,\tilde{\alpha},\tilde{\beta})}\right]=O\left(e^{-N^{\delta_3}}\right)$ for some $\delta_3>0$.
\end{proposition}

We can then deduce Proposition~\ref{prop:fin_bootstrap}
 from Proposition~\ref{prop:init_bootstrap} by iterating 
 Proposition~\ref{prop:boot_beta}.


\begin{figure}
\centering
\includegraphics[width=0.6\textwidth,]{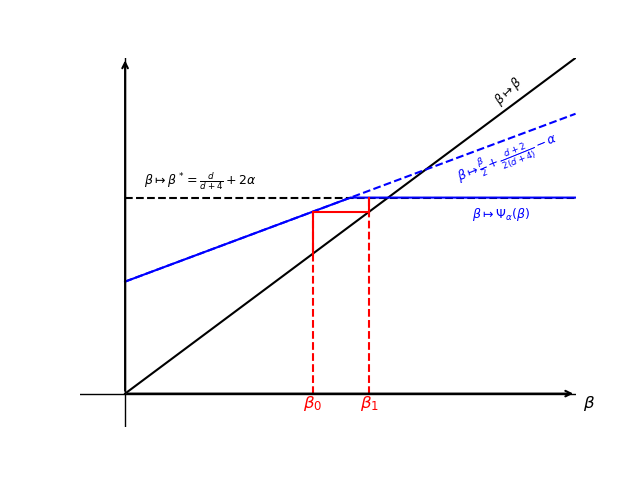}
\vspace{-1cm}
\caption{{\small Illustration of the proof of Proposition~\ref{prop:fin_bootstrap} for $d=4$, only two iterations are needed. }}
	\label{fig:preuve_recurrence}
\end{figure}


\begin{proof}[Proof of Proposition~\ref{prop:fin_bootstrap}]
Using Proposition~\ref{prop:init_bootstrap}, and defining  for some $\delta_0>0$
\begin{align}
\label{eq: initial alpha 0}
\alpha_0=\frac{1}{4(d+2)}+
 \delta_0,
\qquad \text{ and }  \qquad
\beta_0=\frac{1}{2}-\frac{1}{2(d+2)}-\delta_0,
\end{align}
 we know that $\mathbb{P}\left[\overline{\mathcal{G}_N(\delta_0,\alpha_0,\beta_0)}\right]=O\left(e^{-N^{\delta'}}\right)$ holds for some $\delta'>0$. We now use Proposition~\ref{prop:boot_beta} to construct a sequence of $\alpha_n, \beta_n, \delta_n$ satisfying the same property, defined recursively by \eqref{eq: recurrrence Psi}.

We just need to prove that the values for $\alpha^*, \beta^*$ given in Proposition~\ref{prop:fin_bootstrap} are eventually reached (see Figure \ref{fig:preuve_recurrence}).

For this, we first solve the recurrence \eqref{eq: recurrrence Psi} with $\delta =0$ and introduce the sequence $\tilde \beta_{n+1} = \Psi_{\alpha^*} (\tilde \beta_n)$ with $\alpha^*=\frac{1}{4(d+2)}$ and initial data $\beta_0$. This recurrence has the following properties.
\begin{itemize}
\item The recurrence fixed point $\tilde \beta^*$  is 
\begin{align}
\label{eq: tilde beta}
\tilde \beta^* = \frac{d}{d+4} + 2 \alpha^* = \frac{2 d^2 + 5 d+4}{2(d+4)(d+2)}>\frac{1}{2}.
\end{align}
Note that this fixed point $\beta^*$ is greater than the threshold $\frac{d-2}{d+4}+6\alpha^*$ 
beyond which $\Psi_{\alpha^*}$ is constant.
\item 
The sequence $(\tilde \beta_n)_n$ is non-decreasing, satisfying $\frac{1}{3}<\tilde \beta_n\leq \tilde \beta^*<\frac{d+2}{d+4}-2\alpha_n$ for all $n$. 
In particular, this sequence satisfies the necessary condition \eqref{eq:cond_beta} to use Proposition~\ref{prop:boot_beta}. 
\item Since $\Psi_{\alpha^*}(\beta)$ is constant and equal to $\tilde \beta^*$ for $\beta\geq \frac{d-2}{d+4}+6\alpha^*$, the sequence $(\tilde \beta_n)_n$ reaches $\tilde \beta^*$ in a finite number of iterations.
\end{itemize}
Up to considering $\delta_0$ small enough in \eqref{eq: initial alpha 0}, the full recurrence \eqref{eq: recurrrence Psi} shares the same properties as above. In particular the fixed point is reached in a finite number of steps and is therefore a perturbation of 
$\tilde \beta^*$.
This concludes the proof of Proposition~\ref{prop:fin_bootstrap}.
\end{proof}


\begin{proof}[Proof of Proposition~\ref{prop:init_bootstrap}]

Given $\delta >0$, we know, from Proposition \ref{prop:good_set_good}, that the particle configurations belong to the good set $\mathcal{G}_N(\delta)$ with high probability. 
We are going to upgrade this information to prove that, in fact, the configurations can be chosen in  the better set
\eqref{eq: better set niveau 1}.

Fix $(\alpha,\beta)$ as in \eqref{eq: sub-optimal beta, alpha} and $\mathcal{I}\in\{1,...,d\}^{|\mathcal{I}|}$ a multi-index with $|\mathcal{I}|\in\{1,2,3\}$.
We set $F = \partial_{\mathcal{I}} \Phi$.
Recalling \eqref{eq: borne sqrt Phi}, our goal is to show that
\begin{align}
\label{eq: brone srqt local}
\forall T\in[0, N^{\beta}], \qquad 
\left|\frac{1}{N}\int_0^T\sum_{i} F (X_t-x^i_t)dt\right|
\leq& \sqrt{\frac{T}{N}}N^{\alpha},
\end{align}
with probability greater than $1-e^{-N^{\delta'}}$ for some $\delta'>0$. 
Note that the same proof holds for \eqref{eq:control_boot_avec_contrainte_sup} with an additional velocity cut-off. 

To derive \eqref{eq: brone srqt local}, we approximate the dynamics by a sum of (almost) independent random variables, in order to use an Hoeffding-like inequality. This approximation is constructed by separating the influence of fast background particles (interacting for a short time with the tagged particle) and the slow background particles, which induce some time correlation, but which are rare enough so that their influence can be controlled.
\medskip

Let $\delta_v>0$ satisfy 
\begin{equation}\label{eq:borne_delta_v_init}
\delta_v\in[ N^{-\frac{1}{5}}\vee N^{-\frac{1}{d}},N^{-\delta}].
\end{equation}
The exact lower bound above doesn't matter, we only need it to be greater than $N^{-\frac{1-\delta}{4}}$ and $N^{-\frac{1}{d}}$, and we will tune its value latter. Likewise, the upper bound appears for technical reasons (in Step 2.3 below), and one could in fact take it as close to $1$ as necessary, we simply give an explicit value for the sake of clarity.

We now define 
\begin{equation}
    \label{eq: parameters}
T_f=N^{ \frac{1}{2} -\delta}\delta_v, \quad 
\qquad \delta_t=10\times \frac{6R}{\delta_v}, 
\qquad t_k=k\delta_t,
\qquad K= \left\lfloor \frac{T}{\delta_t} \right\rfloor,
\end{equation}
where $T_f$ stands for the final time such that by Lemma~\ref{lem:max_inter_time} and for configurations in the good set $\mathcal{G}_N(\delta)$, a particle interacting at time $0$ with a relative speed greater than $\delta_v$ does not re-interact before  time $T_f$ . In this proof, we always choose $T\leq T_f$. 

Proposition \ref{prop:good_set_good}  ensures that with high probability the configurations belong to the good set $\mathcal{G}_N(\delta)$, thus it is enough to establish \eqref{eq: brone srqt local}  for configurations in $\mathcal{G}_N(\delta)$.
We decompose the integral as follows 
\begin{align}
\int_0^{T}\frac{1}{N}\sum_{i}F(X_s -x^i_s) ds=& \sum_{k=0}^{K} \frac{1}{N}\sum_i 
 \int_{t_k}^{t_{k+1}}  F(X_s-x^i_s)ds
= \Af_T + \As_T+\mathcal{E}^G,
\label{eq: decomposition AfT et AsT}
\end{align}
where the contribution of  the fast relative velocities is 
\begin{align}
\label{eq: AfT}
\Af_T: =
\sum_{k=0}^{K} \frac{\mathds{1}_{\mathcal{G}^{k+1}_N(\delta)}}{N}\sum_{i } 
 \int_{t_k}^{t_{k+1}} F(X_s - x^i_s) \mathds{1}_{|v^i_s -  V_s | \geq  2\delta_v}  ds,
\end{align}
the contribution of the 
slow relative velocities is
\begin{align}
\label{eq: AsT}
\As_T :=
\sum_{k=0}^{K} \frac{1}{N}\sum_{i }
  \int_{t_k}^{t_{k+1}} F(X_s - x^i_s) \mathds{1}_{| v^i_s -   V_s  |< 2\delta_v} ds,
\end{align}
and finally the error due to the good set is
\begin{align}
\label{eq:Error_G}
\mathcal{E}^G:=\sum_{k=0}^{K} \frac{\mathds{1}_{\overline{\mathcal{G}^{k+1}_N(\delta)}}}{N}\sum_{i }  \int_{t_k}^{t_{k+1}}  F(X_s-x^i_s)\mathds{1}_{|v^i_s -  V_s | \geq  2\delta_v} ds.
\end{align}
 We denote by $\mathcal{G}^k_N(\delta)$ the set of initial conditions such that \eqref{eq:G_N_1}-\eqref{eq:G_N_5}, (but, crucially, not \eqref{eq:G_N_0}) holds for  $0\leq t\leq t_k$. 
Unlike $\mathcal{G}_N(\delta)$, the set  $\mathcal{G}^k_N(\delta)$ is not global in time and will be suitable later to implement a martingale type argument: it   involves only particles  interacting with the tagged particle on the time frame $[0,t_k]$. 
 Since the process is deterministic, it is equivalent to say that the configuration at time $0$ ensures the various controls, or that it is the configuration at any other given time which does so. 
In particular the indicator $\mathds{1}_{\mathcal{G}^k_N(\delta)} = \mathds{1}_{(Z_{t_{k}},z^{1:\mathcal{N}}_{t_{k}})\in \mathcal{G}^k_N(\delta)}$ can be understood as a function of the configuration at time $t_{k}$.

The goal of decomposition \eqref{eq: decomposition AfT et AsT} is to separate the effect of the particles with slow relative velocities which are rare but might induce long range correlations as they remain close to the tagged particle for a long time.  Since $\mathbb{E}\left[\mathcal{E}^G\right]=O\left(e^{-N^{\delta/4}/4}\right)$ by Proposition~\ref{prop:good_set_good} and since $F$ is bounded, we directly have by Markov's inequality
\begin{equation}
\label{eq:borne_proba_erreur}    \mathbb{P}\left[\left|\mathcal{E}^G\right|\geq \frac{T\delta_t^2}{N}\right]=O\left(e^{-N^{\delta/8}}\right).
\end{equation}
Then, setting  
\begin{align}
\label{eq: def Istk}
I(t_k):=   \mathds{1}_{\mathcal{G}^{k+1}_N(\delta)}\frac{1}{N}\sum_{i }   \int_{t_k}^{t_{k+1}} 
F(X_s - x^i_s) \mathds{1}_{|v^i_s -  V_s | \geq  2\delta_v}
ds,
\end{align}
we further decompose  $\Af$ into 
\begin{align}
\label{eq: def odd even}
\Afodd_T =  \sum_{\scriptsize{\begin{array}{l}k=0,\\ k\text{ odd}\end{array}}}^{K} I(t_k), 
\qquad 
\Afeven_T =  \sum_{\scriptsize{\begin{array}{l}k=0,\\ k\text{ even}\end{array}}}^{K} I(t_k).
\end{align}
Given the parameter $\delta_v$, the fast particles interact with  the tagged particle at most during a time of order $\delta_t$ (and in fact, since we ensure $|v^i_s -  V_s | \geq  2\delta_v$, a particle interacts for a time smaller than $\delta_t/2$), thus two disjoint time intervals $[t_k, t_{k+1}], [t_\ell, t_{\ell+1}]$ with $|k - l |\geq 2$ in \eqref{eq: AfT} involve different particles. This allows us to recover some independence and to average out in time.

Bounding $\Afodd_T$ and $\Afeven_T$ is done equivalently, up to a small time shift $\delta_t$, so we only deal with the former.
We split the proof into the following steps:
\begin{itemize}
\item \underline{Step 1:} We create a set $A_{\delta_v, \delta_t}(t)$ 
to distinguish between the fast background particles and the slow given the initial conditions and to  prescribe on which interval $[t_k, t_{k+1}]$ they interact.
\item \underline{Step 2:} 
 We first deal with the fast background particles and prove a concentration inequality to control the time average of $\Afodd_T$.
\subitem \underline{Step 2.1:} We start by defining a conditional-like expectation with respect to the trajectory up to time $t_k$, denoted $\widehat I(t_k)$, using the sets $A_{\delta_v, \delta_t}(t_k)$, which later allows us to use the martingale structure of  $\Afodd_T$  and separate iteratively the last term from the rest of the sum.
\subitem \underline{Step 2.2:} Since the terms of the sum are not exactly of mean $0$, we decompose $\Afodd_T$  into $\sum_k (I(t_k)-\widehat I(t_k))+ \sum_k \widehat I(t_k)$ and deal with the second sum.
\subitem \underline{Step 2.3:} We now use Hoeffding's argument to deal with $\sum I(t_k)-\widehat I(t_k)$, and conclude on the control of $\Afodd_T$.
\item \underline{Step 3:} The term $\As_T$ \eqref{eq: AsT}
involving slow background particles is small as it depends on fewer particles.
\item \underline{Step 4:} We conclude by tuning the parameters to obtain the best possible bound. 
\end{itemize}
\vspace{0.5cm}


\noindent\textbf{$\bullet$ Step 1: Localizing the time integrals.}

In order to deal independently with each integral in \eqref{eq: AfT}, we thus start by "fixing" the set of fast background particles which may interact during a given time frame of order $\delta_t$.
 For a time $t \geq 0$, we introduce the following  set 
 \begin{align*}
A_{\delta_v, \delta_t}(t):=\left\{(x,v)\in\mathbb{T}\times\mathbb{R}^d\ \text{ s.t. }x\in\bigcup_{s\in[ \delta_t,2\delta_t]}\mathcal{B}\left(X_t+ s(V_t-v),\frac{3}{2}R\right) \ \text{ and }\ |v-V_t|\geq \delta_v\right\}.
\end{align*}
We stress that this set is determined only  by the tagged particle at time $t$, i.e. by $Z_t$.
We will see that any background  particle belonging to this set at time $t$ may interact with the tagged particle only during the time frame $[t+\delta_t, t+2\delta_t]$. The time separation between $t$ and $t + \delta_t$ will be used to decouple the dynamical correlations. 

These quantities are constructed such that, if the relative velocity of a background particle is greater than $2\delta_v$, then it interacts for at most a time $\delta_t/20$, and once it leaves the interaction radius cannot come back before a time $T_f$ (both these properties are consequences of Lemma~\ref{lem:max_inter_time}) for configurations in the good set $\mathcal{G}_N(\delta)$.
From this we can deduce the following properties which will be useful afterwards.


\begin{lemma}
    \label{lem: set A}
\begin{itemize}
\item[(i)]     Given a configuration in $\mathcal{G}^{k-1}_N(\delta)$, the background particles  in $A_{\delta_v, \delta_t}(t_{k-1})$ at time $t_{k-1}$ did not interact with the tagged particle in $[0, t_{k-1}]$.
\item[(ii)]  For a configuration in $\mathcal{G}^{k+1}_N(\delta)$, if there exists $t\in[t_k,t_{k+1}]$ such that $x^1_t\in\mathcal{B}\left(X_t,R\right)$ and $|v^1_t-V_t|\geq 2\delta_v$, then, for $N$ large enough,  $z^1_{t_{k-1}} \in A_{\delta_v, \delta_t}(t_{k-1})$.
\end{itemize}
\end{lemma}

This lemma follows by localizing in time the various controls obtained in $\mathcal{G}_N(\delta)$.
In particular, note that 
 $\delta_t\leq 60R N^{\frac{1}{5}}$ and thus, for any 
$t \leq t_{k-1}$
 and configurations in $\mathcal{G}^{k+1}_N(\delta)$, we have  (see Definition~\ref{def:hyp_good_set})
\begin{align}
\sup_{s\in[ t, t+ 2\delta_t]}|X_t+(s-t)V_t-X_s|=&O\left(\frac{\delta_t^2}{N^{\frac{1-\delta}{2}}}\right)=O\left(\frac{1}{N^{\frac{1}{10}-\frac{\delta}{2}}}\right),\label{eq:init_boot_control_1}\\
\sup_{s\in[ t, t+ 2\delta_t]}|x^1_t+(s-t)v^1_t-x^1_s|=&O\left(\frac{\delta_t^2}{N}\right)=O\left(\frac{1}{N^{\frac{3}{5}}}\right),\label{eq:init_boot_control_2}\\
\text{and }\sup_{s\in[ t, t+ 2\delta_t]}\Big| |v^1_t-V_t|-|v^1_s-V_s| \Big|
=&O\left(\frac{\delta_t}{N^{\frac{1-\delta}{2}}}\right)=O\left(\frac{1}{N^{\frac{3}{10}-\frac{\delta}{2}}}\right)\ll \delta_v.\label{eq:init_boot_control_3}
\end{align}

We may now use the sets $A_{\delta_v, \delta_t}(t_k)$ in our computations.
\vspace{0.5cm}


\noindent\textbf{Step 2.1: Conditional expectation, integrating on $A_{\delta_v,\delta_t}(t_{k-1})$.}

We are going to define an analogous of the conditional expectation of $I(t_k)$ \eqref{eq: def Istk} given a particle configuration $\{ Z_s, z^{1:M'}_s , \ s \leq t_{k-1} \}$ belonging to the set $\mathcal{G}_N^{k-1}(\delta)$.
As the dynamics is deterministic, it is enough to prescribe the configuration at time $t_{k-1}$ and the set $A_{\delta_v, \delta_t}(t_{k-1}) \subset \bbT \times \bbR^d$ is determined only by $Z_{t_{k-1}}$. 
The background configuration is then split into the particles $\{ z^{1:M}_s, \ s \leq t_{k-1} \}$ which do not belong to $A_{\delta_v, \delta_t}(t_{k-1})$ and the $m = M'-M$ others. 
According to Lemma \ref{lem: set A} (i), for configurations in $\mathcal{G}_N^{k-1}(\delta)$, the tagged particle may have interacted only with the particles in $\overline{A_{\delta_v, \delta_t}(t_{k-1})}$ up to time $t_{k-1}$. Thus we can average over the remaining particles in $A_{\delta_v, \delta_t}(t_{k-1})$ to define
\begin{align}
\widehat I(t_k)  
 = & 
\frac{\mathds{1}_{\mathcal{G}_N^{k-1}(\delta)}}{\mathcal{\hat Z}}  \sum_{m \geq 1}^\infty \frac{N^m}{m!}
\int dz_{M+1:M+m} \prod_{i=M+1}^{M+m} \gamma(v_i)  \mathds{1}_{z^i_{t_{k-1}}\in A_{\delta_v,\delta_t}(t_{k-1})}
 I (t_k) ,
\label{eq:  Istk averaged}
\end{align}
where the normalization constant $\mathcal{\hat Z}$ depends on  $Z, z_{1:M}$ and $t_{k-1}$
\begin{align}
\label{eq:  Istk averaged hat Z}
\mathcal{\hat Z} (Z, z_{1:M})
: =   \sum_{m  \geq 1}^\infty\frac{N^m}{m!}
\int dz_{M+1:M+m} \prod_{i=M+1}^{M+m} \gamma(v_i)  
\mathds{1}_{z^i_{t_{k-1}}\in A_{\delta_v,\delta_t}(t_{k-1})} .
\end{align} 
By construction $\widehat I(t_k)$ is a function of the configurations at time $t_{k-1}$ (or alternatively by the backward flow of the configurations at time $0$).
Note that in \eqref{eq:  Istk averaged} and in \eqref{eq:  Istk averaged hat Z}, the integral is over the initial data and the constraint is at time $t_{k-1}$ obtained by free motion of the corresponding particles as the configurations are in $\mathcal{G}_N^{k-1}(\delta)$.
\medskip

The following lemma summarizes the key property of $\widehat I(t_k)$.


\begin{lemma}\label{lem:mart_I}
Assume, without loss of generality, that $k$ is odd. For $\lambda \geq 0$, we set
\begin{align}
\label{eq:  G tilde}
\tilde G^k = \exp \Bigg( \lambda \sum_{\scriptsize{\begin{array}{l}i=1,\\ i\text{ odd}\end{array}}}^{k-2}\left(I(t_i)-\widehat I(t_i)\right)  \Bigg) ,
\end{align} 
then 
\begin{align}
\label{eq:  Istk averaged expectation}
\bbE \left[ \tilde G^k \left(  I(t_k)  - \widehat I(t_k) \right) \right] = 0.
\end{align}
\end{lemma}


\begin{proof}
We are going to use time $t_{k-1}$ as the new time origin (recall that $\widehat I(t_k)$ is a function of the configuration at time $t_{k-1}$). 
We write $\mathcal{G}_N^{-k}(\delta)$ as the time shift by $-t_k$ of $\mathcal{G}^k_N(\delta)$, it refers to the set of initial conditions $(Z, z_{1:M} )$ such that \eqref{eq:G_N_1}-\eqref{eq:G_N_4} holds for 
$-t_k\leq t \leq 0$.
Denote $\tilde{G}^{-k}$ the time-shifted version of $\tilde{G}^{k}$ by $t_{k-1}$.
By time invariance, proving \eqref{eq:  Istk averaged expectation}  is equivalent to showing
\begin{align}
\bbE \left[ \tilde G^{-k}    I(t_1) \right] = 
\bbE \left[ \tilde G^{-k}  \widehat I(t_1)  
\right] .
\label{eq:time_shift_G_tilde_2}
\end{align}
We start by computing the RHS of \eqref{eq:time_shift_G_tilde_2}.
By decomposing the configuration at time 0 according to the set $A_{\delta_v,\delta_t}(0)$ which depends only on $Z_0 = Z$, we get
\begin{align*}
\bbE&\left[ \tilde G^{-k} I(t_1) \right]
=\frac{1}{\mathcal{Z}}\sum_{M'=0}^\infty \frac{N^{M'}}{M'!}\int dZ \int dz_{1:M'}\rho_{M'}(Z,z_{1:M'}) 
\, \tilde G^{-k}    I(t_1)  \mathds{1}_{\mathcal{G}_N^{-(k-1)}(\delta)} \\
&\qquad=\frac{1}{\mathcal{Z}}\sum_{M'=0}^\infty \frac{N^{M'}}{M'!}\sum_{m=0}^{M'}\left(\begin{array}{c}M'\\m\end{array}\right)\int dZ \int dz_{1:M'}\rho_{M'}(Z,z_{1:M'})\left(\prod_{i=1}^{M'-m}\mathds{1}_{z^i}\notin A_{\delta_v,\delta_t}(0) \right)\\
&\qquad\qquad \qquad\qquad \times\left(\prod_{i=M'-m+1}^{M'}\mathds{1}_{z^i \in A_{\delta_v,\delta_t}(0)}\right)
\, \tilde G^{-k}    I(t_1)  \mathds{1}_{\mathcal{G}_N^{-(k-1)}(\delta)} .
\end{align*}
By construction $I(t_k)$ is non zero for configurations in $\mathcal{G}^{k+1}_N(\delta) \subset \mathcal{G}^{k-1}_N(\delta)$ (see \eqref{eq: def Istk}). Thus it was legitimate to add the indicator $\mathds{1}_{\mathcal{G}_N^{-(k-1)}(\delta)}$ in the equations above.
Note that in $\mathcal{G}^{k-1}_N(\delta)$, the entire dynamics for $[-t_{k-1},0]$ (and therefore $\tilde G^{-k}$) is independent of particles $i$ such that $z^i \in A_{\delta_v,\delta_t}(0)$, as formalized in Lemma \ref{lem: set A} (i).
Likewise, $\mathds{1}_{\mathcal{G}_N^{-(k-1)}(\delta)} $ is independent of $z_{M'-m+1:M'}$, thus by abuse of notation we can consider this indicator as a function of $z_{1:M'-m}$. Through a direct change of variables $M'=M+m$, and separating the integrals, we get
\begin{align}
\bbE\left[  \tilde G^{-k}    I(t_1)  \right]
= & \frac{1}{\mathcal{Z}}\sum_{M=0}^\infty \frac{N^{M}}{M!}\sum_{m=0}^{\infty}\frac{N^m}{m!}\int dZ \int dz_{1:M}\rho_{M}(Z,z_{1:M})\left(\prod_{i=1}^{M}\mathds{1}_{z^i \notin A_{\delta_v,\delta_t}(0)}\right)
\mathds{1}_{\mathcal{G}_N^{-(k-1)}(\delta)} \tilde G^{-k}   
\nonumber\\
& \times\int dz_{M+1:M+m}\rho_{m}(Z,z_{M+1:M+m})\left(\prod_{i=M+1}^{M+m}\mathds{1}_{z^i \in A_{\delta_v,\delta_t}(0 )}\right) \,  I(t_1, Z,z_{1:M+m}) .
\label{eq:jaiseparef}
\end{align}
For the same reason, the particles in $A_{\delta_v,\delta_t}(0)$ do not interact with the tagged particle at time $0$, so that the measure $\rho_m \left(Z,z_{M+1:M+m}\right)$ in the calculations above is simply a product of Gaussian measures.
We get 
\begin{align}
\bbE&\left[  \tilde G^{-k}    I(t_1)    \right]
=\frac{1}{\mathcal{Z}}\sum_{M=0}^\infty \frac{N^{M}}{M!}
\int dZ \int dz_{1:M}\rho_{M}(Z,z_{1:M})\left(\prod_{i=1}^{M}\mathds{1}_{z^i \notin A_{\delta_v,\delta_t}(0)}\right)
\mathds{1}_{\mathcal{G}_N^{-(k-1)}(\delta)} \tilde G^{-k}   
\nonumber\\
&\qquad\qquad \qquad\qquad  \times   \sum_{m \geq 1}^\infty \frac{N^m}{m!}
\int dz_{M+1:M+m} \prod_{i=M+1}^{M+m} \gamma(v_i) 
\mathds{1}_{z^i \in A_{\delta_v,\delta_t}(0)}
 I(t_1, Z,z_{1:M+m}) \nonumber \\
 &\qquad=\frac{1}{\mathcal{Z}}\sum_{M=0}^\infty \frac{N^{M}}{M!}
\int dZ \int dz_{1:M} \rho_{M}(Z,z_{1:M})
\left(\prod_{i=1}^{M}\mathds{1}_{z^i \notin A_{\delta_v,\delta_t}(0)}\right)
 \tilde G^{-k}   \, 
\widehat I(t_k)\mathcal{\hat Z} (Z,z_{1:M}) ,
\label{eq:jaiseparef 2}
\end{align}
where we used the definition \eqref{eq:  Istk averaged}  of 
$\widehat I(t_k)$.
Replacing  $\mathcal{\hat Z}$ by its value \eqref{eq:  Istk averaged hat Z}, one can rebuild the full 
expectation and conclude the derivation of \eqref{eq:time_shift_G_tilde_2}. Thus the proof of the Lemma \ref{lem:mart_I} is complete.
\end{proof}
\vspace{0.3cm}


\noindent\textbf{Step 2.2: Splitting the sum over the fast background particles.} 

Let us now consider 
\begin{align}
\Afodd_T =  \sum_{\scriptsize{\begin{array}{l}k=0,\\ k\text{ odd}\end{array}}}^{K} I(t_k)
= \sum_{\scriptsize{\begin{array}{l}k=0,\\ k\text{ odd}\end{array}}}^{K} \left(I(t_k)-\widehat I(t_k)\right) 
+ \sum_{\scriptsize{\begin{array}{l}k=0,\\ k\text{ odd}\end{array}}}^{K} \widehat I(t_k) 
\label{eq:breakitdown_0}.
\end{align}
As shown in Lemma~\ref{lem:mart_I}, the term $\sum I(t_k)-\widehat I(t_k)$ has essentially a martingale structure, as it is a sum of almost independent terms with mean $0$. 
In Step 2.3, we will deal with this term by adapting the standard Hoeffding's argument. The following lemma takes care of the remainder.


\begin{lemma}
\label{lemma: control_moy_pas_mart}
There exists $C_{2,F}>0$ (independent of $N$) such that
\begin{align}
\mathbb{P}\left[\Big|\sum_{\scriptsize{\begin{array}{l}k=0,\\ k\text{ odd}\end{array}}}^{K} \widehat I(t_k) \Big|\geq C_{2,F}\frac{T\delta_t^2}{N}\right]
=O\left(e^{-N^{\delta/4}/4}\right).\label{eq:control_moy_pas_mart}
\end{align}
\end{lemma}


\begin{proof}
It is enough to show that for any $k$ odd  the following holds
\begin{align}
    \mathbb{P}\left[\left|  \widehat I(t_k)  \right|\geq C_{2,F}\frac{\delta_t^3}{N}\right]=O\left(e^{-N^{\delta/4}/4}\right).
    \label{eq:control_moy_pas_mart separes}
\end{align}
Summing then over $K = T/\delta_t$ terms, \eqref{eq:control_moy_pas_mart} can be obtained by  the previous bound. We turn now to the derivation of \eqref{eq:control_moy_pas_mart separes}.
By definition \eqref{eq:  Istk averaged} of $\widehat I(t_k)$ and \eqref{eq: def Istk} of $I(t_k)$,
we have
\begin{align}
    \widehat I(t_k)=\frac{\mathds{1}_{\mathcal{G}_N^{k-1}(\delta)}}{\mathcal{\hat Z}}  \sum_{m \geq 1}^\infty & \frac{N^m}{m!}\int dz_{M+1:M+m} \prod_{i=M+1}^{M+m} \gamma(v_i)  \mathds{1}_{z^i_{t_{k-1}}\in A_{\delta_v,\delta_t}(t_{k-1})} \nonumber \\
 & \qquad \qquad \times    \int_{t_k}^{t_{k+1}} \frac{\mathds{1}_{\mathcal{G}^{k+1}_N(\delta)}}{N}
 \sum_{i=M+1}^{M+m}F(X_s-x^i_s)\mathds{1}_{|v^i_s-V_s|\geq 2\delta_v}ds,
 \label{eq: version symetrique}
 \end{align}
where we used that under the condition $\mathcal{G}^{k+1}_N (\delta)$, any particle involved in $I(t_k)$ belongs to $A_{\delta_v,\delta_t}(t_{k-1})$. From the exchangeability of the background particles 
\begin{align}
\widehat I(t_k)  =&\frac{\mathds{1}_{\mathcal{G}_N^{k-1}(\delta)}}{\mathcal{\hat Z}}  \sum_{m \geq 1}^\infty \frac{N^{m-1}}{(m-1)!}\int dz_{M+1:M+m} \prod_{i=M+1}^{M+m} \gamma(v_i)  \mathds{1}_{z^i_{t_{k-1}}\in A_{\delta_v,\delta_t}(t_{k-1})} 
\nonumber \\
& \qquad \qquad \times  
\int_{t_k}^{t_{k+1}}\mathds{1}_{\mathcal{G}^{k+1}_N (\delta)}F(X_s-x^{M+1}_s)\mathds{1}_{|v^{M+1}_s-V_s|\geq 2\delta_v}ds.
\label{eq: version asymetrique}
\end{align}
First, since $z^{M+1}_{t_{k-1}}\in A_{\delta_v,\delta_t}(t_{k-1})$  ensures that there hasn't been any interaction in the time frame $[0,t_{k-1}]$ and since $\delta_t=O\left(N^{\frac{1}{5}}\right)$, Lemma~\ref{lem:influence_une_particule} yields (denoting here $\overline{Z}$ the process in which particle $M+1$ has been removed) 
\begin{align}
\mathds{1}_{\mathcal{G}^{k+1}_N(\delta)} &\left| \int_{t_k}^{t_{k+1}}F(X_s-x^{M+1}_s)\mathds{1}_{|v^{M+1}_s-V_s|\geq 2\delta_v}ds 
- \int_{t_k}^{t_{k+1}}F(\overline{X}_s-x_{M+1}-sv_{M+1})\mathds{1}_{|v^{M+1}_s-V_s|\geq 2\delta_v}ds\right|\nonumber\\
&\hspace{3cm}=O\left(\frac{\delta_t^2}{N}\mathds{1}_{\mathcal{G}^{k+1}_N(\delta)}\int_{t_{k}}^{t_{k+1}}\mathds{1}_{x_{M+1}\in\mathcal{B}(\overline{X}_s-sv_{M+1},2R)}\right).
\label{eq:on_peut_moyenner_1}
\end{align}
This could a priori allow us to integrate over $x_{M+1}$, if it were not for the indicator function on the velocity. Our goal is therefore to replace $\mathds{1}_{|v^{M+1}_s-V_s|\geq 2\delta_v}$ by $\mathds{1}_{|v_{M+1}-\overline{V}_s|\geq 2\delta_v}$. Using \eqref{eq:est_diff_ordre_1} and the fact that $\delta_t=O(\delta_v^{-1})$ with $\delta_v^2\gg N^{-1}$, we know that for configurations in $\mathcal{G}^{k+1}_N(\delta)$, 
 there exists some universal constant $C$ such that (for $N$ large enough) 
\begin{align*}
\mathds{1}_{ |v_{M+1}-\overline{V}_s | \geq 2\delta_v + \frac{C}{N\delta_v}}
\leq  \mathds{1}_{|v^{M+1}_s-V_s|\geq 2\delta_v} 
\leq \mathds{1}_{ |v_{M+1}-\overline{V}_s| \geq 2\delta_v - \frac{C}{N\delta_v}},
\end{align*}
so that 
\begin{align*}
0 \leq  \mathds{1}_{|v^{M+1}_s-V_s|\geq 2\delta_v} - \mathds{1}_{ |v_{M+1}-\overline{V}_s | \geq 2\delta_v + \frac{C}{N\delta_v}}
\leq \mathds{1}_{|v_{M+1}-\overline{V}_s|\in\left[2\delta_v-\frac{C}{N\delta_v}, 2\delta_v+\frac{C}{N\delta_v}\right]}.
\end{align*}
 As a consequence
\begin{align}
& \mathds{1}_{\mathcal{G}^{k+1}_N(\delta)}  \mathds{1}_{z^{M+1}_{t_{k-1}}\in A_{\delta_v,\delta_t}(t_{k-1})}\int_{t_k}^{t_{k+1}}F(\overline{X}_s-x_{M+1}-sv_{M+1})\mathds{1}_{|v^{M+1}_s-V_s|\geq 2\delta_v}ds\nonumber\\
& =\mathds{1}_{\mathcal{G}^{k+1}_N(\delta)}\mathds{1}_{z^{M+1}_{t_{k-1}}\in A_{\delta_v,\delta_t}(t_{k-1})}\int_{t_k}^{t_{k+1}}F(\overline{X}_s-x_{M+1}-sv_{M+1})\mathds{1}_{|v_{M+1}-\overline{V}_s|\geq 2\delta_v+\frac{C}{N\delta_v}}ds 
+O  \left(
 \mathds{1}_{\mathcal{G}^{k+1}_N(\delta)} \mathcal{E}_1 \right),
\label{eq:on_peut_moyenner_2}
\end{align}
with the error term 
\begin{align}
\mathcal{E}_1 = \mathds{1}_{z^{M+1}_{t_{k-1}}\in A_{\delta_v,\delta_t}(t_{k-1})}
\int_{t_{k}}^{t_{k+1}}\mathds{1}_{|v_{M+1}-\overline{V}_s|\in\left[2\delta_v-\frac{C}{N\delta_v}, 2\delta_v+\frac{C}{N\delta_v}\right]}\mathds{1}_{x_{M+1}\in\mathcal{B}(\overline{X}_s-sv_{M+1},R)}ds .\label{eq:on_peut_moyenner_2 error}
\end{align}
We estimate first the error term after integrating over  the background particles.
For this, we  note that 
$$
\mathds{1}_{\mathcal{G}^{k+1}_N(\delta)} \mathcal{E}_1 \leq \mathds{1}_{\mathcal{G}^{k-1}_N(\delta)} \mathcal{E}_1.
$$ 
Integrating the LHS, we are going to show that  uniformly over the configurations $\{ Z (t_{k-1}), z_{1:M}(t_{k-1}) \}$ in $\mathcal{G}^{k-1}_N(\delta)$
\begin{align}
 \frac{1}{\mathcal{\hat Z}}  \sum_{m \geq 1}^\infty \frac{N^{m-1}}{(m-1)!}
 & \int dz_{M+1:M+m} \prod_{i=M+1}^{M+m} \gamma(v_i)  \mathds{1}_{z^i_{t_{k-1}}\in A_{\delta_v,\delta_t}(t_{k-1})} \mathds{1}_{\mathcal{G}^{k-1}_N(\delta)} \mathcal{E}_1  
 = O\left( \delta_t \frac{\delta_v^{d-2}}{N}\right) .
 \label{eq:on_peut_moyenner_4}
\end{align}
For this, first integrate the velocity $v_{M+1}$ of the background particle in $\mathcal{E}_1$. This leads to  
\begin{align*}
\int dv \gamma(v)\int dx \int_{t_{k}}^{t_{k+1}} &  \mathds{1}_{|v-\overline{V}_s|\in\left[2\delta_v-\frac{C}{N\delta_v}, 2\delta_v
+\frac{C}{N\delta_v}\right]}\mathds{1}_{x \in \mathcal{B}(\overline{X}_s-sv_{M+1},R)}\nonumber\\
& = 
O\left(\int_{t_{k}}^{t_{k+1}} \int dv \gamma(v) \mathds{1}_{|v-\overline{V}_s|\in\left[2\delta_v-\frac{C}{N\delta_v}, 2\delta_v+\frac{C}{N\delta_v}\right]}\right)
= O\left( \delta_t \frac{\delta_v^{d-2}}{N}\right) ,
\end{align*}
Then to complete \eqref{eq:on_peut_moyenner_4},  we use that under the constraint  $\mathcal{G}^{k-1}_N(\delta)$ the variables in the integral \eqref{eq:on_peut_moyenner_4} are independent so that the factor $\mathcal{\hat Z}$ (see \eqref{eq:  Istk averaged hat Z}) is compensated.

The second error term in \eqref{eq:on_peut_moyenner_1} can be integrated in the same way and bounded from above by $O\left( \frac{ \delta_t^3}{N}\right)$.
Thus combining  \eqref{eq:on_peut_moyenner_1}, \eqref{eq:on_peut_moyenner_2} and \eqref{eq:on_peut_moyenner_4}, we have
\begin{align}
 \widehat I(t_k)
=&
\frac{\mathds{1}_{\mathcal{G}_N^{k-1}(\delta)}}{\mathcal{\hat Z}}  \sum_{m \geq 1}^\infty \frac{N^{m-1}}{(m-1)!}\int dz_{M+1:M+m} \prod_{i=M+1}^{M+m} \gamma(v_i)  \mathds{1}_{z^i_{t_{k-1}}\in A_{\delta_v,\delta_t}(t_{k-1})} \nonumber\\
&\qquad\qquad \times \int_{t_k}^{t_{k+1}}
\mathds{1}_{\mathcal{G}^{k+1}_N(\delta)}  F(\overline{X}_s-x_{M+1}-sv_{M+1})\mathds{1}_{|v_{M+1}-\overline{V}_s|\geq 2\delta_v}ds + O\left(\frac{\delta_t^3}{N}\right).
\label{eq:erreur_a_expliquer}
\end{align}
To conclude, we are going to integrate \eqref{eq:erreur_a_expliquer} over $x_{M+1}$ to use that $F$ has mean 0. However the set $\mathds{1}_{\mathcal{G}^{k+1}_N(\delta)}$ depends also on $x_{M+1}$ and to remove this correlation, we add an additional error term by writing $\mathds{1}_{\mathcal{G}_N^{k+1}(\delta)}= 1- \mathds{1}_{\overline{\mathcal{G}_N^{k+1}}(\delta)}$: 
\begin{align}
\widehat I(t_k)
=&  \frac{\mathds{1}_{\mathcal{G}_N^{k-1}(\delta)}}{\mathcal{\hat Z}}  \sum_{m \geq 1}^\infty \frac{N^{m-1}}{(m-1)!}\int dz_{M+2:M+m} \prod_{i=M+2}^{M+m} \gamma(v_i)  \mathds{1}_{z^i_{t_{k-1}}\in A_{\delta_v,\delta_t}(t_{k-1})} \nonumber\\
&\qquad\qquad \int dz_{M+1}\gamma(v_{M+1})\int_{t_k}^{t_{k+1}}F(\overline{X}_s-x_{M+1}-sv_{M+1})\mathds{1}_{|v_{M+1}-\overline{V}_s|\geq 2\delta_v}ds\label{eq:breakitdown_1}\\
&+O\left( \frac{\mathds{1}_{\mathcal{G}^{k-1}_N(\delta)}}{\mathcal{\hat Z}}  
\sum_{ m \geq 1}^\infty \frac{N^{m-1}}{(m-1)!}\int dz_{M+1:M+m} \prod_{i=M+1}^{M+m} \gamma(v_i)  \mathds{1}_{z^i_{t_{k-1}}\in A_{\delta_v,\delta_t}(t_{k-1})} \right.\nonumber\\
&\hspace{5cm}
\left.\int_{t_k}^{t_{k+1}}ds \, \mathds{1}_{x_{M+1}\in\mathcal{B}(\overline{X}_s-sv_{M+1},R)}\mathds{1}_{\overline{\mathcal{G}^{k+1}_N}(\delta)}\right)\label{eq:breakitdown_2}\\
&+ O\left(\frac{\delta_t^3}{N}\right) .
\label{eq:breakitdown_3}
\end{align}
We have used that, by definition of $A_{\delta_v,\delta_t}(t_{k-1})$, we have
\begin{align*}
\left\{(x,v)\text{ s.t. } |v-\overline{V}_s|\geq 2\delta_v,\ \exists u\in[t_k,t_{k+1}],\ x\in \mathcal{B}(\overline{X}_u-uv,R)\right\}\subset A_{\delta_v,\delta_t}(t_{k-1}).
\end{align*}
Therefore, integrating over $x_{M+1}$ and using the mean 0 of $F$, we get that \eqref{eq:breakitdown_1} is $0$. 
To evaluate \eqref{eq:breakitdown_2}, one can bound $\mathds{1}_{x_{M+1}\in\mathcal{B}(\overline{X}_s-sv_{M+1},R)}$ by $1$ to get
\begin{align*}
\eqref{eq:breakitdown_2}
\leq& \delta_t\frac{\mathds{1}_{\mathcal{G}^{k-1}_N(\delta)}}{\mathcal{\hat Z}}  
\sum_{ m \geq 1}^\infty \frac{N^{m}}{m!}\int dz_{M+1:M+m} \frac{m}{N}\prod_{i=M+1}^{M+m} \gamma(v_i)  \mathds{1}_{z^i_{t_{k-1}}\in A_{\delta_v,\delta_t}(t_{k-1})}\mathds{1}_{\overline{\mathcal{G}^{k+1}_N}(\delta)},
\end{align*}
Then, taking the expectation of \eqref{eq:breakitdown_2} (i.e. integrating over the remaining particles) and proceeding as in \eqref{eq:jaiseparef 2} to evaluate the conditional expectation, we find that 
$$
\mathbb{E}\left[\eqref{eq:breakitdown_2}\right] 
\leq \delta_t \, \mathbb{E}\left[ \frac{\mathcal{N}}{N}
\,  \mathds{1}_{\overline{\mathcal{G}^{k+1}_N}(\delta)} \right] =O\left(e^{-N^{\delta/4}/2}\right) ,
$$ 
using Proposition~\ref{prop:good_set_good}. 
Thanks to Markov's inequality, we find that $\eqref{eq:breakitdown_2}=O\left(\delta_t^3/N\right)$ with high probability (i.e. exponentially small in $N$ probability of not occurring).
This completes the derivation of \eqref{eq:control_moy_pas_mart separes} and thus of Lemma \ref{lemma: control_moy_pas_mart}.
\end{proof}
This ends this step, and we now turn our attention to the term $\sum I(t_k)-\widehat I(t_k)$ in \eqref{eq:breakitdown_0}.\\


\noindent\textbf{Step 2.3: Concentration estimates.} 

For $\lambda \geq 0$ and $l\in\{1,... ,K\}$ odd, we recall the notation $\tilde G^k$ \eqref{eq:  G tilde}. 
We proceed recursively to estimate $\bbE\left[ \exp \left( \lambda\sum I(t_k)-\widehat I(t_k) \right) \right] = \bbE[ \tilde G^{K+2} ]$.
By construction of the good set $\mathcal{G}^k_N(\delta)$, the property \eqref{eq:G_N_5} (with $M_N=2\delta_v$) implies 
a first estimate on the integral $I(t_k)-\widehat I(t_k)$ introduced in \eqref{eq: def Istk}-\eqref{eq:  Istk averaged}
\begin{align}
\label{eq:  Istk upper bound}
\big| I(t_k)-\widehat I(t_k) \big|
\leq 
2\left|  \frac{\mathds{1}_{\mathcal{G}^{k+1}_N(\delta)}}{N}\sum_{i }   \int_{t_k}^{t_{k+1}} 
F(X_s - x^i_s) \mathds{1}_{|v^i_s -  V_s | \geq  2\delta_v}
ds \right|
\leq C \frac{\delta_t}{N^{\frac{1-\delta}{2}}},
\end{align}
for some explicit constant $C>0$ which  depends only on $R$, $d$ and $F$. Using the inequality $e^x \leq x + e^{x^2}$ for all $x\in\mathbb{R}$, we get
\begin{align*}
\bbE \left[ \tilde G^{\ell +2} \right] 
=& \bbE \left[ \tilde G^\ell \exp \left( \lambda  \left(I(t_\ell)-\widehat I(t_\ell)\right) \right) \right]
\leq \lambda \bbE \left[ \tilde G^\ell \left(I(t_\ell)-\widehat I(t_\ell)\right) \right]
+ 
e^{C^2\frac{\lambda^2\delta_t^2}{N^{1-\delta}}} \bbE \left[ \tilde G^\ell  \right]\\
\leq &e^{C^2\frac{\lambda^2\delta_t^2}{N^{1-\delta}}} \bbE \left[ \tilde G^\ell \right],
\end{align*}
where we used \eqref{eq:  Istk averaged expectation} to cancel the linear term. By iteration, we get
\begin{align*}
\bbE  \big[ \tilde G^{K+2}  \big]  & \leq \exp\left(\frac{C^2\lambda^2 T\delta_t}{N^{1-\delta}}\right).
\end{align*}
This upper bound on the exponential moment implies the concentration estimate by Markov's inequality 
\begin{align}
\bbP  \left[ \Big|\sum_{\scriptsize{\begin{array}{l}k=0,\\ k\text{ odd}\end{array}}}^{K} I(t_k)-\widehat I(t_k) \Big| > N^{\frac{\delta}{2}}\sqrt{\frac{T\delta_t}{N^{1-\delta}}} \right]
\leq&  2\exp \left( - \lambda N^{\frac{\delta}{2}}\sqrt{\frac{T\delta_t}{N^{1-\delta}}}\right) \bbE \big[ \tilde G^{K+2} \big]
=O\left( e^{ - N^{\delta/4}}\right)
\label{eq:control_hoeff} ,
\end{align}
by choosing $\lambda = \frac{N^{\frac{\delta}{2}}}{2C} \sqrt{\frac{N^{1-\delta}}{T\delta_t}}$. Using \eqref{eq:breakitdown_0} with both \eqref{eq:control_moy_pas_mart} and \eqref{eq:control_hoeff}, we get 
\begin{equation}\label{eq:boot_control_rapide_impair}
\mathbb{P}\left[\left|\Afodd_T\right|
\geq {C_{2,F} }\frac{T\delta_t^2}{N}
+  N^{\frac{\delta}{2}}\sqrt{\frac{T\delta_t}{N^{1-\delta}}} \right]=O\left(\exp\left(-N^{\delta/4}\right)\right).
\end{equation}
The same result holds for $\Afeven_T$.
\vspace{0.3cm}


\noindent\textbf{Step 3: Slow background particles.} Let us now show that 
\begin{equation}
\label{eq:boot_control_lent}
\mathbb{P}\left[
|\As_T| \geq \frac{T \delta_v^{\frac{d}{2}} }{N^\frac{1-\delta}{2}}\right]
=
\mathbb{P}\left[\left|\int_{0}^T\frac{1}{N}\sum_{i}F(X_t-x^i_t)\mathds{1}_{|v^i_t-V_t|< 2\delta_v}dt\right|
\geq \frac{T \delta_v^{\frac{d}{2}}}{N^\frac{1-\delta}{2}}\right]
=O\left(e^{-N^{\delta/4}}\right).
\end{equation}
This is a direct consequence of \eqref{eq:control_drift_avec_contrainte}, as 
\begin{equation*}
\mathbb{P}\left[\exists t\in[0,N],\quad \left|\frac{1}{N}\sum_{i}F(X_t-x^i_t)\mathds{1}_{|v^i_t-V_t|< 2\delta_v}dt\right|\geq \frac{\delta_v^{\frac{d}{2}}}{N^\frac{1-\delta}{2}}\right]=O\left(e^{-N^{\delta/4}}\right).
\end{equation*}
Here we use the assumption $\delta_v\geq N^{-\frac{1}{d}}$ from \eqref{eq:borne_delta_v_init}.\\


\noindent\textbf{Step 4: Conclusion.}  
Recall that $T_f=N^{ \frac{1}{2} -\delta} \delta_v /2$ by \eqref{eq: parameters}.
Plugging \eqref{eq:borne_proba_erreur}, \eqref{eq:boot_control_rapide_impair} and \eqref{eq:boot_control_lent} back into \eqref{eq: decomposition AfT et AsT}, we get, with probability greater than $1-e^{-N^{\delta/8}}$, that for $T \leq T_f$
\begin{align*}
\left|\int_0^{T}\frac{1}{N}\sum_{i}F(X_t-x^i_t)dt\right|
 = O\left(\frac{T}{N\delta_v^2}+\sqrt{\frac{T}{\delta_vN^{1-2\delta}}}+\frac{T}{N^\frac{1-\delta}{2}}\delta_v^{\frac{d}{2}}\right)
 = O\left(\sqrt{\frac{T}{\delta_vN^{1-2\delta}}}+\frac{T}{N^\frac{1-\delta}{2}}\delta_v^{\frac{d}{2}}\right),
\end{align*}
where the contribution $\frac{T}{N\delta_v^2}$ could be neglected as 
 $T \leq T_f \leq N^{1+ 2 \delta}\delta_v^3$ from 
the assumption 
$\delta_v \gg N^{-1/4}$ \eqref{eq:borne_delta_v_init}.
The two terms above quantify the errors due to the fluctuations of the fast background particles and the error  from the slow particles.
If $T_f\leq N^\delta \delta_v^{-(d+1)}$
then the fluctuations are the main error term, i.e.
$\sqrt{\frac{T}{\delta_vN^{1-2\delta}}}\geq \frac{T}{N^\frac{1-\delta}{2}}\delta_v^{\frac{d}{2}}$.
Recalling \eqref{eq: parameters},
this condition is satisfied with the following choice
\begin{align*}
\delta_v= N^{-\frac{1}{2(d+2)}}
= N^{- 2 \alpha + 2 \delta},\qquad T_f=
N^{\frac{d+1}{2(d+2)}-\delta} = N^\beta,
\end{align*}
with the parameters 
$\alpha, \beta$ as in \eqref{eq: sub-optimal beta, alpha}
(note that $\delta_v$ satisfies \eqref{eq:borne_delta_v_init} for $d\geq3$). 
We thus get
\begin{align*}
T \leq T_f, \qquad 
\left|\int_0^{T}\frac{1}{N}\sum_{i}F(X_t-x^i_t)dt\right|=O\left(\sqrt{\frac{T}{\delta_vN^{1-2\delta}}}\right)=O\left(\sqrt{\frac{T}{N}}N^{\alpha}\right).
\end{align*}
Note that we technically assumed that $T\geq \delta_t$ in order to derive the inequality above. 
For $T\leq \delta_t$, the same result holds as we simply use that in the good set
\begin{align*}
\left|\int_0^{T}\frac{1}{N}\sum_{i}F(X_t-x^i_t)dt\right|
=O\left(\frac{T}{N^{(1-\delta)/2}}\right)
=O\left(\sqrt{\frac{T}{N}}N^{\alpha}\right).
\end{align*} 
By time invariance, we obtain that there exist $C,\delta>0$ such that for all $s,t\in[-2N,2N]$ satisfying $|t-s|\leq N^{\beta}$ we have
\begin{align*}
\mathbb{P}\left[\left|\int_s^{t}\frac{1}{N}\sum_{i}F(X_u-x^i_u)du\right|\geq C\sqrt{\frac{|t-s|}{N}}N^\alpha\right]=O\left(e^{-N^{\delta/8}}\right).
\end{align*}
By considering a sequence of time discretization $(s_j)_{j}$ and $(t_k)_k$ with $s_j=-2N+\frac{j}{N}$ and  $t_k=-2N+\frac{k}{N}$, and by using $\mathbb{P}\left[\exists s,t\in [-2N,2N],\ \left|\int_s^{t}\frac{1}{N}\sum_{i}F(X_t-x^i_t)dt\right|\geq C\frac{|t-s|}{N^{\frac{1-\delta}{2}}}\right]=O\left(e^{-N^{\delta/8}}\right)$, we may obtain that there exist $C>0$ such that
\begin{align*}
\mathbb{P}\left[\exists s,t\in [-2N,2N],\ \left|\int_s^{t}\frac{1}{N}\sum_{i}F(X_u-x^i_u)du\right|\geq C\sqrt{\frac{|t-s|}{N}}N^\alpha\right]=O\left(e^{-N^{\delta/16}}\right).
\end{align*}
Up to changing the factor $\delta$, the constant $C$ can be removed as $\alpha = \frac{1}{4(d+2)}+\delta$.

Finally \eqref{eq:control_boot_avec_contrainte_sup} is then a consequence of the exact same proof: note that the trick done in \eqref{eq:on_peut_moyenner_1} and \eqref{eq:on_peut_moyenner_2}-\eqref{eq:on_peut_moyenner_4} to use the mean $0$ of $F$ still works\footnote{To use the same trick as \eqref{eq:on_peut_moyenner_2}-\eqref{eq:on_peut_moyenner_4} in order to remove the influence of particle $M+1$ from the indicator function, one rigorously needs $(N\delta_v)^{-1}\ll M_N\ll 1$. The lower bound however is not necessary, as for $M_N\lesssim (N\delta_v)^{-1}$ the constraint $|v-\overline{V}_s|\in\left[M_N-\frac{2C}{N\delta_v}, M_N+\frac{2C}{N\delta_v}\right]$ in particular implies $|v-\overline{V}_s|\leq M_N+\frac{2C}{N\delta_v}$, which is of small enough probability. }, and that otherwise we only use $F$ bounded. Hence the result.
\end{proof}


\begin{proof}[Proof of Proposition~\ref{prop:boot_beta}]
The proof of Proposition~\ref{prop:boot_beta} follows closely the one of Proposition~\ref{prop:init_bootstrap}. Being in a better set $\mathcal{G}_N(\delta,\alpha, \beta)$ improves the estimates (in particular on the recollision time and on the almost sure bound of the martingale increments) and thus allows to bootstrap the argument up to adjusting the parameters. We insist that, in this proof, we consider the set of initial configurations $\mathcal{G}_N(\delta,\alpha, \beta)$ instead of $\mathcal{G}_N(\delta)$.

Again, let $F$ be $\partial_{\mathcal{I}} \Phi$ for some multi-indices $\mathcal{I}\in\{1,...,d\}^{|\mathcal{I}|}$ with $|\mathcal{I}|\in\{1,2,3 \}$ and let 
\begin{equation}\label{eq:delta_v_boot}
\delta_v=N^{-\frac{1}{d+4}}.
\end{equation}
This specific value for $\delta_v$ is motivated by later calculations. We then define similarly as before
\begin{equation}
\label{eq: new Tf}
T_f=\frac{N^{\frac{1+\beta}{2}-\alpha}\delta_v}{5}, \quad \delta_t=10\times \frac{6R}{\delta_v},\quad t_k=k\delta_t,\quad K=\left\lfloor\frac{T}{\delta_t}\right\rfloor.
\end{equation}
Note that, using \eqref{eq:cond_beta}
\begin{itemize}
\item For this given value of $\delta_v$, the choice for $T_f$ given in \eqref{eq: new Tf} is such that a particle interacting at time $0$ does not re-interact before time $T_f$. 
\item The time frame $\delta_t$ is small enough that the various controls on $\mathcal{G}_N(\delta,\alpha, \beta)$ and Lemma~\ref{Lem: Better set estimates} are valid.
\item We have $\delta_v\geq N^{-\frac{1-2\alpha}{3}+\delta}$, as well as $\delta_v\geq N^{-1/d}$, ensuring that \eqref{eq:control_drift_avec_contrainte} holds.
\end{itemize}
We now consider $T\leq T_f$.
Finally, recall the definition of $A_{\delta_v, \delta_t}$ from the proof of Proposition~\ref{prop:init_bootstrap}. In the set of initial conditions $\mathcal{G}_N(\delta,\alpha, \beta)$, for any $t\in[0,T]$, we have
\begin{align*}
\sup_{s\in[t,t+2\delta_t]}|X_t+(s-t)V_t-X_s|=&O\left(\frac{\delta_t^{\frac{3}{2}}}{N^{\frac{1}{2}-\alpha}}\right)=O\left(\frac{1}{N^{\frac{3}{2}\delta}}\right)\text{ since $\delta_t\leq N^{\frac{1-2\alpha}{3}-\delta}$},\\
\sup_{s\in[t,t+2\delta_t]}|x^1_t+(s-t)v^1_t-x^1_s|=&O\left(\frac{\delta_t^2}{N}\right)=O\left(\frac{1}{N^{\frac{1}{3}}}\right)\text{ since $\delta_t\leq N^{\frac{1}{3}}$}\\
\text{and}
\sup_{s\in[t,t+2\delta_t]} \big|\left|v^1_t-V_t\right|-\left|v^1_s-V_s\right| \big|
=&O\left(\frac{\delta_t^{\frac{1}{2}}}{N^{\frac{1}{2}-\alpha}}\right)=O\left(\frac{1}{N^{\frac{1-2\alpha}{3}+\frac{\delta}{2}}}\right)\ll N^{-\frac{1-2\alpha}{3}+\delta}\leq \delta_v.
\end{align*}
Note that all the right hand side terms not only tend to $0$ as $N$ goes to infinity, but the latter is also of a smaller order of magnitude than $\delta_v$. Therefore, if there exists $t\in[0,T]$ such that $x^1_t\in\mathcal{B}\left(X_t,R\right)$ and $|v^1_t-V_t|\geq 2\delta_v$, then, for $N$ large enough, there exists $k\in\{0,...,K\}$ such that $(x^1_{t_k},v^1_{t_k})\in A_{\delta_v, \delta_t}(t_k)$. Similarly, any background particle satisfying $(x^i_{t_k},v^i_{t_k})\notin A_{\delta_v, \delta_t}(t_k)$ and $x^i_t\in\mathcal{B}\left(X_t,R\right)$ for some $t\in[t_{k+1},t_{k+2}]$ must satisfy $|v^1_t-V_t|<2\delta_v$. Using the decomposition \eqref{eq: decomposition AfT et AsT}, we write
\begin{align*}
\int_0^{T}\frac{1}{N}\sum_{i}F(X_t-x^i_t)dt
= \Af_T + \As_T+\mathcal{E}^G,
\end{align*}
 where the indicator function of $\mathcal{G}^{k+1}_N(\delta)$ is replaced by the indicator of  $\mathcal{G}^{k+1}_N(\delta,\alpha, \beta)$ which is the set of initial configurations for which the various controls of $\mathcal{G}_N(\delta,\alpha, \beta)$ are restricted to the time interval $[0, t_{k+1}]$.
 Using again \eqref{eq:borne_proba_erreur} to estimate the error term and
 \eqref{eq:boot_control_lent}
 to control the slow particles, we get
\begin{equation}
\label{eq:boot_control_lent_2}
\mathbb{P}\left[\left|\mathcal{E}^G\right|\geq \frac{T\delta_t^2}{N}\right]=O\left(e^{-N^{\delta/8}}\right), 
\qquad 
\mathbb{P}\left[
|\As_T| \geq \frac{T \delta_v^{\frac{d}{2}} }{N^\frac{1-\delta}{2}}\right]
=O\left(\exp\left(-N^{\delta/4}\right)\right).
\end{equation}
The better set is only useful to improve the fluctuation estimates on the fast background particles by looking at larger timescales.

Define, similarly as \eqref{eq: def Istk}, \eqref{eq:  Istk averaged}, the terms $I(t_k),\widehat I(t_k)$. Thanks to \eqref{eq:control_boot_avec_contrainte_sup}, which we assume to hold for the specific choice $M_N=2\delta_v$, we know that 
\begin{align*}
\left|I(t_k) - \widehat I(t_k)\right|
= 2 \left|\frac{\mathds{1}_{\mathcal{G}^{k+1}_N(\delta,\alpha, \beta)}}{N} \sum_{i}\int_{t_k}^{t_{k+1}} ds  F(X_t-x^i_t)\mathds{1}_{|v^i_t-V_t|\geq 2\delta_v}
\right|\leq \sqrt{\frac{\delta_t}{N}} N^\alpha.
\end{align*}
This estimate is the counterpart to \eqref{eq:  Istk upper bound}. With the exact same calculations as in Proposition~\ref{prop:init_bootstrap} (using Lemma~\ref{lem:influence_une_particule} along with  $\delta_t=O\left(N^{\frac{1-2\alpha}{3}-\delta}\right)$ in order to compare $Z_t$ with $\overline{Z}_t$ on this longer time frame in the counterpart of \eqref{eq:on_peut_moyenner_1}) we obtain
\begin{equation}\label{eq:boot_control_rapide_impair_2}
\mathbb{P}\left[ \Af_T\geq C_{2,F}\frac{T\delta_t^2}{N}+ \sqrt{\frac{T}{N}} N^{\alpha+ \delta/2} \right]=O\left(\exp\left(-N^{\delta/4}\right)\right).
\end{equation}

\medskip

\noindent\textbf{Conclusion.}  Combining \eqref{eq:boot_control_rapide_impair_2} and \eqref{eq:boot_control_lent_2}, with probability greater than $1-e^{-N^{\delta_4}}$ for some $\delta_4>0$, we get
\begin{align}
\left|\int_0^{T}\frac{1}{N}\sum_{i}F(X_t-x^i_t)dt\right|=O\left(\frac{T}{N\delta_v^2}+\sqrt{\frac{T}{N}} N^{\alpha+ \delta/2} +\frac{T}{N^\frac{1-\delta}{2}}\delta_v^{\frac{d}{2}}\right)
=O\left(\sqrt{\frac{T}{N}} N^{\alpha+ \delta/2} \right),\label{eq:borne_T_hoeff_V2}
\end{align}
provided that $T\leq N^{\frac{d}{d+4}+2\alpha}$. We stress that  the value of $\delta_v$  in \eqref{eq:delta_v_boot} maximizes this upper bound on $T$ ensuring \eqref{eq:borne_T_hoeff_V2}. Recall that we should also have $T\leq T_f=\frac{N^{\frac{1+\beta}{2}-\alpha}\delta_v}{5}$,
and we thus choose
\begin{align*}
T\leq \min\left(N^{\frac{d}{d+4}+2\alpha}, N^{\frac{d+2}{2(d+4)}+\frac{\beta}{2}-\alpha}\right)N^{-\delta}=N^{-\delta}\left\{\begin{array}{ll}N^{\frac{d}{d+4}+2\alpha}&\text{ if }\beta\geq\frac{d-2}{d+4}+6\alpha\\N^{\frac{d+2}{2(d+4)}+\frac{\beta}{2}-\alpha}&\text{ else,}\end{array}\right.
\end{align*}
where this last additional $N^{-\delta}$ deals with the various universal constants. Under this constraint, we have 
\begin{align*}
\left|\int_0^{T}\frac{1}{N}\sum_{i}F(X_t-x^i_t)dt\right|=O\left(\sqrt{\frac{T}{N}}N^{\alpha+\frac{\delta}{2}}\right).
\end{align*}
We conclude in the same way  we concluded Proposition~\ref{prop:init_bootstrap}.
\end{proof}

%
%
%
%

\subsection{Probability of recollision}
\label{sec:recollision}

As we have mentioned several times so far, our controls are only valid provided there is no recollision. Denote $C^1_t$ the recollision event
\begin{align*}
C^1_t:=\left\{\exists s\in[0,N] \text{ s.t. }\  |s-t|\geq\frac{6R}{|v^1_t-V_t|} ,\quad  x^1_s\in\mathcal{B}(X_s,R)\right\}\cap\left\{x^1_t\in\mathcal{B}(X_t,R)\right\},
\end{align*}

 which translates the fact that, assuming at time $t$  particle $1$ is in the interaction ball, there exists a time $s$ such that $|s-t|\geq\frac{6R}{|v^1_t-V_t|}$ (which, as shown in Lemma~\ref{lem:max_inter_time}, ensures that the background particle has left the interaction ball) and $x^1_s\in\mathcal{B}(X_s,R)$.
We are going to evaluate the probability of $C^1_t$.


\begin{proposition}
\label{prop:proba_recol}
Assume $d\geq 4$ and fix
\begin{align*}
\gamma_r=\frac{1}{4}+\frac{1}{36}.
\end{align*}
Consider $\delta, \alpha^*, \beta^*$ defined in Proposition~\ref{prop:fin_bootstrap}. For any $\delta>0$ small enough, we have
\begin{equation}\label{eq:proba_recollision}
\hat{\mathbb{E}}_{\mathcal{G}_N(\delta,\alpha^*,\beta^*)}\left[\mathds{1}_{C^1_t}\mathds{1}_{|v^1_t-V_t|\geq N^{-\gamma_r}}\right]=O\left(N^{-\frac{19}{18}+2(d-1)\delta}\right).
\end{equation}
\end{proposition}


\begin{remark}
The exponent $\gamma_r$ is such that $d\gamma_r>1$, and thus the event $\{|v^1_t-V_t|\leq N^{-\gamma_r}\}$ is of small probability when compared to $N^{-1}$. We choose $\gamma_r=\frac{1}{4}+\frac{1}{36}$ in order to have some margin in the calculations. Furthermore, from $\alpha^*, \beta^*$ given in Proposition~\ref{prop:fin_bootstrap}, we have for $d\geq4$ and $\delta$ small enough
\begin{equation}
\label{eq:gamma_r_suff_petit}
    \gamma_r<\frac{1-2\alpha^*}{3}-\delta
  \quad \text{and} \quad   \gamma_r \leq \beta^* .
\end{equation}
\end{remark}


\begin{proof}[Proof of Proposition~\ref{prop:proba_recol}]
By time invariance, we may consider $t=0$ without loss of generality, and by time reversibility, we only deal with recollisions at future time $s>0$. Finally, throughout this proof, the indicator function $\mathds{1}_{\mathcal{G}_N(\delta,\alpha^*,\beta^*)}$ appears in every calculations. For the sake of readability, we shall omit it. 

Define for $\Delta_v\geq N^{-\gamma_r}$ the recollision time
\begin{align}
\label{def: trec def bis}
t_{rec}(\Delta_v):=\frac{1}{5}\min\left(N^{1-2\alpha}\Delta_v^2, N^{\frac{1+\beta^*}{2}-\alpha^*}\Delta_v\right).
\end{align}
By Lemma~\ref{lem:max_inter_time}, given $|v^1_0-V_0|\geq N^{-\gamma_r}$, the time $t_{rec}(|v^1_0-V_0|)$ is the first time the background particle $1$ may re-enter the interaction ball. We know that if $x^1_0\in\mathcal{B}(X_0,R)$ and $|v^1_0-V_0|\geq N^{-\gamma_r}$ (recall \eqref{eq:gamma_r_suff_petit}), then for any $s$ such that $\frac{6R}{|v^1_0-V_0|}\leq s\leq t_{rec}\left(|v^1_0-V_0|\right)$ we have $x^1_s\notin\mathcal{B}(X_s,\frac{3}{2}R)$.  

Now, let us divide the remaining interval ${[t_{rec}\left(|v^1_0-V_0|\right), T]}$ into subintervals of size $N^{-\alpha_r}$ with a small $\alpha_r>0$, $T\simeq N$ being the final time. We set the value
\begin{align}
\label{Def: alpha r}
\alpha_r= \frac{1}{36(d+1)} .
\end{align}
To evaluate the exit time, it is enough to probe the trajectories at the sampling times $t_j =jN^{-\alpha_r}$, with ${j\in\{0, ...., TN^{\alpha_r}\}}$.
 Indeed, if there exists $s\in[0,T]$ satisfying $s \geq \frac{6R}{|v^1_0-V_0|}$ such that $x^1_s\in\mathcal{B}(X_s,R)$ and $x^1_0\in\mathcal{B}(X_0,R)$ then 
\begin{itemize}
\item  either there exists $j$ such that $x^1_{t_j}\in\mathcal{B}(X_{t_j},\frac{3}{2}R)$ and $\forall u\in\left[\frac{6R}{|v^1_0-V_0|}, t_j\right],  x^1_u \notin\mathcal{B}(X_u,R)$ (we consider in this sense the first recollision time) with in particular $t_j\geq  t_{rec}\left(|v^1_t-V_t|\right)$,
\item or there exists $j$ such that $|v^1_{t_j}-V_{t_j}|\geq \frac{R}{4}N^{\alpha_r}$.
\end{itemize}
In other words  (see also Figure \ref{fig:preuve_recollision}), if there is a recollision, either the background particle is in the (possibly slightly bigger) interaction radius at one of the discrete time $t_j$, or  the relative velocity at some discrete time $t_j$ is so large that the crossing would have occurred during the small time interval $[t_j, t_{j+1}]$.

Denoting
\begin{align}
\label{def: Is}
I_s:=\mathds{1}_{x^1_{0}\in\mathcal{B}(X_{0},R)}\mathds{1}_{\forall u\in\left[\frac{6R}{|v^1_0-V_0|}, s\right], x^1_u\notin\mathcal{B}(X_u,R)}\mathds{1}_{x^1_{s}\in\mathcal{B}(X_{s},\frac{3}{2}R)}\mathds{1}_{s\geq t_{rec}(|v^1_0-V_0|)}\mathds{1}_{|v^1_0-V_0|\geq N^{-\gamma_{r}}},
\end{align}
we deduce from the previous sampling that
\begin{align}
\mathds{1}_{C^1_0}&\mathds{1}_{|v^1_0-V_0|\geq N^{-\gamma_{r}}}
\leq \sum_{j=0}^{TN^{\alpha_r}} I_{t_j} + \mathds{1}_{|v^1_{t_j}-V_{t_j}|\geq \frac{R}{4}N^{\alpha_r}}\label{eq:control_rec}.
\end{align}
As we work in $\mathcal{G}_N(\delta,\alpha^*,\beta^*)$, we can cutoff the velocity and assume that $|v_{t_j}-V_{t_j}|\geq \frac{R}{4}N^{\alpha_r}$ never occurs (it has, in fact, a probability exponentially small to happen, similarly to \eqref{eq:G_N_0}). 
Let us decompose the terms of $I_{t_j}$:
\begin{itemize}
\item the three indicator functions $\mathds{1}_{x^1_{0}\in\mathcal{B}(X_{0},R)} \mathds{1}_{\forall u\in\left[\frac{6R}{|v^1_0-V_0|}, t_j\right], x^1_u\notin\mathcal{B}(X_u,R)}\mathds{1}_{x^1_{t_j}\in\mathcal{B}(X_{t_j},\frac{3}{2}R)}$ ensure that: the background particle is interacting at time $0$, is out of the interaction ball between times $\frac{6R}{|v^1_0-V_0|}$ (which is an upper bound on the exit time) and $t_j$, and finally that at time $t_j$  is back into the (slightly extended) interaction ball,
\item the function $\mathds{1}_{t_j\geq t_{rec}(|v^1_0-V_0|)}$ indicates that we only consider times $t_j$ greater than the lower bound on the possible re-entry time,
\item and finally $\mathds{1}_{|v^1_0-V_0|\geq N^{-\gamma_{r}}}$ is a bound on the relative velocity ensuring that the various controls we use indeed hold.
\end{itemize}
We shall individually control each $I_{t_j}$ \eqref{def: Is}. \\


\begin{figure}
    \centering
	\includegraphics[width=\linewidth,]{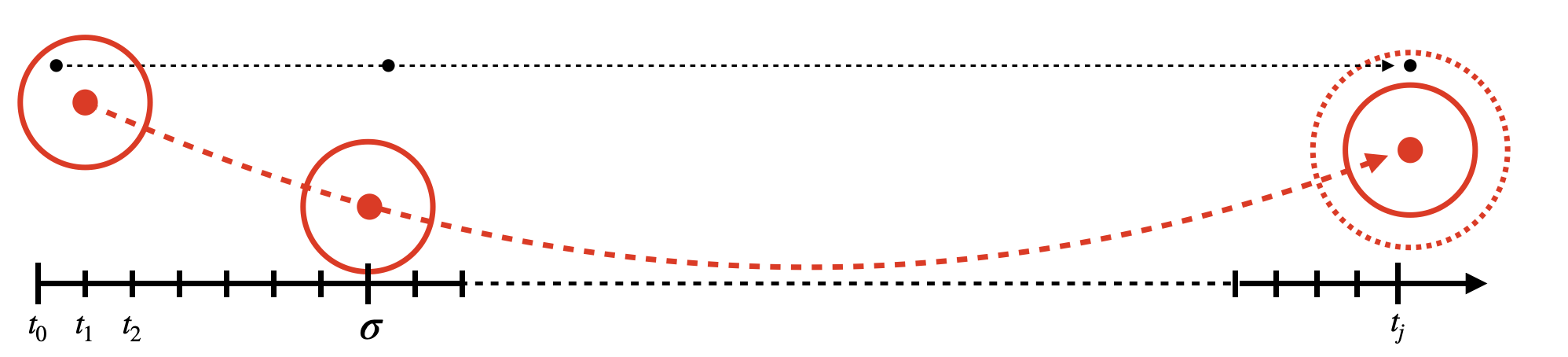}
	\caption{{\small Illustration of the proof of Proposition~\ref{prop:proba_recol}. The tagged particle (in red) and the background particle 1 (in black) overlap initially and then recollide at a later time. The recollision time is estimated by using the sampling $\{t_i \}$ with mesh $N^{- \alpha_r}$. The time at which both particles no longer overlap is estimated in terms of $\sigma$, indexed by a sampling of much cruder mesh $N^{\alpha_r}$.  
    }}
	\label{fig:preuve_recollision}
\end{figure}


\noindent\textbf{Re-entry at a given time.} In order to give a bound on $\hat{\mathbb{E}}\left[I_s\right]$ for any given $s\geq0$, the idea is to start by shifting time and considering the time $0$ to be when the background particle has just left the interaction ball. This way, because we assume that $\forall u\in\left[\frac{6R}{|v^1_0-V_0|}, s\right], x_u\notin\mathcal{B}(X_u,R)$, we may separately deal with the dynamics of the tagged particle and of the background particle. However, we cannot shift the time by $\frac{6R}{|v^1_0-V_0|}$, which would be the natural candidate, because it is an a priori random time. We thus separate $|v^1_0-V_0|$ into smaller subsets and use these subsets to construct the bounds. For the sake of simplicity, we also consider these small subsets to be defined by the parameter $\alpha_r$ \eqref{Def: alpha r}. 
For some $k\in\mathbb{Z}, k\geq-\frac{\gamma_r}{\alpha_r}$,
assume 
\begin{align}
\label{def: k recollision}
N^{k\alpha_r} \leq |v^1_0-V_0| \leq N^{(k+1)\alpha_r}
\end{align}
Notice that $\frac{\gamma_r}{\alpha_r}\in\mathbb{N}$. Like previously, note that we only need to consider $k\leq 0$, as the cutoff of velocities ensures that for $k\geq1$ this is of exponentially small probability. Given $k \leq 0$, let $\sigma:=6RN^{-k\alpha_r}\leq6RN^{\gamma_r} $
and fix the relative velocity so that \eqref{def: k recollision} holds.
 \begin{align*}
 \mathds{1}_{\forall u\in\left[\frac{6R}{|v^1_0-V_0|},s\right],\ x^1_u\notin\mathcal{B}(X_u,R)}\leq\mathds{1}_{\forall u\in[\sigma,s],\ x^1_u\notin\mathcal{B}(X_u,R)}.
  \end{align*}
 We will use this $\sigma$ (which is no longer random) to 
shift the time and  define the new time origin at some point where the background particle has exited the tagged one. 

We start by controlling the position of the background particle at time $\sigma$ based on the controls at time $0$. 
From direct computations,  in $\mathcal{G}_N(\delta,\alpha^*,\beta^*)$ 
and if $x^1_0\in\mathcal{B}(X_0,R)$,  we know using \eqref{eq:control_diff_pos_boot} (since $\sigma<N^{\beta^*}$  by \eqref{eq:gamma_r_suff_petit}) that for $N$ large enough
\begin{align*}
4R \leq |x^1_\sigma-X_\sigma|\leq 8R N^{\alpha_r}
 \quad \text{as} \quad 
\sigma |v^1_0-V_0| \leq 6 R N^{-k \alpha_r} N^{(k+1) \alpha_r} \leq 6  R N^{ \alpha_r}.
\end{align*}
 Therefore we have 
\begin{align*}
\hat{\mathbb{E}}_{\mathcal{G}_N(\delta,\alpha^*,\beta^*)}\left[I_s\right]
& \leq\sum_{k=-\frac{\gamma_r}{\alpha_r}}^0 \hat{\mathbb{E}}_{\mathcal{G}_N(\delta,\alpha^*,\beta^*)}\left[\mathds{1}_{4R \leq|x^1_\sigma-X_\sigma|\leq 8R N^{\alpha_r}}\mathds{1}_{x^1_s\in\mathcal{B}(X_s,\frac{3}{2}R)}\mathds{1}_{\forall u\in[\sigma,s],\ x^1_u\notin\mathcal{B}(X_u,R)} \right.\\
&\hspace{2cm}\left. \mathds{1}_{s\geq t_{rec}(|v^1_0-V_0|)} \, \mathds{1}_{|v^1_0-V_0|\in\left[N^{k\alpha_r},N^{(k+1)\alpha_r}\right]}\right]\\
\leq&\sum_{k=-\frac{\gamma_r}{\alpha_r}}^0   \mathds{1}_{s\geq t_{rec}(N^{k\alpha_r})}  \hat{\mathbb{E}}_{\mathcal{G}_N(\delta,\alpha^*,\beta^*)}\left[\mathds{1}_{4R \leq|x^1_\sigma-X_\sigma|\leq 8R N^{\alpha_r}}\mathds{1}_{x^1_s\in\mathcal{B}(X_s,\frac{3}{2}R)}\mathds{1}_{\forall u\in[\sigma,s],\ x^1_u\notin\mathcal{B}(X_u,R)} \right],
\end{align*}
 where we used the constraint on the initial velocity  and the fact that $t_{rec}(\cdot)$ is increasing \eqref{def: trec def bis} to obtain a lower bound on the recollision time $s$. Once this is done, we drop the condition on the initial velocity in the last line. 
Finally we shift time by $-\sigma$ so that both particles no longer overlap at the new time origin
\begin{align*}
\hat{\mathbb{E}}_{\mathcal{G}_N(\delta,\alpha^*,\beta^*)}\left[I_s\right]
\leq & \sum_{k=-\frac{\gamma_r}{\alpha_r}}^0
 \mathds{1}_{s \geq t_{rec}(N^{k\alpha_r})}
\hat{\mathbb{E}}_{\mathcal{G}_N(\delta,\alpha^*,\beta^*)}\left[\mathds{1}_{4R \leq|x^1_0-X_0|\leq8RN^{\alpha_r}}\mathds{1}_{x^1_{s-\sigma}\in\mathcal{B}(X_{s-\sigma},\frac{3}{2}R)}
\right.\\
&\hspace{2cm}\left.
\mathds{1}_{\forall u\in[0,s-\sigma],\ x^1_u\notin\mathcal{B}(X_u,R)} \right].
\end{align*}
To show that the expectation of $I_s$ is small, 
we are now going to use the strong geometric constraint imposed by a recollision after a long time.
Note that if $\forall u\in[0,s-\sigma],\ x^1_u\notin\mathcal{B}(X_u,R)$, we have both $\forall u\in[0,s-\sigma],\ x^1_u=x^1_0+uv^1_0$ and $\forall u\in[0,s-\sigma],\ (\overline{X}_{u},\overline{V}_{u}, \overline{x}_{1:\mathcal{N}}, \overline{v}_{1:\mathcal{N}} )=(X_u,V_{u}, x_{1:\mathcal{N}}, v_{1:\mathcal{N}} )$.  This way we can write
\begin{align*}
\mathds{1}_{x^1_{s-\sigma}\in\mathcal{B}(X_{s-\sigma},\frac{3}{2}R)}=\mathds{1}_{x^1_0+(s-\sigma)v^1_0\in\mathcal{B}\left(\overline{X}_{s-\sigma},\frac{3}{2}R\right)}=\mathds{1}_{v^1_0\in\mathcal{B}\left(\frac{\overline{X}_{s-\sigma}-x^1_0}{s-\sigma},\frac{3}{2(s-\sigma)}R\right)},
\end{align*}
and thus
\begin{align*}
\hat{\mathbb{E}}_{\mathcal{G}_N(\delta,\alpha^*,\beta^*)}&\left[I_s\right]\\
\leq&\sum_{k=-\frac{\gamma_r}{\alpha_r}}^0
\mathds{1}_{s\geq t_{rec}\left(N^{k\alpha_r}\right)}
\hat{\mathbb{E}}_{\mathcal{G}_N(\delta,\alpha^*,\beta^*)}\left[\mathds{1}_{4R \leq|x^1_0-X_0|\leq 8RN^{\alpha_r}}\mathds{1}_{v^1_0\in\mathcal{B}\left(\frac{\overline{X}_{s-\sigma}-x^1_0}{s-\sigma},\frac{3}{2(s-\sigma)}R\right)}\right].
\end{align*}
Notice that the above quantity requires $v^1_0$ to be in a small ball of radius $\frac{3}{2(s-\sigma)}R$, and $x^1_0$ to be in a ball of radius $N^{\alpha_r}$. The integration therefore yields
\begin{align*}
\hat{\mathbb{E}}_{\mathcal{G}_N(\delta,\alpha^*,\beta^*)}\left[I_s\right]
=O\left(\sum_{k=-\frac{\gamma_r}{\alpha_r}}^0\mathds{1}_{s\geq t_{rec}\left(N^{k\alpha_r}\right)}\frac{N^{d\alpha_r}}{(s-\sigma)^d}\right)
 =O\left(\sum_{k=-\frac{\gamma_r}{\alpha_r}}^0\mathds{1}_{s\geq t_{rec}\left(N^{k\alpha_r}\right)}\frac{N^{d\alpha_r}}{s^d}\right),
\end{align*}
where we used in the last equality that $\sigma =6RN^{-k\alpha_r}\leq6RN^{\gamma_r} \ll t_{rec}\left(N^{k\alpha_r} \right)$ as seen in Lemma~\ref{lem:max_inter_time}.
Finally, $k$ takes only a finite number of values $\{0, \dots, -\frac{\gamma_r}{\alpha_r} \}$ (independent of $N$) and $t_{rec}( \cdot)$ is increasing so that we simply get
\begin{align}
\label{eq: estimation finale Is}
\hat{\mathbb{E}}_{\mathcal{G}_N(\delta,\alpha^*,\beta^*)}\left[I_s\right]
 =O\left( \mathds{1}_{s\geq t_{rec}\left(N^{- \gamma_r}\right)}\frac{N^{d\alpha_r}}{s^d}\right).
\end{align}

\vspace{0.3cm}


\noindent\textbf{Back to \eqref{eq:control_rec}.}
Recalling that $t_j = j N^{-\alpha_r}$,
we deduce from \eqref{eq: estimation finale Is} that 
\begin{align*}
\hat{\mathbb{E}}_{\mathcal{G}_N(\delta,\alpha^*,\beta^*)}\left[\mathds{1}_{C^1_0}\mathds{1}_{|v^1_0-V_0|\geq N^{-\gamma_{r}}}\right]
=&O\left(\sum_{j=0}^{TN^{\alpha_r}}\mathds{1}_{t_j\geq t_{rec}\left(N^{- \gamma_r}\right)}\frac{N^{d\alpha_r}}{t_j^d}\right)
=
O\left( \frac{N^{(d+1)\alpha_r}}{t_{rec}\left(N^{- \gamma_r}\right)^{d-1}}\right).
\end{align*}
Recalling the definition \eqref{def: trec def bis} of $t_{rec}$,  we know from the choice of $\alpha^*, \beta^*$ that 
\begin{align*}
t_{rec}\left(N^{-\gamma_r}\right)=
O\left(N^{1-2\alpha^*- 2 \gamma_r}\right).
\end{align*}
Since for $d\geq 4$,
\begin{align*}
 (d-1) \left(1-2\alpha^*-2\gamma_r \right)- &(d+1)\alpha_r
\geq\frac{19}{18}-2(d-1)\delta,
\end{align*}
we deduce that 
\begin{align*}
\hat{\mathbb{E}}_{\mathcal{G}_N(\delta,\alpha^*,\beta^*)}\left[\mathds{1}_{C^1_0}\mathds{1}_{|v^1_0-V_0|\geq N^{-\gamma_{r}}}\right]=&O\left(N^{-\left(\frac{19}{18}-2(d-1)\delta\right)}\right).
\end{align*}
This completes the proof of Proposition \ref{prop:proba_recol}.
\end{proof}

%
%
%
%

\appendix

%
%
%
%

\section{Technical lemmas}

%
%
%
%

\subsection{Some results from the proof of Proposition~\ref{prop:martingale_formulation}}\label{sec:preuves_lemmes_mart}


\begin{proof}[Proof of Lemma~\ref{lem:mart_toutvabien_avec_bar}]
We start by deriving \eqref{eq:mart_toutvabien_avec_bar}. First note that 
we can write
\begin{align*}
\mathds{1}_{\mathfrak{G}_N(t)}=\mathds{1}_{x^1_t\in\mathcal{B}(X_t,\frac{3}{2}R)}\mathds{1}_{|v^1_t-V_t|\geq N^{-\gamma_r}}\mathds{1}_{\overline{C^1_t}}\mathds{1}_{\text{controls from Lemma~\ref{lem:abstract_approx}}},
\end{align*}
and thus 
\begin{align}
\mathds{1}_{\overline{\mathfrak{G}_N(t)}}=&\mathds{1}_{x^1_t\notin\mathcal{B}(X_t,\frac{3}{2}R)}+\mathds{1}_{x^1_t\in\mathcal{B}(X_t,\frac{3}{2}R)}\mathds{1}_{|v^1_t-V_t|< N^{-\gamma_r}}+\mathds{1}_{x^1_t\in\mathcal{B}(X_t,\frac{3}{2}R)}\mathds{1}_{|v^1_t-V_t|\geq N^{-\gamma_r}}\mathds{1}_{C^1_t}\nonumber\\
&+\mathds{1}_{x^1_t\in\mathcal{B}(X_t,\frac{3}{2}R)}\mathds{1}_{|v^1_t-V_t|\geq N^{-\gamma_r}}\mathds{1}_{\overline{C^1_t}}\mathds{1}_{\text{controls from Lemma~\ref{lem:abstract_approx} do not hold}}.\label{eq:int_toutvabien_avec_bar_1}
\end{align}
Then
\begin{align}
\mathds{1}_{|\overline{x}^1_t-\overline{X}_t|\leq R}\mathds{1}_{\overline{\mathfrak{G}_N(t)}}=&\mathds{1}_{|\overline{x}^1_t-\overline{X}_t|\leq R}\mathds{1}_{\overline{\mathfrak{G}_N(t)}}\mathds{1}_{\mathcal{G}_N(\delta,\alpha^*,\beta^*)}+\mathds{1}_{|\overline{x}^1_t-\overline{X}_t|\leq R}\mathds{1}_{\overline{\mathfrak{G}_N(t)}}\mathds{1}_{\overline{\mathcal{G}_N(\delta,\alpha^*,\beta^*)}}\nonumber\\
\leq& \mathds{1}_{|\overline{x}^1_t-\overline{X}_t|\leq R}\mathds{1}_{\overline{\mathfrak{G}_N(t)}}\mathds{1}_{\mathcal{G}_N(\delta,\alpha^*,\beta^*)}
+
\mathds{1}_{\overline{\mathcal{G}_N(\delta,\alpha^*,\beta^*)}}.\label{eq:int_toutvabien_avec_bar_2}
\end{align}
By the definition \eqref{eq: definition E chapeau} of 
$\hat{\bbE}$, we get
\begin{align*}
\hat{\bbE}\left[\mathds{1}_{\overline{\mathcal{G}_N(\delta,\alpha^*,\beta^*)}}\right]
=&\bbE\left[ \frac{\mathcal{N}}{N} \mathds{1}_{\overline{\mathcal{G}_N(\delta,\alpha^*,\beta^*)}}\right]
\leq \bbE\left[\frac{\mathcal{N}^2}{N^2}\right]^{1/2}\bbE\left[\mathds{1}_{\overline{\mathcal{G}_N(\delta,\alpha^*,\beta^*)}}\right]^{1/2}=O\left(e^{-N^{\delta'}/2}\right).
\end{align*}
This deals with the second term of \eqref{eq:int_toutvabien_avec_bar_2}. Let us now consider the first one, using \eqref{eq:int_toutvabien_avec_bar_1}. 
For $|t-\sigma^+_1|\leq 9RN^{\gamma_r}$, and since Lemma~\ref{Lem: Better set estimates} ensures that in $\mathcal{G}_N(\delta,\alpha^*,\beta^*)$, then $\overline{X}_t$ is close to $X_t$ and $\overline{x}^1_t$ to $x^1_t$ so that 
\begin{align*}
\mathds{1}_{|\overline{x}^1_t-\overline{X}_t|\leq R} \mathds{1}_{\mathcal{G}_N(\delta,\alpha^*,\beta^*)}
\leq \mathds{1}_{|x^1_t-X_t|\leq \frac{3}{2}R}.
\end{align*}
Thus \eqref{eq:int_toutvabien_avec_bar_1} can be estimated by applying 
\eqref{eq:proba_abstract_good}
\begin{equation}
\hat{\mathbb{E}}\left[\mathds{1}_{|\overline{x}^1_t-\overline{X}_t|  \leq R}\mathds{1}_{\overline{\mathfrak{G}_N(t)}}\mathds{1}_{\mathcal{G}_N(\delta,\alpha^*,\beta^*)}\right]
\leq
\hat{\mathbb{E}}\left[\mathds{1}_{\overline{\mathfrak{G}_N(t)}}\mathds{1}_{|x^1_t-X_t|\leq \frac{3}{2}R}\right]=O\left(N^{-(1+\omega')}\right).
\end{equation}
This yields \eqref{eq:mart_toutvabien_avec_bar}. 
Concerning \eqref{eq:mart_toutvabien_avec_bar_chez_bar}, that is with $\widehat{\good(t)}$ defined by
\begin{align*}
\mathds{1}_{\widehat{\good(t)}}=\mathds{1}_{\overline{x}^1_t\in\mathcal{B}(\overline{X}_t,\frac{3}{2}R)}\mathds{1}_{|\overline{v}^1_t-\overline{V}_t|\geq N^{-\gamma_r}}\mathds{1}_{\forall s\text{ s.t. }|s-t|\geq \frac{9R}{|\overline{v}^1_t-\overline{V}_t|},\ \overline{x}^1_s\notin\mathcal{B}(\overline{X}_s,\frac{3}{2}R)},
\end{align*}
the proof of Lemma~\ref{lem:abstract_tout_va_bien} can be directly adapted\footnote{Since the process $(\overline{Z}_t)_t$ is independent from $(x_1,v_1)$, one can first easily prove $\hat{\bbE}\left[\mathds{1}_{\overline{x}^1_t\in\mathcal{B}(\overline{X}_t,\frac{3}{2}R)}\mathds{1}_{|\overline{v}^1_t-\overline{V}_t|< N^{-\gamma_r}}\right]=O\left(N^{-d\gamma_r}\right)$ as it is a direct integration over $(x_1,v_1)$. Then, again because the tagged particle is independent of particle $1$, the result equivalent to Proposition~\ref{prop:proba_recol} for the process $(\overline{Z}_t)_t$ concerning the possible recollision is a direct modification of its proof: it is just a constraint on $v_1$ ensuring that $\overline{x}$  belongs to two balls distant in time.}.

\end{proof}

\begin{proof}[Proof of Lemma~\ref{lem:reecrire_drift}]
First notice that, by approximating $\left|e^{-\frac{\Phi(x)}{N}}-1\right|\leq \frac{2\|\Phi\|_\infty}{N}\mathds{1}_{x\in\mathcal{B}(0,R)}$ and using the symmetry of $\Phi$ to flip the signs, we have
\begin{align}
\tilde{\Lambda}(t,X,\overline{X}_t,\overline{V}_t)
=& - \int_{\mathbb{T}}\int_{\mathbb{R}^d}\gamma(v_1)\text{Hess}\Phi(x_1+tv_1 -\overline{X}_t)\int_0^t\int_0^s \nabla\Phi(-\overline{X}_t+(t-u)\overline{V}_t+x_1+uv_1)dudsdv_1dx_1\nonumber\\
&\times\left(1+O\left(1/N\right)\right)
\label{eq: Lambda intermediaire}
\\
=&-\int_{\mathbb{T}}\int_{\mathbb{R}^d}\gamma(v+\overline{V}_t)\text{Hess}\Phi\left(x\right)\int_0^t\int_0^s \nabla\Phi\left(x+(u-t) v\right)dudsdvdx \times\left(1+O\left(1/N\right)\right)\nonumber\\
=&-\int_{\mathcal{B}(0,R)}\int_{\mathbb{R}^d}\gamma(v+\overline{V}_t)\text{Hess}\Phi\left(x\right)\int_{-t}^0\int_{-t}^{s} \nabla\Phi\left(x+u v\right)dudsdvdx \times\left(1+O\left(1/N\right)\right), \nonumber
\end{align}
where the second equality follows from the change of variables 
\begin{equation}
\label{eq: change of variable Lambda}
(x_1,v_1) \mapsto(x= x_1+tv_1 -\overline{X}_t, v = v_1 - \overline{V}_t).
\end{equation}
We now use the following lemma, the proof of which can be found below.


\begin{lemma}\label{lem:retirer_temps_drift}
We have 
\begin{align*}
&\left|\int_{\mathcal{B}(0,R)}\int_{\mathbb{R}^d}\gamma(v+\overline{V}_t)\text{Hess}\Phi\left(x\right)\int_{-t}^0\int_{-t}^{s} \nabla\Phi\left(x+u v\right)dudsdvdx\right.\\
&\qquad\left.-\int_{\mathcal{B}(0,R)}\int_{\mathbb{R}^d}\gamma(v+\overline{V}_t)\text{Hess}\Phi\left(x\right)\int_{-\infty}^0\int_{-\infty}^{s} \nabla\Phi\left(x+u v\right)dudsdvdx\right|=O\left(\frac{1}{t^{d-2}}\wedge 1\right)
\end{align*}
\end{lemma}
This finishes the proof of Lemma~\ref{lem:reecrire_drift}.
\end{proof}


\begin{proof}[Proof of Lemma~\ref{lem:reecrire_diffusion}]
 We proceed as in \eqref{eq: Lambda intermediaire}, 
 by approximating $\left|e^{-\frac{\Phi(x)}{N}}-1\right|\leq \frac{2\|\Phi\|_\infty}{N}\mathds{1}_{x\in\mathcal{B}(0,R)}$, using the symmetry of $\Phi$ and the change of variables \eqref{eq: change of variable Lambda}, we have
\begin{align*}
\tilde{D}_{i,j}(t,X,\overline{X}_t,\overline{V}_t)=&\int_{\mathbb{T}} \int_{\mathbb{R}^d}\gamma(v_1)\left(\int_0^t\partial_j \Phi(- \overline{X}_t+(t-s)\overline{V}_t+x_1+sv_1)ds\right)\\
&\qquad\times\partial_i\Phi(- \overline{X}_t+x_1+tv_1)dv_1dx_{1} 
\times \left(1+O\left(1/N\right)\right)\\
=&\int_{\mathbb{T}}\int_{\mathbb{R}^d}\gamma(v+\overline{V}_t)\partial_i\Phi(x)\left(\int_0^t\partial_j \Phi(x+(s-t)v)ds\right)dvdx \times \left(1+O\left(1/N \right) \right)\\
=&\int_{\mathbb{T}}\int_{\mathbb{R}^d}\gamma(v+\overline{V}_t)\partial_i\Phi(x)\left(\int_{-t}^0\partial_j \Phi(x+sv)ds\right)dvdx \times \left(1+O\left(1/N\right) \right).
\end{align*}
We conclude using the following lemma.


\begin{lemma}\label{lem:retirer_temps_diffusion}
We have 
\begin{align}
\label{eq: difference lemme A.2}
&\left|\int_{\mathbb{T}}\int_{\mathbb{R}^d}\gamma(v+\overline{V}_t)\partial_i\Phi(x)\left(\int_{-t}^0\partial_j \Phi(x+sv)ds\right)dvdx\right.\\
&\qquad\left.-\int_{\mathbb{T}}\int_{\mathbb{R}^d}\gamma(v+\overline{V}_t)\partial_i\Phi(x)\left(\int_{-\infty}^0\partial_j \Phi(x+sv)ds\right)dvdx\right|=O\left(\frac{1}{t^{d-1}}\wedge 1\right)  \nonumber
\end{align}
\end{lemma}
\end{proof}


\begin{proof}[Proof of Lemma~\ref{lem:retirer_temps_drift}]
We may compute
\begin{align*}
&\left|\int_{\mathcal{B}(0,R)}\int_{\mathbb{R}^d}\gamma(v+\overline{V}_t)\text{Hess}\Phi\left(x\right)\int_{-t}^0\int_{-t}^{s} \nabla\Phi\left(x+u v\right)dudsdvdx\right.\\
&\qquad \qquad \qquad\left.-\int_{\mathcal{B}(0,R)}\int_{\mathbb{R}^d}\gamma(v+\overline{V}_t)\text{Hess}\Phi\left(x\right)\int_{-\infty}^0\int_{-\infty}^{s} \nabla\Phi\left(x+u v\right)dudsdvdx\right| \leq I_1+I_2,
\end{align*}
with
\begin{align}
I_1 := & \left|\int_{\mathcal{B}(0,R)}\int_{\mathbb{R}^d}\gamma(v+\overline{V}_t)\text{Hess}\Phi\left(x\right)\int_{-t}^0\int_{-\infty}^{-t} \nabla\Phi\left(x+u v\right)dudsdvdx\right|,
\label{eq: def I1}\\
I_2 := &\left|\int_{\mathcal{B}(0,R)}\int_{\mathbb{R}^d}\gamma(v+\overline{V}_t)\text{Hess}\Phi\left(x\right)\int_{-\infty}^{-t}\int_{-\infty}^{s} \nabla\Phi\left(x+u v\right)dudsdvdx\right| .
\label{eq: def I2}
\end{align}
We first deal, since $x\in\mathcal{B}(0,R)$ and $x+uv\in\mathcal{B}(0,R)$ imply $v\in\mathcal{B}\left(0,\frac{2R}{|u|}\right)$, with
\begin{align*}
I_1
&=O\left( \int_{-t}^0\int_{-\infty}^{-t}\int_{\mathbb{R}^d}\gamma(v+\overline{V}_t)\mathds{1}_{v\in\mathcal{B}(0,\frac{2R}{|u|})}\int_{\mathbb{T}}\mathds{1}_{x\in\mathcal{B}(0,R)}dxdvduds\right)\\
&=O\left(   \int_{-t}^0\int_{-\infty}^{-t}\int_{\mathbb{R}^d}\gamma(v+\overline{V}_t)\mathds{1}_{v\in\mathcal{B}(0,\frac{2R}{|u|})}dvduds\right)
=O\left( t\int_{-\infty}^{-t}\left(\frac{1}{|u|^d}\wedge 1\right)du\right),
\end{align*}
since we can always bound by $\int_{\mathbb{R}^d}\gamma(v+\overline{V}_t)\mathds{1}_{v\in\mathcal{B}(0,\frac{2R}{|u|})}$, by $1$ or by $\int_{\mathbb{R}^d}\mathds{1}_{v\in\mathcal{B}(0,\frac{2R}{|u|})}dv=C_d \left(\frac{2R}{|u|}\right)^d$. This yields
\begin{align*}
I_1=&O\left( t\left((1-t)_++\frac{1}{d-1}\left(\frac{1}{t^{d-1}}\wedge1\right)\right)\right)=O\left(\frac{1}{t^{d-2}}\wedge 1\right).
\end{align*}
Let us now deal with the second term
\begin{align}
I_2 =&O\left( \int_{-\infty}^{-t}\int_{-\infty}^{s}\int_{\mathbb{R}^d}\gamma(v+\overline{V}_t)\int_{\mathbb{T}}\mathds{1}_{x\in\mathcal{B}(0,R)}\mathds{1}_{x+uv\in\mathcal{B}(0,R)}dxdvduds\right) 
\label{eq: bound I2}\\
=&O\left(\int_{-\infty}^{-t}\int_{-\infty}^{s}\int_{\mathbb{R}^d}\gamma(v+\overline{V}_t)\mathds{1}_{v\in\mathcal{B}(0,\frac{2R}{|u|})}dvduds\right) \nonumber \\
=&O\left(\int_{-\infty}^{-t}\int_{-\infty}^{s}\left(\frac{1}{|u|^{d}}\wedge 1\right)duds\right) = O\left(\frac{1}{t^{d-2}}\wedge1\right) .
\nonumber
\end{align}
Hence the result.
\end{proof}


\begin{proof} [Proof of Lemma~\ref{lem:retirer_temps_diffusion}]
Likewise, we may compute the difference in \eqref{eq: difference lemme A.2}
and, similarly as the proof of Lemma~\ref{lem:retirer_temps_drift}
\begin{align}
\label{eq: boundedness D}
&\left|\int_{\mathbb{T}}\int_{\mathbb{R}^d}\gamma(v+\overline{V}_t)\partial_i\Phi(x)\left(\int_{-\infty}^{-t}\partial_j \Phi(x+sv)ds\right)dvdx\right|\\
&\qquad\qquad\leq \|\partial_i\Phi\|_\infty\|\partial_j\Phi\|_\infty\int_{\mathbb{T}}\int_{\mathbb{R}^d}\gamma(v+\overline{V}_t)\int_{-\infty}^{-t}\mathds{1}_{x\in\mathcal{B}(0,R)}\mathds{1}_{x+sv\in\mathcal{B}(0,R)} \nonumber\\
&\qquad\qquad=O\left(\int_{-\infty}^{-t}\left(\frac{1}{|s|^d}\wedge 1\right)\right)= O\left(\frac{1}{t^{d-1}}\wedge 1\right). \nonumber
\end{align}
Hence the result.
\end{proof}

%
%
%
%

\subsection{Results on the macroscopic coefficients}\label{sec:tech_coeff}


\begin{proof}[Proof of Lemma~\ref{lem:prop_coeff}]


\textbf{(i)} 
Note that $\Lambda$ coincides with $I_2$ \eqref{eq: def I2} with $t=0$ and it has been bounded in \eqref{eq: bound I2}.
In the same way, $D_i,j(V)$ has been bounded in \eqref{eq: boundedness D}.\\


\noindent\textbf{(ii)} We are going to recover a symmetric form of the diffusion coefficient $D$ \eqref{def:diffusion}.
First, by using the symmetry of $\Phi$ and the change of variables $x \to -x$ and $s \to -s$, we get
\begin{align}
\label{eq: D symmetric temps}
D(V)=\frac{1}{2} \int_{\mathbb{R}^d}dv\gamma(v+V)
\int_{\bbR^d}dx \nabla\Phi(x)\otimes\int_\bbR \nabla\Phi(x+sv)ds.
\end{align}
Given $v,V$, we decompose $x = u v + x'$ with $x'$ in the hyperplane 
$v^\perp$ orthogonal to $v$ to obtain a symmetric representation
\begin{align}
D(V)&=\frac{1}{2} \int_{\mathbb{R}^d}dv\gamma(v+V)
\int_\bbR du \, |v| \int_{v^\perp}dx' \nabla\Phi(x' + u v )\otimes\int_\bbR \nabla\Phi(x'+(s+u)v)ds \nonumber \\
& =\frac{1}{2} \int_{\mathbb{R}^d}dv \gamma(v+V)|v|
 \int_{v^\perp}dx'
 \left( \int_\bbR du \nabla\Phi(x' + u v ) \right) \otimes  \left(  \int_\bbR ds \nabla\Phi(x'+s v) \right).
 \label{eq: D symmetric}
\end{align}
Thus $D$ is symmetric and definite positive.\\


\noindent\textbf{(iii)}
We compute for $i,j,k\in\{1,...,d\}$
\begin{align*}
\partial_k & D_{i,j}(V)=- \frac{1}{2} \int_{\mathbb{T}}\int_{\mathbb{R}^d}(v_k+V_k)\gamma(v+V)\partial_{i}\Phi(x)
\int_{-\infty}^{+\infty} \partial_{j}\Phi(x+sv)dsdvdx\\
=&-\int_{\mathbb{T}}\int_{\mathbb{R}^d}\gamma(v+V)\partial_{i}\Phi(x)\int_{-\infty}^0\partial_k\partial_{j}\Phi(x+sv)sdsdvdx
\quad \text{integrating by parts  }v_k,\\
=&\int_{\mathbb{T}}\int_{\mathbb{R}^d}\gamma(v+V)\partial_{i}\Phi(x)\int_{-\infty}^0\int_{-\infty}^s\partial_k\partial_{j}\Phi(x+uv)dudsdvdx
\quad \text{integrating by parts $s$},\\
=&-\int_{\mathbb{T}}\int_{\mathbb{R}^d}\gamma(v+V)\partial_k\partial_{j}\Phi(x)\int_{-\infty}^0\int_{-\infty}^s\partial_{i}\Phi(x+uv)dudsdvdx
\end{align*}
where we used the change of variable $x \to -(x+uv)$ and the symmetry of $\Phi$ to conclude. 
Thus $\sum_{i=1}^d\partial_iD_{i,j}(V)=\Lambda_j(V)$,
i.e, since $D$ is symmetric, for all $i\in\{1,...,d\}$ we have 
\begin{equation}
\label{eq: identite bis Lambda D}
\Lambda_i(V)=\sum_{j=1}^d\partial_jD_{i,j}(V).
\end{equation}


\noindent\textbf{(iv)} 
Using the expression \eqref{eq: D symmetric temps}, we can write
\begin{align*}
\forall i \leq d, \qquad 
\sum_{j=1}^d
\partial_j D_{i,j}(V)=&- \frac{1}{2} \sum_{j=1}^d \int_{\mathbb{T}}\int_{\mathbb{R}^d}(v_j + V_j )\gamma(v+V)\partial_{i}\Phi(x)\int_{-\infty}^{+\infty} \partial_{j}\Phi(x+sv)dsdvdx.
\end{align*}
Since 
\begin{align*}
\int_{-\infty}^{+\infty} v \cdot \nabla_x \Phi(x+sv)ds =  \int_{-\infty}^{+\infty}  \partial_s \Phi(x+sv)ds = 0,
\end{align*}
we deduce from \eqref{eq: identite bis Lambda D} that 
\begin{align*}
\forall i \leq d, \quad 
\Lambda_i (V) = 
\sum_{j=1}^d\partial_jD_{i,j}(V)=-\sum_{j}V_jD_{i,j}(V)
\quad \Rightarrow \quad
\Lambda (V)= -D(V)V.
\end{align*}
This completes (iv).\\


\noindent\textbf{(v)} Write
\begin{align*}
\hat{\Phi}(k)=\int_{\mathbb{R}^d}e^{-ik\cdot x}\Phi(x) dx\qquad\implies\qquad \Phi(x)=\frac{1}{(2\pi)^d}\int_{\mathbb{R}^d}e^{i k\cdot x}\hat{\Phi}(k) dk.
\end{align*}
Note that, since $\Phi$ is even and real, so is $\hat{\Phi}$. Therefore we have, for $j,l\in\{1,...,d\}$ 
\begin{align*}
D_{j,l}(V)=&\frac{1}{2}\int_{\mathbb{R}^d}dx\int_{\mathbb{R}^d}dv\gamma(v+V)\partial_{j}\Phi(x)\int_{-\infty}^\infty\partial_{l}\Phi(x+sv)ds\\
=&-\frac{1}{2(2\pi)^{2d}}\int_{\mathbb{R}^d}dx\int_{\mathbb{R}^d}dv\gamma(v+V)\int_{\mathbb{R}^d}dk k_je^{i k\cdot x}\hat{\Phi}(k) \int_{-\infty}^\infty\int_{\mathbb{R}^d}dk' k'_le^{i k'\cdot (x+sv)}\hat{\Phi}(k')ds\\
=&-\frac{1}{2(2\pi)^{2d}}\int_{\mathbb{R}^d}dv\int_{\mathbb{R}^d}dk\int_{\mathbb{R}^d}dk' k_j k'_l \gamma(v+V) \hat{\Phi}(k)\hat{\Phi}(k')\left(\int_{\mathbb{R}^d}dxe^{i(k+k')\cdot x}\right) \left(\int_{-\infty}^\infty e^{i s k'\cdot v}ds\right)\\
=&-\frac{\pi}{(2\pi)^{d}}\int_{\mathbb{R}^d}dv\int_{\mathbb{R}^d}dk\int_{\mathbb{R}^d}dk' k_j k'_l \gamma(v+V) \hat{\Phi}(k)\hat{\Phi}(k')\delta_{k+k'}\delta_{k'\cdot v},
\end{align*}
and thus, using the parity of $\hat{\Phi}$, we get
\begin{align*}
D(V)=&\frac{\pi}{(2\pi)^{d}} \int_{\mathbb{R}^d}dv\int_{\mathbb{R}^d}dk  \gamma(v+V) (k\otimes k)|\hat{\Phi}(k)|^2\delta_{k\cdot v}
\end{align*}
Since, by the previous point, for all $V\in\mathbb{R}^d$ we have $\Lambda(V)=\nabla\cdot D(V)$, where the divergence is taken line by line, we have
\begin{align*}
\Lambda (V)=&\frac{\pi}{(2\pi)^{d}}\int_{\mathbb{R}^d}dv\int_{\mathbb{R}^d}dk  (k\otimes k)\nabla\gamma(v+V)|\hat{\Phi}(k)|^2\delta_{k\cdot v}.
\end{align*}

\end{proof}


\begin{proof}[Proof of Lemma~\ref{lem:coeff_bornes}]


\noindent\textbf{$\bullet$ Boundedness.} The uniform in $V$ bound on $\Lambda$ and $D$ is obtained by the same proof as Lemma~\ref{lem:prop_coeff} (i).\\


\noindent\textbf{$\bullet$ Lipschitz continuity.} Similarly\footnote{For the last inequality: first assume $|V|\geq 2$. In this case, for $|v|\leq 1$, we have $|v+V|\geq 1\geq |v|$. Using that  $x\mapsto |x|e^{-x^2/4}$ is bounded and that $d\geq 3$, we write
\begin{align*}
\int_{\mathbb{R}^d}dv\gamma(v+V)\frac{|v+V|}{|v|^2}=&O\left( \int_{\mathbb{R}^d}dv\gamma\left(\frac{v+V}{\sqrt{2}}\right)\frac{1}{|v|^2}\right)=O\left(\int_{\mathcal{B}(0,1)}dv\gamma\left(\frac{v+V}{\sqrt{2}}\right)\frac{1}{|v|^2}+\int_{\mathbb{R}^d\setminus \mathcal{B}(0,1)}dv\gamma\left(\frac{v+V}{\sqrt{2}}\right)\frac{1}{|v|^2}\right)\\
=& O\left(\int_{\mathcal{B}(0,1)}dv\frac{\gamma\left(\frac{v+V}{2}\right)}{|v+V|^2}+\int_{\mathbb{R}^d\setminus \mathcal{B}(0,1)}dv\gamma\left(\frac{v+V}{\sqrt{2}}\right)\right)<\infty.
\end{align*}
If $|V|\leq 2$, we just bound $|v+V|\leq |v|+2$ and $|v+V|^2\geq \frac{|v|^2}{2}-4$. This yields the result.
}
\begin{align*}
\left|\nabla \Lambda(V)\right|=&\left|\int_{\mathcal{B}(0,R)}dx\int_{\mathbb{R}^d}dv(v+V)\gamma(v+V)\text{Hess}\Phi(x)\int_{-\infty}^0\int_{-\infty}^s\nabla\Phi(x+uv)duds\right|\\
=&O\left(\int_{\mathbb{R}^d}\gamma(v+V)|v+V|\int_{-\infty}^0\int_{-\infty}^s\mathds{1}_{|u|\in\mathcal{B}(0,\frac{2R}{|v|})}dudsdv\right)\\
=&O\left(\int_{\mathbb{R}^d}dv\gamma(v+V)\frac{|v+V|}{|v|^2}\right)<\infty,
\end{align*}
which implies that $\Lambda$ is Lipschitz continuous. Again,  similar calculations allow us to  bound $|\partial_{k}D_{i,j}(V)|$ uniformly in $V$ for all $i,j,k\in\{1,...,d\}$. Hence the result.

\end{proof}

%
%
%
%

\subsection{Tightness estimates}\label{sec:proof_tight}

Without loss of generality, it is enough to assume in this section $g_0=1$. We first state a variation on Proposition~\ref{prop:martingale_formulation} which will be useful later on. 


\begin{lemma}
\label{lem:calcul_inc_pow4}
Let $f(v) = v^k$ with $k=2,4$, then  we have for all $t\in[0,N]$
\begin{align*}
\mathbb{E}\left[f(V_t-V_0)\right]=&\frac{1}{N}\mathbb{E}\left[\int_0^t 2\nabla f(V_s-V_0)\cdot \Lambda(V_s)+\text{Hess}f(V_s-V_0):D(V_s)
ds\right]
+O\left(\frac{t}{N^{1+\frac{2}{3}\omega}}+\frac{t\wedge N^{\gamma_r}}{N^{1-\frac{\omega}{3}}}\right).
\end{align*}
Note that the error term is worse than in  Proposition~\ref{prop:martingale_formulation}.
\end{lemma}
The proof of the lemma is postponed to the end of this section. It 
follows from Proposition~\ref{prop:martingale_formulation} by regularizing the test function. As an almost direct consequence we get


\begin{lemma}
Consider $s,t\in[0,N]$ such that $|t-s|\geq N^{\gamma_r+\frac{\omega}{3}}$, where we assume $\gamma_r+\frac{\omega}{3}\leq \frac{1-2\alpha^*}{3}-\delta$ (up to considering a smaller $\omega$ in Proposition~\ref{prop:martingale_formulation}). We have
\begin{equation}\label{eq:kolm_2_1}
\bbE[|V_t-V_s|^2]=O\left(\frac{|t-s|}{N}\right).
\end{equation}
Furthermore, for any $s,t\in[0,N]$, we have
\begin{equation}\label{eq:kolm_2_2}
\bbE[|V_t-V_s|^2]=O\left(\frac{|t-s|^2}{N^{1-\delta}}\right).
\end{equation}
\end{lemma}


\begin{proof}
By time-invariance, let us prove this result for $s=0$. Assume $t\geq N^{\gamma_r+\frac{\omega}{3}}$. We have, by Lemma~\ref{lem:calcul_inc_pow4}, using the boundedness of $\Lambda, D$ from Lemma~\ref{lem:coeff_bornes} and the fact that $\bbE[|V_t-V_0|]=O(1)$

\begin{align*}
\mathbb{E}\left[\left|V_t-V_0\right|^2\right]=O\left(\frac{t}{N}+\frac{t}{N^{1+\frac{2\omega}{3}}}+\frac{t\wedge N^{\gamma_r}}{N^{1-\frac{\omega}{3}}}\right)=O\left(\frac{t}{N}\right),
\end{align*}
hence \eqref{eq:kolm_2_1}. 
 To get \eqref{eq:kolm_2_2}, it is sufficient to 
use the control \eqref{eq:G_N_3} in the good set
\begin{align}
\mathbb{E}\left[|V_t-V_0|^2\right]=&\mathbb{E}\left[\left|\frac{1}{N}\int_0^t \sum_{i}\nabla\Phi(X_s-x^i_s)ds\right|^2\mathds{1}_{\mathcal{G}_N(\delta)}\right]
+ t^2 \mathbb{E}\left[\left|\frac{1}{N t}\int_0^t \sum_{i}\nabla\Phi(X_s-x^i_s)ds\right|^2\mathds{1}_{\overline{\mathcal{G}_N(\delta)}}\right]
\nonumber \\
\leq & O\left(\frac{t^2}{N^{1-\delta}}\right)
+ t^2 \mathbb{E}\left[ \left(\frac{\sum_{i}\nabla\Phi(X_0-x^i_0)}{N}\right)^4\right]^{1/2}
\mathbb{P}\left[\overline{\mathcal{G}_N(\delta)}\right]^{1/2}
= O\left(\frac{t^2}{N^{1-\delta}}\right),
\label{eq: difference temps court}
\end{align}
where we used successively Jensen and H\"older inequalities as well as Proposition \ref{prop:good_set_good} to estimate the complement of the good set.
\end{proof}

 We now turn our attention to the proof of \eqref{eq:equi_pow_4}
 

\begin{proof}[Proof of Lemma~\ref{lem:equi}] 
The proof is split into two steps according to short and large times.
We start by assuming $t\geq N^{\gamma_r+\frac{2\omega}{3}}$. Using Lemma \ref{lem:calcul_inc_pow4}, we get 
\begin{align}
\mathbb{E}\left[|V_t-V_0|^4\right]=&\frac{1}{N}\mathbb{E}\left[\int_0^t \left(8\Lambda (V_s)\cdot (V_s-V_0)|V_s-V_0|^2+ 8 D(V_s): (V_s-V_0)\otimes(V_s-V_0)\right.\right.\nonumber\\
&\left.\left.\qquad\qquad+4\text{Tr}(D(V_s))|V_s-V_0|^2\right)ds\right] +O\left(\frac{t}{N^{1+\frac{2}{3}\omega}}+\frac{t\wedge N^{\gamma_r}}{N^{1-\frac{\omega}{3}}}\right)\label{eq:int_equi_pow_4_int_1}.
\end{align}
By the properties $\Lambda\in L^\infty(\mathbb{R}^d)$  and $D\in L^\infty(\mathbb{R}^d,\mathbb{R}^{d\times d})$ from Lemma~\ref{lem:coeff_bornes}, we get
\begin{align}
\mathbb{E}\left[|V_t-V_0|^4\right]\leq&\frac{1}{N}\mathbb{E}\left[\int_0^t 4\|\Lambda\|_\infty |V_s-V_0|^4+(4\|\Lambda\|_\infty+(8+4d)\|D\|_\infty )|V_s-V_0|^2ds\right]\nonumber\\
&+O\left(\frac{t}{N^{1+\frac{2\omega}{3}}}+\frac{t\wedge N^{\gamma_r}}{N^{1-\frac{\omega}{3}}}\right).\label{eq:int_equi_pow_4}
\end{align}
Notice that for all $s\in[0,t]$
\begin{align}
\mathbb{E}\left[|V_s-V_0|^4\right]\leq 8\mathbb{E}\left[|V_s|^4+|V_0|^4\right]=O\left(1\right),\label{eq:int_kolm_4_1}
\end{align}
and that for $t\geq N^{\gamma_r+\frac{2\omega}{3}}$, using \eqref{eq:kolm_2_1}-\eqref{eq:kolm_2_2}, we have
\begin{align}
\frac{1}{N}\int_0^t \bbE[|V_s-V_0|^2]ds=&\frac{1}{N}\int_0^{N^{\gamma_r+\frac{2\omega}{3}}} \bbE[|V_s-V_0|^2]ds+\frac{1}{N}\int_{N^{\gamma_r+\frac{2\omega}{3}}}^t \bbE[|V_s-V_0|^2]ds\nonumber\\
=&O\left(\frac{1}{N}\int_0^{N^{\gamma_r+\frac{2\omega}{3}}} \frac{s^2}{N^{1-\delta}}ds+\frac{1}{N}\int_{N^{\gamma_r+\frac{2\omega}{3}}}^t \frac{s}{N}ds\right)\nonumber\\
=&O\left(\frac{N^{3(\gamma_r+\frac{2\omega}{3})}}{N^{2-\delta}}+\frac{t^2}{N^2}\right)
=O\left(\left(\frac{t}{N}\right)^{3/2}\right),\label{eq:int_kolm_4_2}
\end{align}
where for this last estimate we use $\frac{3}{2}\gamma_r+\omega\leq 1-\delta$ (again, up to considering a smaller $\omega$ in Proposition~\ref{prop:martingale_formulation}).

Plugging \eqref{eq:int_kolm_4_1}-\eqref{eq:int_kolm_4_2} back into \eqref{eq:int_equi_pow_4} first yields
\begin{align*}
\mathbb{E}\left[|V_t-V_0|^4\right]=O\left(\frac{t}{N}+\left(\frac{t}{N}\right)^{3/2}+\frac{t}{N^{1+\frac{2\omega}{3}}}+\frac{t\wedge N^{\gamma_r}}{N^{1-\frac{\omega}{3}}}\right)=O\left(\frac{t}{N}\right),
\end{align*}
and then using again this new estimate in \eqref{eq:int_equi_pow_4}, we conclude that 
\begin{align*}
\mathbb{E}\left[|V_t-V_0|^4\right]=O\left(\left(\frac{t}{N}\right)^{3/2}+\frac{t}{N^{1+\frac{2\omega}{3}}}+\frac{t\wedge N^{\gamma_r}}{N^{1-\frac{\omega}{3}}}\right).
\end{align*}
For $t\geq  N^{\gamma_r+\frac{2\omega}{3}}$, we thus have 
\begin{equation}
\label{eq:tight_borne_eps}
\mathbb{E}\left[|V_t-V_0|^4\right]=O\left(\left(\frac{t}{N}\right)^{1+\epsilon}\right)
\quad \text{with} \quad 
\epsilon=\min\left( \frac{\omega/3}{1-\gamma_r-\frac{\omega}{3}}, \frac{2\omega}{3},\frac{1}{2}\right)>0.
\end{equation}

\medskip

Considering now $t\leq N^{\gamma_r+\frac{2\omega}{3}}$, we get, by proceeding as in \eqref{eq: difference temps court}
\begin{align*}
\mathbb{E}\left[|V_t-V_0|^4\right]
=&O\left(\frac{t^4}{N^{2(1-\delta)}}\right)=O\left(\left(\frac{t}{N}\right)^{1+\epsilon}\right),
\end{align*}
where we use \eqref{eq:tight_borne_eps} for this last bound. This and \eqref{eq:tight_borne_eps} completes the proof of Lemma \ref{lem:equi}. 
\end{proof}

\medskip

It now only remains to prove Lemma~\ref{lem:calcul_inc_pow4}.


\begin{proof}[Proof of Lemma~\ref{lem:calcul_inc_pow4}] 
We simply consider the case $f(v) =|v|^4$.
We expand 
\begin{align*}
\mathbb{E}\left[|V_t-V_0|^4\right]=&2\mathbb{E}\left[|V_0|^2(|V_t|^2-|V_0|^2)\right]-4\mathbb{E}\left[V_0\cdot (V_t|V_t|^2-V_0|V_0|^2)\right]-4\mathbb{E}\left[V_0|V_0|^2\cdot (V_t-V_0)\right]\\
&+4\sum_{i,j=1}^d\mathbb{E}\left[V_0^iV_0^j(V_t^iV_t^j-V_0^iV_0^j)\right]+\mathbb{E}\left[(|V_t|^4-|V_0|^4)\right]\\
=&:I_1+I_2+I_3+I_4+I_5
\end{align*}
Note that $I_5=0$, we only add it to be able to close the calculations below. For each term above, we want to use our martingale approximation in order to control the increments. 
Below, we extensively use the following facts
\begin{itemize}
\item for any $s\in[0,t]$ and for any $k\in\mathbb{N}$, we have 
\begin{align}
\mathbb{E}\left[|V_s|^k\right]=O_k(1),\qquad \mathbb{E}\left[|\mathcal{N}|^k\right]=O_k(N^k|\mathds{T}|^k)\label{eq:lem_inc_4_erreur_1},
\end{align}
\item there exists $\epsilon>0$ such that for any $s\in[0,t]$
\begin{align}
\mathbb{E}\left[\mathds{1}_{|V_s|\geq N^{\omega/12}}\right]=O\left(e^{-N^{\epsilon}}\right) \label{eq:lem_inc_4_erreur_2}.
\end{align}
\end{itemize}
We only write how to deal with $I_1$, as the four other terms can be dealt with similarly.

Consider a smooth compactly supported function $f$ satisfying $f(v):=|v|^2$ for $|v|\leq N^{\frac{\omega}{12}}$ and $\|f\|_\infty+\|\nabla f\|_{W^{2,\infty}}=O\left(N^{\frac{\omega}{6}}\right)$. This way we may write
\begin{align*}
&\left||V_0|^2(|V_t|^2-|V_0|^2)-f(V_0)(f(V_t)-f(V_0))\right|\\
&\qquad=O\left( \mathcal{N}\frac{t}{N}\left(\left(N^{\frac{\omega}{6}}\right)^2+|V_0|^2(|V_t|+|V_0|)\right)\left(\mathds{1}_{|V_0|\geq N^{\omega/12}}+\mathds{1}_{|V_t|\geq N^{\omega/12}}\right)\right).
\end{align*}
Here, we use $|V_t-V_0|=O\left(\mathcal{N}\frac{t}{N}\right)$, $|f(V_0)(f(V_t)-f(V_0))|\leq \|f\|_\infty\|\nabla f\|_\infty|V_t-V_0|$, and $|V_t|^2-|V_0|^2=(V_t-V_0)\cdot (V_t+V_0)$. Note that, using \eqref{eq:lem_inc_4_erreur_1} and \eqref{eq:lem_inc_4_erreur_2}, this error actually yields
\begin{align*}
\mathbb{E}\left[\left||V_0|^2(|V_t|^2-|V_0|^2)-f(V_0)(f(V_t)-f(V_0))\right|\right]=O\left(\frac{t}{N^{1+\frac{2\omega}{3}}}\right)
\end{align*}
  For $v\leq N^{\omega/12}$, we have $\nabla f(v)=2v$ and $\text{Hess}f(v)=2I_d$, which yields
\begin{align*}
\mathbb{E}\left[f(V_0)(f(V_t)-f(V_0))\right]=&\frac{1}{N}\mathbb{E}\left[f(V_0)\int_0^t 2\nabla f(V_s)\cdot \Lambda(V_s)+\text{Hess}f(V_s):D(V_s)ds \right]\\
&+O\left(\left(\|f\|_\infty+\|\nabla f\|_{W^{2,\infty}}\right)^2\left(\frac{t}{N^{1+\omega}}+\frac{t\wedge N^{\gamma_r}}{N}\right)\right)\\
=&\frac{1}{N}\mathbb{E}\left[|V_0|^2\int_0^t \left(4V_s\cdot \Lambda(V_s)+2\text{Tr}(D(V_s))\right)\mathds{1}_{|V_0|, |V_s|\leq N^{\omega/12}}ds \right]\\
&+O\left(\frac{N^{\frac{\omega}{3}}}{N}\int_0^t\mathbb{E}\left[\mathds{1}_{|V_0|\geq N^{\omega/12}}+\mathds{1}_{|V_s|\geq N^{\omega/12}}\right]ds\right)+O\left(\frac{t}{N^{1+\frac{2\omega}{3}}}+\frac{t\wedge N^{\gamma_r}}{N^{1-\frac{\omega}{3}}}\right).
\end{align*}
Again, from \eqref{eq:lem_inc_4_erreur_2}, we may get rid of the indicator function in the time integral, and we also have 
\begin{align*}
\frac{N^{\frac{\omega}{3}}}{N}\int_0^t\mathbb{E}\left[\mathds{1}_{|V_0|\geq N^{\omega/12}}+\mathds{1}_{|V_s|\geq N^{\omega/12}}\right]ds=O\left(\frac{t}{N^{1+\frac{2\omega}{3}}}\right).
\end{align*}
Finally we get
\begin{align*}
\mathbb{E}\left[|V_0|^2(|V_t|^2-|V_0|^2)\right]=\frac{1}{N}\mathbb{E}\left[|V_0|^2\int_0^t \left(4V_s\cdot \Lambda(V_s)+2\text{Tr}(D(V_s))\right)ds \right]+O\left(\frac{t}{N^{1+\frac{2\omega}{3}}}+\frac{t\wedge N^{\gamma_r}}{N^{1-\frac{\omega}{3}}}\right).
\end{align*}
Reproducing the same argument (i.e. considering smooth compactly supported approximation of the various terms involved, and applying Proposition~\ref{prop:martingale_formulation}), for $I_2,\dots, I_5$ and combining the approximations, we conclude the proof. The case $f(v)=|v|^2$ would be dealt with similarly, using
\begin{align*}
\mathbb{E}\left[\left|V_t-V_0\right|^2\right]=\mathbb{E}\left[\left|V_t\right|^2+\left|V_0\right|^2-2V_t\cdot V_0\right]=-2\mathbb{E}\left[V_0\cdot \left(V_t-V_0\right)\right].
\end{align*}
\end{proof}

%
%
%
%

\subsection{Precise Grönwall's lemmas}

We present here several Grönwall-like estimates for second order systems. The first one is a quite usual estimate, in which we use a Grönwall-like argument on the norm of the solution.


\begin{lemma}\label{lem:Grönwall_precis_alpha}
Let $x:\mathbb{R}^+\mapsto \mathbb{R}^d$. Assume that there are two continuous functions $\textbf{a}:\mathbb{R}^+\mapsto \mathbb{R}^{d\times d}$ and $\textbf{b}:\mathbb{R}^+\mapsto \mathbb{R}^{d}$ satisfying
\begin{align*}
\exists C_A,a>0,\quad \forall t\geq0,\qquad \|\textbf{a}(t)\|\leq \frac{C_A}{N^a},
\end{align*}
 such that 
 \begin{align*}
x''(t)=\textbf{a}(t)x(t)+\textbf{b}(t)\quad\text{ for }\quad t\geq0.
\end{align*}
There exists $c>0$ depending only on $C_A$ such that for any $N\in\mathbb{N}$ and any $0\leq t\leq N^{\frac{a}{2}}$ we have
\begin{equation}
\label{eq:gron_formule_simple_cas_simple}
|x(t)|+t|x'(t)|\leq c t^2\sup_{s\in[0,t]}|\textbf{b}(s)|+c\left(|x(0)|^2+t^2|x'(0)|^2\right)^{1/2}.
\end{equation}
\end{lemma}


\begin{remark}[Sharpness of the bound from Lemma~\ref{lem:Grönwall_precis_alpha}]
Consider the specific case of 
\begin{align*}
x''(t)=\frac{1}{N^a}x(t)+\frac{1}{N^b},
\end{align*}
with $x(0)=x'(0)=0$. We may explicitly compute the solution and get for $t\geq0$
\begin{align*}
x(t)=\frac{N^a}{N^b}\left(\frac{e^{\frac{t}{N^{a/2}}}+e^{-\frac{t}{N^{a/2}}}}{2}-1\right).
\end{align*}
In particular, for $t\ll N^{\frac{a}{2}}$, we obtain
 \begin{align*}
x(t)\simeq \frac{N^a}{N^b}\frac{t^2}{N^{a}}=\frac{t^2}{N^b},
\end{align*}
which is exactly the control we have from Lemma~\ref{lem:Grönwall_precis_alpha}. Likewise if $t$ is of order $N^{\frac{a}{2}}$, we obtain in both cases $x(t)\simeq \frac{N^a}{N^b}$. Finally, notice that one cannot expect a better control for $t\gg N^{\frac{a}{2}}$ because of the exponential term.
\end{remark}

The second lemma is a bit more challenging. We no longer want to use a pointwise control on the drift term (i.e. on $\|a(t)\|$ for all $t$), but rather an averaged control  (i.e.  on $\left\|\int_0^ta(s)ds\right\|$). This seems standard when considering Grönwall's inequalities. In dimension one, for a system of order one and dismissing the term $b$, this would (roughly) amount to considering $e^{-\int_0^t a(s)ds}x(t)$, derivating to observe a nonpositive term, and finally obtaining a control of $|x(t)|$ by $e^{\left|\int_0^t a(s)ds\right|}x(0)$. In the case of non constant matrix $a$, this is however no longer possible because $a(t)$ does not a priori commute with $\int_0^ta(s)ds$, which creates several issue. We thus rely on a explicit formulation of a solution via a Peano series.


\begin{lemma}\label{lem:gronw_opt_racine}
Let $T>0$ and consider $x:[0,T]\mapsto \mathbb{R}^d$. Assume that there are two continuous functions $\textbf{a}:[0,T]\mapsto \mathbb{R}^{d\times d}$ and $\textbf{b}:[0,T]\mapsto \mathbb{R}^{d}$ satisfying 
\begin{equation}
\exists C_A,a,\xi\geq0,\quad \forall t,s\in [0,T] \text{ s.t. } 0\leq t-s\leq N^\xi,\quad  \left\|\int_s^{t}\textbf{a}(u)du\right\|\leq C_A\frac{\sqrt{t-s}}{N^a}
\end{equation}
such that
\begin{align*}
x''(t)=\textbf{a}(t)x(t)+\textbf{b}(t)\quad\text{ for }\quad t\in [0,T].
\end{align*}
Assume also that $\xi\geq \frac{2a}{3}$. There exists $c$ an explicit constant depending only on $C_A$ (and independent of $N$ and $t$), such that for all $\epsilon\in[0, \frac{2a}{3}]$ and all $t\in [0,N^\epsilon\wedge T]$, we have the following estimate
\begin{align*}
\left(|x(t)|^2+N^{2\epsilon}|x'(t)|^2\right)^{\frac{1}{2}}\leq& c\left(|x(0)|^2+N^{2\epsilon}|x'(0)|^2\right)^{\frac{1}{2}}+cN^{2\epsilon}\sup_{s\in[0,t]}|\textbf{b}(s)|.
\end{align*}
In particular, for $x(0)=x'(0)=0$, we obtain for any $t\leq  N^{\frac{2a}{3}}\wedge T$
\begin{equation}\label{eq:gron_formule_simple_cas_complique}
|x(t)|\leq ct^2\sup_{s\in[0,t]}|\textbf{b}(s)|\quad\text{ and }\quad|x'(t)|\leq ct\sup_{s\in[0,t]}|\textbf{b}(s)|.
\end{equation}
\end{lemma}


\begin{proof}[Proof of Lemma~\ref{lem:Grönwall_precis_alpha}]
Let us start by proving that for all $\epsilon>0$, we have the following estimate for all $t\geq0$
\begin{equation}
|x(t)|^2+N^{2\epsilon}|x'(t)|^2\leq \frac{N^{4\epsilon}\sup_{[0,t]}|\textbf{b}(s)|^2}{\left(2+\frac{C_A}{N^{a-2\epsilon}}\right)}\left(e^{\left(2+\frac{C_A}{N^{a-2\epsilon}}\right)\frac{t}{N^\epsilon}}-1\right)+\left(|x(0)|^2+N^{2\epsilon}|x'(0)|^2\right)e^{\left(2+\frac{C_A}{N^{a-2\epsilon}}\right)\frac{t}{N^\epsilon}}.\label{eq:gron_formule_generale_cas_simple}
\end{equation}
Compute, for $0\leq s \leq t$
\begin{align*}
\frac{d}{ds}\left(|x(s)|^2+N^{2\epsilon}|x'(s)|^2\right)=&2x(s)\cdot x'(s)+2N^{2\epsilon}x'(s)\cdot x''(s)\\
=&2x(s)\cdot x'(s)+2N^{2\epsilon}\textbf{a}(s)x(s)\cdot x'(s)+2N^{2\epsilon}x'(s)\cdot \textbf{b}(s)\\
\leq& \frac{1}{N^\epsilon}\left(1+\frac{N^{2\epsilon}C_A}{N^a}\right)|x(s)|\left(N^{\epsilon}|x'(s)|\right)+2\left(N^{\epsilon}|x'(s)|\right)\left(N^\epsilon \sup_{u\in[0,t]}|\textbf{b}(u)|\right)\\
\leq&\frac{1}{N^\epsilon}\left(2+\frac{N^{2\epsilon}C_A}{N^a}\right)(|x
(s)|^2+N^{2\epsilon}|x'(s)|^2)+N^{3\epsilon}\sup_{u\in[0,t]}|\textbf{b}(u)|^2.
\end{align*}
Therefore, Grönwall's lemma yields \eqref{eq:gron_formule_generale_cas_simple}.
Considering  $\epsilon=\frac{\ln(t)}{\ln(N)}$ for $t\leq N^{a/2}$ then yields
\begin{align*}
|x(t)|^2+t^2|x'(t)|^2\leq \sup_{[0,t]}|\textbf{b}(s)|^2\frac{t^4}{2+C_A}\left(e^{\left(2+C_A\right)}-1\right)+\left(|x(0)|^2+t^2|x'(0)|^2\right)e^{(2+C_A)}
\end{align*}
which in turn yields \eqref{eq:gron_formule_simple_cas_simple}.
\end{proof}


\begin{proof}[Proof of Lemma~\ref{lem:gronw_opt_racine}] 
Let us do the proof assuming that $x(0)=x'(0)=0$, the calculations being otherwise similar. Let $\epsilon\leq \frac{2a}{3}$ and consider $t\leq N^\epsilon$. We denote $C_B=\sup_{s\in[0,t]}|\textbf{b}(s)|$.
Writing $Y(t)=\left(x(t), N^{\epsilon}x'(t)\right)^T$, we have
\begin{equation}\label{eq:first_order_ode}
Y'(t)=A(t)Y(t)+B(t), \quad A(t)=\left(\begin{array}{cc}0 &  N^{-\epsilon}I_d\\ N^{\epsilon}\textbf{a}( t)&0\end{array}\right),\quad B(t)=\left(\begin{array}{c}0\\  N^{\epsilon}\textbf{b}(t)\end{array}\right).
\end{equation}
The explicit solution to \eqref{eq:first_order_ode} is given by the Peano-Baker series\footnote{Note that this series is well-defined, since $\|I_k(t,s)\|\leq \frac{(t-s)^k}{k!}\|A\|^k_\infty$ ($a$ is assume to be continuous, therefore $\sup_{[0,T]}\|a\|<\infty$).}
\begin{align*}
Y(t)=\Phi(t,0)Y(0)+\int_{0}^t\Phi(t,s)B(s)ds,
\end{align*}
where, for $N^{\epsilon}\geq t\geq s\geq0$
\begin{align*}
\Phi(t,s)=\sum_{k=0}^\infty I_k(t,s),\quad  I_0(t,s)=I_d,\quad I_k(t,s)=\int_s^t\int_s^{t_k}...\int_s^{t_2}A(t_k)...A(t_2)A(t_1)dt_1....dt_k.
\end{align*}
See for instance \cite{BS11}. One can directly compute
\begin{align*}
|I_0(t,s)B(s)|\leq C_B,\qquad \left|I_1(t,s)B(s)\right|=\left|(t-s)b(s)\right|\leq C_B(t-s).
\end{align*}
Given the form of the matrix $A$, notice that for $k\geq1$
\begin{align*}
I_{2k}(t,s)
=&\int_{s}^tdt_{2k}\int_{s}^{t_{2k}}dt_{2k-1}....\int_{s}^{t_2}dt_{1}\left(\begin{array}{cc} \textbf{a}(t_{2k-1})...\textbf{a}(t_1) &0\\0&\textbf{a}(t_{2k})\textbf{a}(t_{2k-2})...\textbf{a}(t_2) \end{array}\right),\\
I_{2k+1}(t,s)
=&\int_{s}^t\int_{s}^{t_{2k+1}}....\int_{s}^{t_2} \left(\begin{array}{cc} 0& N^{-\epsilon}\textbf{a}(t_{2k})...\textbf{a}(t_2) \\ N^{\epsilon }\textbf{a}(t_{2k+1})\textbf{a}(t_{2k-1})...\textbf{a}(t_1) &0\end{array}\right).
\end{align*}
i.e. the odd and even integrand have been separated.  Note that, given the form of the vector $B(s)$, only the right column of $I_k$ matters. We therefore calculate
\begin{align*}
&\int_{s}^t\int_{s}^{t_{2k}}....\int_{s}^{t_2}\textbf{a}(t_{2k})\textbf{a}(t_{2k-2})...\textbf{a}(t_2)\\
=& \int_{s}^t dt_{2k-1}\int_{s}^{t_{2k-1}}dt_{2k-3}...\int_s^{t_3}dt_1\left(\int_{t_{2k-1}}^tdt_{2k}\textbf{a}(t_{2k})\right)...\left(\int_{t_{1}}^{t_3}dt_{2}\textbf{a}(t_{2})\right),
\end{align*}
 i.e. we re-organise into the even integrals then the odd integrals. This way, since $0 \leq t-s\leq N^\epsilon \leq N^{\xi}$, we may bound
 \begin{align*}
&\left| I_{2k}(t,s) B(s)\right|\\
&\qquad\leq \int_{s}^t dt_{2k-1}\int_{s}^{t_{2k-1}}dt_{2k-3}...\int_s^{t_3}dt_1\left\|\int_{t_{2k-1}}^{t}dt_{2k}\textbf{a}(t_{2k})\right\|...\left\|\int_{t_{1}}^{t_3}dt_{2}\textbf{a}(t_{2})\right\|\left|B(s)\right|\\
&\qquad\leq C_B N^{\epsilon}C_A^kN^{-a k}\int_{s}^t dt_{2k-1}\int_{s}^{t_{2k-1}}dt_{2k-3}...\int_s^{t_3}dt_1\sqrt{t-t_{2k-1}}...\sqrt{t_3-t_{1}}\\
&\qquad= \frac{2}{3}C_B N^{\epsilon}C_A^kN^{-a k}\int_{s}^t dt_{2k-1}\int_{s}^{t_{2k-1}}dt_{2k-3}...\int_s^{t_3}dt_1\sqrt{t-t_{2k-1}}...\sqrt{t_5-t_3}(t_3-s)^{3/2}.
\end{align*}
Let $\Gamma$ be the gamma function. We now use the fact that, for $\omega>1$
\begin{align*}
\int_{s}^t \sqrt{t-u}\left(u-s\right)^\omega du=\frac{\Gamma\left(\frac{3}{2}\right)\Gamma\left(\omega+1\right)}{\Gamma\left(\omega+\frac{3}{2}+1\right)}(t-s)^{\omega+\frac{3}{2}},
\end{align*}
and, by computing recursively the multiple integrals, we
obtain 
\begin{align*}
\left| I_{2k}(t,s) B(s)\right|\leq&
\frac{2}{3}C_B N^{\epsilon}C_A^kN^{-a k}\Gamma\left(\frac{3}{2}\right)^{k-1}\frac{\Gamma\left(\frac{5}{2}\right)}{\Gamma\left(\frac{3}{2}k+1\right)}(t-s)^{\frac{3}{2}k}.
\end{align*}
Using $\Gamma\left(3/2\right)=\frac{\sqrt{\pi}}{2}$ and $\Gamma(5/2)=\frac{3\sqrt{\pi}}{4}$, we get
\begin{equation}\label{eq:I_t_2_k}
\left| I_{2k}(t,s) B(s)\right| \leq C_B N^{\epsilon}\frac{\left[C_A^{\frac{2}{3}}\left(\frac{\sqrt{\pi}}{2}\right)^{\frac{2}{3}}N^{-\frac{2a}{3}}(t-s)\right]^{\frac{3}{2}k}}{\Gamma\left(\frac{3}{2}k+1\right)}  .
\end{equation}
Noticing that
\begin{align*}
\left|I_{2k+1}(t,s) B(s)\right|\leq N^{-\epsilon}\int_s^t \left|I_{2k}(t_{2k+1},s) B(s)\right|dt_{2k+1},
\end{align*}
we then directly have
\begin{equation}\label{eq:I_t_2_k+1}
\left| I_{2k+1}(t,s) B(s)\right|\leq C_B\left(t-s\right)\frac{\left[\left(C_A\frac{\sqrt{\pi}}{2}\right)^{\frac{2}{3}}N^{-\frac{2}{3}a}(t-s)\right]^{\frac{3}{2}k}}{\Gamma\left(\frac{3}{2}k+2\right)}.
\end{equation}
Let us, for the sake of conciseness and only in the following calculations, denote ${A=\left(C_A\frac{\sqrt{\pi}}{2}\right)^{\frac{2}{3}}N^{-\frac{2a}{3}}(t-s)}$. From \eqref{eq:I_t_2_k}-\eqref{eq:I_t_2_k+1}, we  have
\begin{align*}
\left|\Phi(t,s)B(s)\right|\leq& C_B N^{\epsilon}\left(\sum_{k=0}^\infty \frac{A^{\frac{3}{2}k}}{\Gamma\left(\frac{3}{2}k+1\right)}+\frac{(t-s)}{N^{\epsilon}}\sum_{k=0}^\infty\frac{A^{\frac{3}{2}k}}{\Gamma\left(\frac{3}{2}k+2\right)} \right)\\
\leq&C_B N^{\epsilon}\left(\sum_{p=0}^\infty \frac{A^{3p}}{3p!}+A^{\frac{1}{2}}\sum_{p=0}^\infty \frac{A^{3p+1}}{\left(3p+1\right)!}\right.
\left.+\frac{(t-s)}{N^{\epsilon}}\sum_{p=0}^\infty\frac{A^{3p}}{(3p)!} +\frac{(t-s)}{N^{2\epsilon}}A^{\frac{1}{2}}\sum_{p=0}^\infty\frac{A^{3p+1}}{(3p+1)!} \right)\\
\leq&C_B N^{\epsilon}(1+A^{1/2})\left(1+\frac{(t-s)}{N^{\epsilon}}\right)e^A,
\end{align*}
(where we use $\Gamma\left(3p+\frac{5}{2}\right)>(3p+1)!$, $\Gamma(3p+1)=3p!$, $\Gamma(3p+2)=(3p+1)!>(3p)!$ and $\Gamma\left(3p+\frac{7}{2}\right)>(3p+1)!$). This way, using $\epsilon\leq \frac{2a}{3}$ and $t-s\leq N^{\epsilon}$, we obtain $A\leq \left(C_A\frac{\sqrt{\pi}}{2}\right)^{\frac{2}{3}}$ and thus for any $t,s\in[0,N^\epsilon]$
\begin{align*}
\left|\Phi(t,s)B(s)\right|\leq&cC_B N^{\epsilon},
\end{align*}
for some universal constant $c>0$ depending only on $C_A$.

For $t\leq N^\epsilon$ and for $Y(0)=0$, we get the upper bound
\begin{align*}
\left|Y(t)\right|\leq& \int_{0}^t|\Phi(t,s)B(s)|ds
\leq cC_BN^{2\epsilon}.
\end{align*}
For  $Y(0)\neq 0$, we similarly obtain that there exists some $c>0$ such that for all $t\leq N^{\epsilon}$,
\begin{align*}
\left|Y(t)\right|\leq&c\left(|x(0)|^2+N^{2\epsilon}|x'(0)|^2\right)^{\frac{1}{2}}+cC_BN^{2\epsilon}.
\end{align*}
This yields the first result. Assume now $t\leq N^{\frac{2\alpha}{3}}$ and define $\epsilon=\frac{\ln(t)}{\ln(N)}\leq\frac{2\alpha}{3}$ (i.e. $t=N^{\epsilon}$). We obtain, for $x(0)=x'(0)=0$
\begin{align*}
|x(t)|\leq& c \, C_Bt^2.
\end{align*}
Likewise to control $x'(t)$.
\end{proof}


\paragraph{Acknowledgments.} 
This work was carried out while P.L.B. was a postdoc at I.H.E.S. under the Huawei Young Talents Program.
We thank Frank Merle for useful discussions.

\bibliographystyle{plain}
\bibliography{bibliographie}

\begin{thebibliography}{10}

\bibitem{Ayi}
Nathalie Ayi.
\newblock From {Newton}'s law to the linear {Boltzmann} equation without
  cut-off.
\newblock {\em Commun. Math. Phys.}, 350(3):1219--1274, 2017.

\bibitem{BS11}
Michael Baake and Ulrike Schlaegel.
\newblock The {P}eano-{B}aker series.
\newblock {\em Proceedings of the Steklov Institute of Mathematics},
  275(1):155--159, 2011.

\bibitem{basile2015derivation}
Giada Basile, Alessia Nota, Federica Pezzotti, and Mario Pulvirenti.
\newblock Derivation of the {Fick}’s law for the {Lorentz} model in a low
  density regime.
\newblock {\em Communications in Mathematical Physics}, 336(3):1607--1636,
  2015.

\bibitem{BGS16}
Thierry Bodineau, Isabelle Gallagher, and Laure Saint-Raymond.
\newblock The {B}rownian motion as the limit of a deterministic system of
  hard-spheres.
\newblock {\em Invent. Math.}, 203(2):493--553, 2016.

\bibitem{BGS18}
Thierry Bodineau, Isabelle Gallagher, and Laure Saint-Raymond.
\newblock Derivation of an {O}rnstein-{U}hlenbeck process for a massive
  particle in a rarified gas of particles.
\newblock {\em Ann. Henri Poincar\'e}, 19(6):1647--1709, 2018.

\bibitem{DR01}
Laurent Desvillettes and Valeria Ricci.
\newblock A rigorous derivation of a linear kinetic equation of
  {F}okker-{P}lanck type in the limit of grazing collisions.
\newblock {\em J. Statist. Phys.}, 104(5-6):1173--1189, 2001.

\bibitem{DLLS13}
Matthew Dobson, Fr\'ed\'eric Legoll, Tony Leli\`evre, and Gabriel Stoltz.
\newblock Derivation of {L}angevin dynamics in a nonzero background flow field.
\newblock {\em ESAIM Math. Model. Numer. Anal.}, 47(6):1583--1626, 2013.

\bibitem{Due21}
Mitia Duerinckx.
\newblock On the size of chaos via {G}lauber calculus in the classical
  mean-field dynamics.
\newblock {\em Comm. Math. Phys.}, 382(1):613--653, 2021.

\bibitem{DLB25}
Mitia {Duerinckx} and Corentin {Le Bihan}.
\newblock {Lenard-Balescu thermalization: rigorous derivation from a toy
  model}.
\newblock {\em arXiv e-prints}, page arXiv:2511.10778, November 2025.

\bibitem{DS21}
Mitia Duerinckx and Laure Saint-Raymond.
\newblock Lenard-{B}alescu correction to mean-field theory.
\newblock {\em Probab. Math. Phys.}, 2(1):27--69, 2021.

\bibitem{DW23}
Mitia Duerinckx and Raphael Winter.
\newblock Well-posedness of the {L}enard-{B}alescu equation with smooth
  interactions.
\newblock {\em Arch. Ration. Mech. Anal.}, 247(4):Paper No. 71, 52, 2023.

\bibitem{Dup06}
Bertrand Duplantier.
\newblock Brownian motion, "{Diverse} and undulating".
\newblock In {\em Einstein, 1905--2005}, volume~47 of {\em Prog. Math. Phys.},
  pages 201--293. Birkh\"auser, Basel, 2006.
\newblock Translated from the French by Emily Parks.

\bibitem{DGL81}
Detlef D\"urr, Sheldon Goldstein, and Joel~L. Lebowitz.
\newblock A mechanical model of {Brownian} motion.
\newblock {\em Communications in Mathematical Physics}, 78:507--530, 1981.

\bibitem{DGL87}
Detlef D\"urr, Sheldon Goldstein, and Joel~L. Lebowitz.
\newblock Asymptotic motion of a classical particle in a random potential in
  two dimensions: {L}andau model.
\newblock {\em Comm. Math. Phys.}, 113(2):209--230, 1987.

\bibitem{EK86}
Stewart~N. Ethier and Thomas~G. Kurtz.
\newblock {\em Markov processes}.
\newblock Wiley Series in Probability and Mathematical Statistics: Probability
  and Mathematical Statistics. John Wiley \& Sons, Inc., New York, 1986.
\newblock Characterization and convergence.

\bibitem{fougeres2024derivation}
Florent Foug{\`e}res.
\newblock On the derivation of the linear {Boltzmann} equation from the
  nonideal {Rayleigh} gas: Adaptive pruning and improvement of the convergence
  rate.
\newblock {\em Journal of Statistical Physics}, 191(10):136, 2024.

\bibitem{Kal02}
Olav Kallenberg.
\newblock {\em Foundations of modern probability}.
\newblock Probability and its Applications (New York). Springer-Verlag, New
  York, second edition, 2002.

\bibitem{KP79}
Harry Kesten and George~C. Papanicolaou.
\newblock A limit theorem for turbulent diffusion.
\newblock {\em Comm. Math. Phys.}, 65(2):97--128, 1979.

\bibitem{KP80}
Harry Kesten and George~C. Papanicolaou.
\newblock A limit theorem for stochastic acceleration.
\newblock {\em Comm. Math. Phys.}, 78(1):19--63, 1980/81.

\bibitem{KL10}
Shigeo Kusuoka and Song Liang.
\newblock A classical mechanical model of {B}rownian motion with plural
  particles.
\newblock {\em Rev. Math. Phys.}, 22(7):733--838, 2010.

\bibitem{LeBihan24}
Corentin {Le Bihan}.
\newblock {Long time validity of the linearized Boltzmann uncut-off and the
  linearized Landau equations from the Newton Law}.
\newblock {\em arXiv e-prints}, page arXiv:2408.03597, August 2024.

\bibitem{LeBihan25}
Corentin {Le Bihan}.
\newblock {Around the quantum Lenard-Balescu equation}.
\newblock {\em arXiv e-prints}, page arXiv:2501.06544, January 2025.

\bibitem{LT20}
Christopher Lutsko and B{\'a}lint T{\'o}th.
\newblock Invariance principle for the random {Lorentz} gas—beyond the
  {Boltzmann-Grad} limit.
\newblock {\em Communications in Mathematical Physics}, 379(2):589--632, 2020.

\bibitem{MS24}
Karsten {Matthies} and Theodora {Syntaka}.
\newblock {Fractional diffusion as the limit of a short range potential
  Rayleigh gas}.
\newblock {\em arXiv e-prints}, page arXiv:2405.19025, May 2024.

\bibitem{NSV18}
Alessia Nota, Sergio Simonella, and Juan J.~L. Vel\'azquez.
\newblock On the theory of {L}orentz gases with long range interactions.
\newblock {\em Rev. Math. Phys.}, 30(3):1850007, 62, 2018.

\bibitem{NVW21}
Alessia Nota, Juan J.~L. Vel\'azquez, and Raphael Winter.
\newblock Interacting particle systems with long-range interactions: scaling
  limits and kinetic equations.
\newblock {\em Atti Accad. Naz. Lincei Rend. Lincei Mat. Appl.},
  32(2):335--377, 2021.

\bibitem{NVW22}
Alessia Nota, Juan J.~L. Vel\'azquez, and Raphael Winter.
\newblock Interacting particle systems with long-range interactions:
  approximation by tagged particles in random fields.
\newblock {\em Atti Accad. Naz. Lincei Rend. Lincei Mat. Appl.},
  33(2):439--506, 2022.

\bibitem{NWL19}
Alessia Nota, Raphael Winter, and Bertrand Lods.
\newblock Kinetic description of a {R}ayleigh gas with annihilation.
\newblock {\em J. Stat. Phys.}, 176(6):1434--1462, 2019.

\bibitem{Phi87}
John Phillips.
\newblock On the uniform continuity of operator functions and generalized
  powers-stormer inequalities.
\newblock {\em Technical Reports (Mathematics and Statistics)}, 1987.

\bibitem{PV03}
Fr\'ed\'eric Poupaud and Alexis Vasseur.
\newblock Classical and quantum transport in random media.
\newblock {\em J. Math. Pures Appl. (9)}, 82(6):711--748, 2003.

\bibitem{Spohn}
Herbert Spohn.
\newblock {\em Large scale dynamics of interacting particles}.
\newblock Springer Science, 2012.

\bibitem{VW18}
Juan J.~L. Vel\'azquez and Raphael Winter.
\newblock From a non-{M}arkovian system to the {L}andau equation.
\newblock {\em Comm. Math. Phys.}, 361(1):239--287, 2018.

\bibitem{WY71}
Shinzo Watanabe and Toshio Yamada.
\newblock On the uniqueness of solutions of stochastic differential equations.
  {II}.
\newblock {\em J. Math. Kyoto Univ.}, 11:553--563, 1971.

\end{thebibliography}

\end{document}